\documentclass[12pt,leqno]{amsart}
\usepackage{amsmath,stmaryrd}
\usepackage{amssymb}
\usepackage{amsthm} % replace :  \usepackage{theorem}
\usepackage{epsf}
\usepackage{graphicx}
\usepackage{esint} % added allow $\fint$
\usepackage{dsfont}
\usepackage[pagebackref]{hyperref} 
\hypersetup{pdfpagemode=FullScreen,  colorlinks=true} 
\usepackage{enumerate} 

\numberwithin{equation}{section}

\newcommand{\nn}{\nonumber}
\newcommand{\ms}{\medskip}

\newcommand{\R}{\mathbb{R}}
\renewcommand{\H}{\mathcal H}

\newcommand{\bN}{\mathbb{N}}

\renewcommand{\O}{{\mathcal  O}}
\renewcommand{\d}{\partial}
\newcommand{\dist}{\,\mathrm{dist}}

\newcommand{\sm}{\setminus}
\newcommand{\supp}{\mathrm{supp}}
\newcommand{\diam}{\mathrm{diam}}

\newcommand{\wt}{\widetilde}

\newcommand{\cF}{{\mathcal  F}}

\newcommand{\1}{{\mathds 1}}

\newcommand{\dr}{\partial}

\DeclareMathOperator{\Jac}{Jac}

\DeclareMathOperator{\diver}{div}
\DeclareMathOperator{\sgn}{sgn}

\newcommand{\Ak}{\mathfrak A}

\newcommand{\Rt}{\R^{n-d} \setminus \{0\}}

\newcommand{\Ab}{\mathbb A}
\newcommand{\bb}{\mathfrak b}

\theoremstyle{plain}
\newtheorem{theorem}[equation]{Theorem}
\newtheorem{lemma}[equation]{Lemma}
\newtheorem{corollary}[equation]{Corollary}
\newtheorem{proposition}[equation]{Proposition}

\newtheorem{definition}[equation]{Definition}

\theoremstyle{definition}

\theoremstyle{remark}
\newtheorem{remark}[equation]{Remark}

\newcommand{\re}{\mathbb{R}}
\newcommand{\rn}{{\mathbb{R}^n}}
\newcommand{\rd}{{\mathbb{R}^d}}

\newcommand{\dd}{\mathbb{D}}

\newcommand{\RR}{{\mathbb{R}}}
\newcommand{\NN}{{\mathbb{N}}}

\newcommand{\eps}{\varepsilon}

\newcommand{\vp}{\varphi}

\newcommand{\A}{\mathcal{A}}
\newcommand{\C}{\mathcal{C}}

\newcommand{\po}{{\partial\Omega}}

\newcommand{\D}{\Delta}

\newcommand{\bp}{\noindent {\it Proof}.\,\,}
\newcommand{\ep}{\hfill$\Box$ \vskip 0.08in}
\newcommand{\dint}{\int\!\!\!\int}

\def\div{\mathop{\operatorname{div}}}

\begin{document}

\title{Dahlberg's theorem in higher co-dimension}

\author[David]{Guy David}
\address{Guy David. Univ Paris-Sud, Laboratoire de Math\'ematiques, UMR8628, Orsay, F-91405}
\email{guy.david@math.u-psud.fr}

\author[Feneuil]{Joseph Feneuil}
\address{Joseph Feneuil. School of Mathematics, University of Minnesota, Minneapolis, MN 55455, USA}
\email{jfeneuil@umn.edu}

\author[Mayboroda]{Svitlana Mayboroda}
\address{Svitlana Mayboroda. School of Mathematics, University of Minnesota, Minneapolis, MN 55455, USA}
\email{svitlana@math.umn.edu}
\thanks{Mayboroda is supported in part by the Alfred P. Sloan Fellowship, the NSF INSPIRE Award DMS 1344235, NSF CAREER Award DMS 1220089. David is supported in part by the ANR, programme blanc GEOMETRYA ANR-12-BS01-0014.
Feneuil is partially supported by the ANR project ``HAB'' no. ANR-12-BS01-0013.
This work was supported by a public grant as part of the Investissement d'avenir project, 
reference ANR-11-LABX-0056-LMH, LabEx LMH. 
Part of this work was completed during Mayboroda's visit to Universit\'e Paris-Sud, Laboratoire de Math\'ematiques, Orsay, and Ecole Polytechnique, PMC. We thank the corresponding Departments, together with the Fondation Jacques Hadamard, the \'Ecole des Mines, and the Mathematical Sciences Research Institute (NSF grant DMS 1440140) for support and hospitality.
}

\maketitle

\begin{abstract} In 1977 the celebrated theorem of B.\, Dahlberg established that the harmonic measure is absolutely continuous with respect to the Hausdorff measure on a Lipschitz graph of dimension $n-1$ in $\RR^n$, 
and later this result has been
extended to more general non-tangentially accessible domains and beyond. 

In the present paper we prove the first analogue of Dahlberg's theorem in higher co-dimension, on a Lipschitz graph $\Gamma$ of dimension $d$ in $\RR^n$, $d<n-1$, with a small Lipschitz constant. We construct a linear degenerate elliptic operator $L$ such that the corresponding harmonic measure $\omega_L$ is absolutely continuous with respect to the Hausdorff measure on $\Gamma$. More generally, we provide sufficient conditions on the matrix of coefficients of $L$ which guarantee the mutual absolute continuity of $\omega_L$ and the Hausdorff measure.

\ms

\noindent
{\sc R\'esum\'e.} 
Dans son c\'el\`ebre th\'eor\`eme de 1977, B. Dahlberg a prouv\'e que 
pour les domaines de $\R^n$ born\'es par un graphe lipschitzien de dimension $n-1$,
la mesure harmonique est absolument continue par rapport \`a la mesure de surface,
r\'esultat qui a ensuite \'et\'e \'etendu aux domaines avec acc\`es non-tangentiel, 
et au del\`a.

Dans ce papier on d\'emontre le premier analogue de ce th\'eor\`eme
pour le compl\'ementaire d'un graphe lipschitzien $\Gamma$ de dimension $d < n-1$ 
avec une petite constante de Lipschitz.
On construit un op\'erateur lin\'eaire elliptique d\'eg\'en\'er\'e $L$ dont la mesure 
harmonique associ\'ee $\omega_L$ est absolument continue par rapport \`a la 
mesure de Hausdorff $\H^d$ sur $\Gamma$. Plus g\'en\'eralement, on donne des conditions
suffisantes sur la matrice des coefficients de $L$ pour que $\omega_L$ et 
$\H^d_{\vert \Gamma}$ soient mutuellement absolument continues.

\end{abstract}

\ms\noindent{\bf Key words/Mots cl\'es.}
 boundary with co-dimension higher than 1, degenerate elliptic operators, Dahlberg's theorem, harmonic measure in higher codimension.

\ms\noindent

\tableofcontents

\section{Introduction}
\label{S1}

\subsection{History and motivation}

Dimension and structure of harmonic measure have attracted a lot of attention in the past 50 years, with a splash of remarkable new developments on uniformly rectifiable sets most recently.

Following the first results of Carleson \cite{C}, in 1985 the fundamental theorem of Makarov \cite{Mak1}, \cite{Mak2} established that harmonic measure on any continuum on the plane has dimension exactly  $1$. More generally,  for a domain $\Omega$ on the Riemann sphere whose complement has positive logarithmic capacity, the harmonic measure in $\Omega$ is supported in a subset of $\po$ whose 
Hausdorff dimension is at most 1, due to the result of Jones and Wolff \cite{JW}. 
In particular, if $d \in (1, 2)$, $0 < \H^d(E) < \infty$, 
then $\omega$ is always singular with respect to $\H^d|_{E}$.
In space, when the ambient dimension $n$ is greater than or equal to $3$, the situation is more complicated and less understood: on the one hand, Bourgain \cite{Bo} proved that the dimension of harmonic measure always drops: $\dim_\H \omega < n$. On the other hand, even for connected $E=\partial\Omega$, contrary to the planar case, $\dim_\H \omega$ can be strictly bigger than $n-1$, due to a counterexample of Wolff \cite{W} which is nowadays known as the Wolff's snowflake.

Restricting the attention to integer dimensions, one is now bound to consider $d=n-1$ and 
ask on which sets with $0 < \H^{n-1}(E) < \infty$ the harmonic measure is absolutely continuous 
with respect to the Lebesgue measure. This is a very delicate issue, where dimension, 
regularity and topology all play an intricate role, and only the past few years and some outstanding developments in harmonic analysis on uniformly rectifiable sets have brought clarity. 
In short, the emerging philosophy is that the rectifiability of the boundary is necessary for the 
absolute continuity of $\omega$ with respect to $\H^{n-1}$, and that rectifiability,
along with suitable connectedness assumptions, is sufficient. 
Omitting for now precise definitions, let us recall the main results in this regard. 
The 1916 theorem of F.\& M. Riesz has established the absolute continuity of 
the harmonic measure for a simply connected domain in the complex plane 
with a rectifiable boundary \cite{RR}. The quantifiable analogue of this result (the
$A^\infty$ property of harmonic measure) was obtained by Lavrent'ev in 1936 \cite{Lv} 
and the local version, pertaining to subsets of a rectifiable boundary, was proved by 
Bishop and Jones in 1990 \cite{BJ}. In the latter work the authors also showed that some 
connectedness is necessary for the absolute continuity of $\omega$ with respect to $\H^{1}$,
for there exists a planar set with a rectifiable boundary for which the harmonic measure is singular with respect to the Lebesgue measure. 

In higher dimensions, the first breakthrough was the celebrated theorem of Dahlberg which established the absolute continuity of $\omega$ with respect to $\H^{n-1}$ on Lipschitz graphs \cite{Da}. 
It was later extended to non-tangentially accessible, NTA, domains in \cite{DJ}, \cite{Se}. 
Roughly speaking, the non-tangential accessibility is an assumption of quantifiable connectedness, which requires the presence of the interior and exterior corkscrew points, as well as Harnack chains. 
Similarly to the lower-dimensional case, counterexamples show that some topological restrictions 
are needed for the absolute continuity of $\omega$ with respect to $\H^{n-1}$ \cite{Wu}, \cite{Z}. 
Much more recently, in \cite{HM1}, \cite{HMU}, \cite{AHMNT}, the authors proved that under 
a (weaker) 1-sided NTA assumption, the uniform rectifiability of the boundary is equivalent to the complete set of NTA conditions and hence, is equivalent to the absolute continuity of 
the harmonic measure with respect to the Lebesgue measure.  Finally, in 2015 the full converse, ``free boundary"   result was obtained and established that rectifiability is necessary for  the absolute continuity of harmonic measure 
with respect to $\H^{n-1}$ in any dimension $n\geq 2$, without any additional topological assumptions \cite{AHM3TV} (see also \cite{HLMN} on Ahlfors regular sets). 
These results and problems have generated a large
amount of activity and have been tightly intertwined with recent achievements in harmonic analysis on uniformly rectifiable sets, pertaining, in particular, to characterizations of uniform rectifiability via boundedness of singular integrals and, more specifically, Riesz transforms \cite{NToV}, \cite{MMV}. 
While avoiding for now precise definitions, we remind the reader that rectifiability basically 
concerns coverings of a set by countable unions of Lipschitz graphs  
and Dahlberg's theorem for harmonic measure on Lipschitz graphs, much as boundedness of the Cauchy integral on Lipschitz curves \cite{CMM}, underpin the entire theory.

{\it{The main goal of the present paper is to establish the first analogue of Dahlberg's theorem on sets of co-dimension higher than 1.}} As announced in \cite{DFM-CRAS}, we construct a linear degenerate elliptic operator $L$ such that the corresponding harmonic measure $\omega_L$ is absolutely continuous with respect to the Hausdorff measure on a $d$-dimensional Lipschitz graph $\Gamma$, $d<n-1$, with a small Lipschitz constant. More generally, we provide 
sufficient conditions on the matrix of coefficients of $L$ which guarantee mutual absolute continuity of $\omega_L$ and the Hausdorff measure.

Turning to details, we assume that $d$ is an integer and $\Gamma$ is the graph 
of a Lipschitz function $\varphi : \R^d \to \R^{n-d}$, with a small Lipschitz constant.
We want to find an analogue of harmonic measure, that will be defined on $\Gamma$ and
associated to a divergence form operator on $\Omega = \R^n \sm \Gamma$. We write
the operator as $L = - \rm{div} \Ak \nabla$, with $\Ak : \Omega \to \mathbb{M}_n(\R)$, 
and put forward the ellipticity condition with a different homogeneity, i.e., we require that for some 
$C_1 \geq 1$, 
\begin{eqnarray} \label{1.2.1} 
&& \dist(X,\Gamma)^{n-d-1} \Ak(X)\xi \cdot \zeta \leq C_1  |\xi| \,|\zeta|
\ \text{ for } X\in \Omega \text{ and } \xi, \zeta \in \R^n, \\[4pt]
\label{1.2.2}
&& \dist(X,\Gamma)^{n-d-1} \Ak(X)\xi \cdot \xi \geq C_1^{-1}  |\xi|^2
\ \text{ for } X\in \Omega \text{ and } \xi \in \R^n.
\end{eqnarray}
Under merely these assumptions, and even when $d < n-1$ is not an integer and $\Gamma$ 
is merely Ahlfors-regular of dimension $d$, we have recently developed a 
fairly complete
elliptic theory and in particular we can define a harmonic measure that satisfies, among 
others, the doubling property and the change of pole property (see \cite{DFMprelim}). 
In general, such a harmonic measure need not be absolutely continuous on $\Gamma$.
Already in codimension 1, we know that the ellipticity and the boundedness of $A$ are
not enough to ensure absolute continuity  (see \cite{CFK}).
In higher codimension, we don't expect that the assumptions \eqref{1.2.1}--\eqref{1.2.2} 
are sufficient either, even when $\Ak(X)$ is a multiple of the identity matrix, but  
we do not have a counter-example yet.
Due to Dahlberg's theorem, it is, however, the case for the Laplacian when $d=n-1$.
The goal of the present paper is to
find {\em one} elliptic operator $L = - \div \Ak \nabla$ for which the harmonic measure is 
absolutely continuous with respect to the Hausdorff measure when $d<n-1$. 
Ironically, it is not the one associated to the matrix of coefficients $\dist(X,\Gamma)^{n-d-1} I$ 
which perhaps would be a natural first guess, at least not in higher dimensions.

We assume that $\Ak(X) = D(X)^{-n+d+1} I$ for $x\in \Omega$, with 
\begin{equation} \label{1.3}
D(X) = \Big\{ \int_\Gamma |X-y|^{-d-\alpha} d\H^d(y) \Big\}^{-1/\alpha}
\end{equation}
for some constant $\alpha > 0$. 
It is easy to see that 
$D(X)$ is equivalent to $\dist(X,\Gamma)$
(and this would even stay true when $\Gamma$ is an Ahlfors regular set).
When $d=1$ we can also take $\Ak(X) = \dist(X,\Gamma)^{-n+d+1} I$, but
when $d \geq 2$, $\dist(X,\Gamma)$ does not appear to be smooth enough to make our proofs work,   
which is why we use \eqref{1.3} instead. 

With these assumptions we will prove that the harmonic measure described above 
is absolutely continuous with respect to $\H^d_{\vert \Gamma}$, with a density which is 
a Muckenhoupt $A^\infty$ weight. 

This seems to be the first result of this nature in higher co-dimensional sets. 
Some aspects of pre-requisite elliptic theory developed in \cite{DFMprelim} 
existed before \cite{FKS}, \cite{FJK}, and some have even been recently proved 
for non-linear operators, most notably the $p$-Laplacian, on Reifenberg flat sets in \cite{LN}. One could conjecture that the $p$-harmonic measure for a suitable $p$ is also absolutely continuous with respect to the Hausdorff measure but at the moment no result to this effect could be achieved.

To better imagine the properties of the operator $L = - \div \Ak \nabla$ above and associated advantages and challenges, one should notice that it is non-local (the solution in a given ball depends on far away features of the domain, because the distance $D$ and hence, the operator $L$ does). 
In fact, there may be connections with other non-local operators, e.g., fractional Laplacian 
in $\RR^{n-1}$, as the famous Caffarelli-Silvestre extension transforms it into a degenerate operator in $\RR^n_+$ \cite{CS}, though of course, with a different non-degeneracy and, hence, different features than ours. 
In this regard, one should also recall elliptic edge operators which recently enjoyed renewed interest in connection with the K\"ahler-Einstein edge metrics, see, e.g., \cite{JMR}. 
However,  beyond the general idea of treating higher co-dimensional features with elliptic degeneracies, it would be premature to draw connections with the present work. 

\ms
The proof of our main result has two main steps: 
\begin{enumerate}[(1)]
\item construct a nice bi-Lipschitz change of variables that sends $\Gamma$ to a $d$-plane $P_0$;
\item show that for a large class of degenerate elliptic operators on $\R^n \sm P_0$, including the one arising from the aforementioned change of variables, the square function/non-tangential maximal function estimates hold for bounded solutions and further imply the absolute continuity of 
the associated harmonic measure with respect to the surface measure on $P_0$.
\end{enumerate}

This is not a surprising strategy, and our approach owes a lot to  \cite{KKPT}, \cite{KKiPT}, \cite{DPP2015}. However, trying to execute it in higher co-dimension, one quickly learns that 
most of the methods developed in co-dimension one don't work. 
There are three new major players. First, the matrix of coefficients suitable for $(2)$ has to have a special structure in the portion responsible for the $t$ variables (this was not a problem before because such a portion is just one entry in co-dimension 1). 
Secondly, as a result, one needs the aforementioned change of variables to be as close as possible to an isometry at many scales - a property missing in previously used mappings of Lipschitz domains to 
$\RR^{n-1}$ \cite{KenigB}, and only appearing, albeit in a different form, in 
the analysis of Reifenberg flat sets \cite{Toro}, \cite{DT}. Finally, as the
distance is now a factor of the matrix of coefficients, 
it turns out that we want the composition of $D$ with our change of variable
to be close to the distance to $\R^d$
(and this property would also be lost by previously used ``vertical" changes of variables), which ultimately leads to our special choice of $D$. In some sense, a big advantage of Lipschitz graph domains of dimension $n-1$ over more complicated rectifiable sets is the existence of {\it one} special direction, the vertical one. In higher co-dimensional sets this advantage is notoriously missing, and even working formally on a Lipschitz graph, we have to use 
some of the tools of geometric measure theory on uniformly rectifiable domains, 
such as  X.\,Tolsa's $\alpha$-numbers related to the Wasserstein distance, 
to even construct the change of variables. In addition, we lose a possibility to test 
our intuition in the planar case, using conformal mapping techniques, as
for $n=2$ the only higher co-dimensional set could be a point.
A good side of having to overcome all these challenges is a resulting indication that the results can be carried over to much more general sets and, indeed, we conjecture that the absolute continuity of harmonic measure associated to $\Ak(X) = D(X)^{-n+d+1} I$ is absolutely continuous with respect to Hausdorff measure on all uniformly rectifiable domains, in 
contrast with the results for the Laplacian for $d=n-1$ which requires extra topological restrictions due to \cite{BJ}. For now though, we concentrate on Lipschitz graphs, and before passing on to a detailed description of main results, only mention that all aspects of the construction, pertaining to parts (1) or (2) above, are new even in co-dimension one.

\subsection{Main results. Harmonic measure and the $A^\infty$ condition}

We consider a Lipschitz function $\varphi : \, \R^d \to \R^{n-d}$. 
Its Lipschitz constant is denoted by $C_0$, and is defined as the quantity
\begin{equation} \label{defC0}
C_0 := \sup_{\begin{subarray}{c} x,y\in \R^d \\ x\neq y \end{subarray}} \dfrac{|\varphi(x)-\varphi(y)|}{|x-y|} = \sup_{x\in \R^d} \, \sup_{\begin{subarray}{c} h=(h_1,\dots,h_d)\in \R^d \\ |h|=1 \end{subarray}} \left| \sum_{i=1} h_i \dr_{x_i} \varphi(x) \right|.
\end{equation}
Here and throughout the paper,
$\dr_{x_i}$ stands for the partial derivative with respect to the $i^{th}$ 
coordinate of $x \in \R^d$. We shall work with the graph of $\Gamma$, which is
\begin{equation} \label{defGamma}
\Gamma : = \{(x,\varphi(x)), \, x \in \R^d\} \subset \R^n
\end{equation}
(we shall almost systematically  identify $\R^n$ with $\R^d \times \R^{n-d}$) and denote
\begin{equation} \label{defOmega}
\Omega = \R^n \sm \Gamma.
\end{equation}Ahlfors
The set $\Gamma$ is $d$-Ahlfors regular, which means that
there exists a measure $\sigma$ on $\Gamma$ and a constant $C_\sigma \geq 1$ 
such that 
\begin{equation} \label{defADR}
C_\sigma^{-1} r^{-d} \leq \sigma(B(x,r)) \leq C_\sigma r^{d}
\end{equation}
for $x\in \Gamma$ and $r>0$.
Ahlfors regularity is really a property of the set $\Gamma$ rather than of a particular 
measure $\sigma$, because it is not hard to see (but essentially irrelevant for the present paper) 
that if there is a measure $\sigma$ on $\Gamma$
that satisfies \eqref{defADR}, then $\sigma$ is equivalent to $\mathcal H^d_{|\Gamma}$
and $\mathcal H^d_{|\Gamma}$ satisfies \eqref{defADR} too (with a larger constant).
We mention all this because many of the properties that we prove below do not use more than
the Ahlfors regularity of $\Gamma$.
We shall assume later that $\sigma$ is close to $\mathcal H^d_{|\Gamma}$.

We will need to use Corkscrew points for $\Omega$. These are points $A_{x,r} \in \Omega$,
associated to $x\in \Gamma$ and $r > 0$, such that (for some constant $\tau > 0$)
\begin{equation} \label{Corkscrew}
\tau r \leq \dist(A_{x,r},\Gamma) \leq |A_{x,r}-x| \leq r.
\end{equation}
Corkscrew points are very easy to find here, when $\Gamma$ is a Lipschitz graph
(try $A_{x,r} = (x,\varphi(x)+r e)$ for any unit vector $e \in \R^{n-d}$), but they also exist
when $\Gamma$ is any Ahlfors regular set of dimension $d < n-1$, and we can take $\tau$
to depend only on $n$, $d$, and $C_\sigma$ from \eqref{defADR}; 
see Lemma 11.46 in \cite{DFMprelim}.

\medskip
Since the set $\Gamma$ satisfies \eqref{defADR}, it enters the scope of the elliptic theory 
developed in \cite{DFMprelim}. Let us recall some of the main properties that will be needed. 

Let $L = - \div \Ak \nabla$ be a degenerate elliptic operator for which $\Ak$ satisfies  
\eqref{1.2.1} and \eqref{1.2.2}. We say that $u$ is a weak solution of $Lu=0$, if 
$u \in  W^{1,2}_{loc}(\Omega)$ and
\begin{equation} \label{a2.7}
\int_\Omega \Ak \nabla u\cdot \nabla v = 0 \qquad \forall v\in C^\infty_0(\Omega).
\end{equation}
Here, $W^{1,2}_{loc}(\Omega)$
is the set of functions $u\in L^2_{loc}(\Omega)$ whose derivative 
(in the sense of distribution on $\Omega$) also lies in $L^2_{loc}(\Omega)$.

For each $X \in \Omega$, we can define a (unique)
probability measure $\omega^X$ on $\Gamma$, with the following properties.
For any bounded measurable function $f$ on $\Gamma$, the function $u_f$ defined by 
\begin{equation} \label{a2.6}
u_f(X) = \int_\Gamma f(y) d\omega^X(y)
\end{equation}
is a weak solution.
This is only stated in \cite{DFMprelim} when $f \in C^0_0(\Gamma)$ is continuous and 
compactly supported in $\Gamma$ (see Lemma 9.30 and (iii) of Lemma 9.23 there,
and also (8.1) and (8.14) for the definitions) 
and when $f$ is a characteristic function of Borel set (see Lemma 9.38 there); 
the general case would not be hard, but we do not need it anyway. 

There is also a dense subclass on which we can say a little more. 
Denote by ${\mathcal M}(\Gamma)$ the set of measurable functions on $\Gamma$ and then 
define the Sobolev space
\begin{equation}\label{defH}
H=\dot H^{1/2}(\Gamma) :=\left\{g\in {\mathcal M}(\Gamma): \,
\int_\Gamma\int_\Gamma {|g(x)-g(y)|^2 \over |x-y|^{d+1}}d\sigma(x) d\sigma(y)<\infty\right\}.
\end{equation}
\noindent The class $H \cap C^0_0(\Gamma)$ is dense in $C^0_0(\Gamma)$
(see about 13 lines above (9.25) in \cite{DFMprelim} for the proof of density), 
and if $f \in H \cap C^0_0(\Gamma)$, the solution $u_f$ defined by \eqref{a2.6}
lies in the Sobolev space $W^{1,2}(\Omega,\dist(X,\Gamma)^{d+1-n}dX)$, which means that
\begin{equation} \label{a2.8}
\int_\Omega |\nabla u(X)|^2 \dist(X,\Gamma)^{d+1-n} dX <+\infty,
\end{equation}
and also 
\begin{equation} \label{a2.9}
\text{$u_f$ has a continuous extension to $\R^n$, which coincides with $f$ on $\Gamma$.} 
\end{equation}
See (i) of Lemma 9.23 in \cite{DFMprelim}, together with its proof eight lines above (9.25).

It should be stressed that since $\omega^X$ is a probability measure, $u_f$
is a nondecreasing function of $f \geq 0$. This is of course a manifestation of the maximum principle.
We will need some other properties of $L$ and $\omega^X$ when we prove 
Theorems \ref{Itai1} and \ref{Itsf1}, but these will be recorded later.

\ms
Our aim is to find at least one $\Ak$ satisfying \eqref{1.2.1}--\eqref{1.2.2} 
such that the harmonic measure (that is, any $\omega^X$ as before) is absolutely continuous 
with respect to the Hausdorff measure $\sigma$ on $\Gamma$, 
with $A^\infty$ estimates, and this will require additional assumptions.

First of all, we shall restrict to the case when $d$ is an integer and $\Gamma$ is the graph 
of a Lipschitz function (as in \eqref{defGamma}), with a small enough Lipschitz constant, 
but also $L$ will have the special form 
\begin{equation} \label{defL}
L = - \diver D(X)^{d+1-n} \nabla,
\end{equation}
where $D(X)$ is defined as 
\begin{equation} \label{IdefD}
D(X) = D_\alpha(X)= \Big\{ \int_\Gamma |X-y|^{-d-\alpha} d\sigma(y) \Big\}^{-1/\alpha},
\end{equation}
for some constant $\alpha >0$. 
In fact, in \eqref{IdefD} we can only use measures $\sigma$ 
that do not differ much from the restriction of $\H^d$ to $\Gamma$; 
that is, we can for instance  use the product of  $\H^d_{|\Gamma}$ by any function, 
only if that function is sufficiently close, in $L^\infty$-norm, to $1$. 
See Lemma~\ref{L7.1}, the rest of the argument is the same.
When $d=1$, we can also take
\begin{equation} \label{a2.12}
D(X) = \dist(X,\Gamma),
\end{equation}
but in higher dimensions, $\dist(X,\Gamma)$ does not seem to be smooth enough for
our proof to work roughly for the same reason as why the Jones' $\beta_\infty$ coefficients are not suitable in high dimensions, so we'll have to content ourselves with $D_\alpha$.

Observe that \eqref{defL} means that we chose $\Ak = D(X)^{d+1-n}I$, where $I$ is the 
identity matrix. To be able to use the previous theory, we need to check that 
\eqref{1.2.1}--\eqref{1.2.2} are satisfied, and indeed, 
\begin{equation} \label{equivD-bis}
C^{-1} \dist(X,\Gamma) \leq D_\alpha(X) \leq C\dist(X,\Gamma)
\end{equation}
(see Lemma~\ref{lequivD}).

Our main result states the quantitative mutual absolute continuity of the harmonic measure 
$\omega^X$ above and the surface measure $\sigma$, when $\Gamma$ is a Lipschitz graph
with small enough constant. We give the statement first, and then explain 
the $A^\infty$ condition in our context.

\begin{theorem} \label{Main}
Let $\Gamma \subset \R^n$ be, as in \eqref{defGamma}, the graph of a Lipschitz function 
$\varphi : \R^d \to \R^{n-d}$. Define $L = - \diver D^{d+1-n} \nabla$ as in \eqref{defL},
with $D$ as in \eqref{IdefD}, or possibly, if $d=1$, as in \eqref{a2.12}. Then the associated
harmonic measure (defined near \eqref{a2.6}) is $A^\infty$ with respect to 
$\sigma = \H^d_{\vert \Gamma}$ as soon as the Lipschitz constant $C_0$
of \eqref{defC0} is small enough, depending only on $n$, $d$, and $\alpha > 0$.
This means for instance that for every choice of $\tau \in (0,1)$ and $\epsilon \in (0,1)$, 
there exists $\delta \in (0,1)$, that depends only on $\tau$, $\epsilon$, $n$, $d$, and $\alpha$,
such that for each choice of $x\in \Gamma$, $r>0$, a Borel set $E\subset B_\Gamma(x,r)$, and 
a corkscrew point $X = A_{x,r}(\Gamma)$ (as in \eqref{Corkscrew}), 
\begin{equation} \label{AiTh}
\frac{\omega^X_{\Omega,L}(E)}{\omega^X_{\Omega,L}(B_\Gamma(x,r))} < \delta \Rightarrow \frac{\sigma(E)}{\sigma(B_\Gamma(x,r))} < \epsilon.
\end{equation}
\end{theorem}

\ms
Let us comment a little on the $A^\infty$ condition. In the general context of spaces 
of homogeneous type (metric spaces with a doubling measure $\mu$), 
we say that the measure $\omega$ is $A^\infty$ with respect to the doubling measure $\mu$ 
when the following condition holds:
for every $\epsilon \in (0,1)$, there exists $\delta \in (0,1)$ such that for any $x\in \Gamma$, any $r>0$, and any $E \subset B(x,r)$, we have the implication
\begin{equation} \label{a2.17}
\frac{\omega(E)}{\omega(B(x,r))} < \delta \Rightarrow \frac{\mu(E)}{\mu(B(x,r))} < \epsilon.
\end{equation}
We refer to \cite[Theorem 1.4.13]{KenigB}, \cite[Lemma 5]{CFAinfty},
\cite{GCRF}, or \cite{J} for proofs and additional information about the $A^\infty$ condition. 
It is not hard to show that under these conditions, $\omega$ also is doubling, and $\mu$ 
is $A^\infty$ with respect to $\omega$. That is, the $A^\infty$ relation is symmetric.

For the harmonic measure, in our case or in the original context of co-dimension one,
we cannot say that $\omega^X$ is $A^\infty$ with respect to $\sigma$ in the usual sense, because when $X$ is  very close to the boundary (compared to $r$), $\omega^X$ is very small on most of
$\Gamma \cap B(x,r)$. To say this slightly differently, let $X\in \Omega$ be given and let
$x\in \Gamma$ be such that $|X-x| = \dist(X,\Gamma)$. 
For any $r>0$ large, we can find $y_r\in \Gamma$ such that $2r=|y_r-x|$. 
Then \eqref{a2.17} fails because
\[\frac{\omega^X_{\Omega,L}(B_\Gamma(y_r,r))}{\omega^X_{\Omega,L}(B_\Gamma(x,3r))} \rightarrow 0 \ \ \text{ as } r\to +\infty\]
while
\[\frac{\sigma(B_\Gamma(y_r,r))}{\sigma(B_\Gamma(x,3r))} \geq C^{-1} \]
where $C^{-1}$ depends only on the constant in \eqref{defADR}. 
However, for any $X\in \Omega$ and any ball $B_\Gamma(x,R)$ in $\Gamma$, 
the restriction of $\omega^X_{\Omega,L}$ to $B_\Gamma(x,R)$ is in the class $A^\infty$ 
with respect to the restriction of $\sigma$ to the same $B_\Gamma(x,R)$. 
This latter fact is a straightforward consequence of Theorem \ref{Main}, the Harnack inequality,
and the change of pole (see \cite{DFMprelim}, Lemmas 8.42 and 11.135). If desired, one can replace
the corkscrew point $A_{x,r}$ in our statement by any point of $\Omega \sm B(x,2r)$ at essentially 
no cost, and by any point $X\in \Omega$, but with worse constants that depend on $X$.

\subsection{Main results. Sufficient conditions on elliptic operators on 
$\Omega_0 = \R^n \sm \R^d$ for the absolute continuity of harmonic measure}
The second main result of the present paper is absolute continuity of the harmonic 
measure with respect to the Hausdorff measure on $\RR^d$ for a 
general class of elliptic operators in $\RR^n\setminus \RR^d$, satisfying certain 
structural and Carleson measure conditions. This result is of independent interest 
(see \cite{HKMP} and \cite{KePiDrift} and the discussion below for some analogues in 
co-dimension one, but notice that our sufficient conditions are new even in co-dimension 
one and are weaker than in \cite{KePiDrift}). 
It is also a crucial step of the proof of Theorem~\ref{Main}, which uses a change of variables 
(discussed below) that sends us back to the case where $\Gamma = \R^d \subset \R^n$, 
at the expense of producing a new different operator $L_0$.

We need some definitions. We shall continue to identify $\Gamma_0 = \R^d$ with
$\R^d \times \{ 0 \} \subset \R^n$, and we set $\Omega_0 = \R^n \sm \R^d$.
The running point of $\R^n$ will be denoted by $X = (x,t)$ or $Y=(y,s)$, 
with $x,y\in \R^d$ and $s,t\in \R^{n-d}$. When $t=0$, we may write $x$ instead 
of $(x,0) \in \R^n$.

\begin{definition} \label{defCMI}
A {\bf Carleson measure} on $\Omega_0$ is a positive measure $\mu$ 
on $\Omega_0$ such that for some
constant $C \geq 0$, 
\begin{equation} \label{CM1}
\mu(\Omega_0 \cap B(x,r)) \leq C r^d \ \text{ for $x\in \Gamma_0$ and } r > 0.
\end{equation}
We say that a function $u$ defined on $\Omega_0$ satisfies 
{\bf the Carleson measure condition} when 
\[ 
|u(y,s)|^2  \frac{dy ds}{|s|^{n-d}} \ \text{ is a Carleson measure,} 
\]
that is, when there is a constant $C_u \geq 0$ such that
\begin{equation} \label{CM2}
\int_{(y,s) \in \Omega_0 \cap B(x,r)} |u(y,s)|^2\,\frac{dyds}{|s|^{n-d}} \leq C_u r^d
\ \text{ for $x\in \Gamma_0$ and } r > 0.
\end{equation}
When this happens, we shall more briefly write that
$u \in CM(C_u)$ and refer to the smallest possible $C_u$ in \eqref{CM2} 
as the Carleson norm of $u$ (even though it scales like a square).
\end{definition}

The following result, which 
builds on the ideas from \cite{KKPT,KKiPT,DPP2015}, 
says that for any matrix $L_0$ satisfying \eqref{1.2.1} and \eqref{1.2.2}, the $A^\infty$ absolute continuity of the corresponding harmonic measure follows from
Carleson measure estimates \eqref{CM2} for bounded solutions.
Let us simplify the notation, write $A_0 = A_0(X) = A_0(x,s)$ (instead of $\Ak(X)$)
for the degenerate elliptic matrix that defines our operator $L_0$, and notice that since
$\Gamma_0 = \R^d$, the conditions \eqref{1.2.1} and \eqref{1.2.2} become
\begin{eqnarray} \label{1.2.1a}
&& |s|^{n-d-1} A_0(x,s)\xi \cdot \zeta \leq C_1  |\xi| \,|\zeta|
\ \text{ for } X= (x,s) \in \Omega_0 \text{ and } \xi, \zeta \in \R^n, \\[4pt]
\label{1.2.2a}
&& |s|^{n-d-1} A_0(x,s)\xi \cdot \xi \geq C_1^{-1}  |\xi|^2
\ \, \text{ for } X= (x,s)\in \Omega_0 \text{ and } \xi \in \R^n.
\end{eqnarray}
As before, the operator $L_0=-\div A_0 \nabla$ satisfies the assumptions of \cite{DFMprelim},
so we can construct harmonic measures on $\Gamma_0$
that we denote by $\omega^X = \omega^X_{\Omega_0,L_0}$, $X \in \Omega_0$,
which satisfies the properties described near \eqref{a2.6}. In particular, for every Borel set 
$H \subset \Gamma_0$, we can define a weak solution $u_H$ by
\begin{equation} \label{a2.24}
u_H(X) = \omega^X_{\Omega_0,L_0}(H)  \ \text{ for } X \in \Omega_0
\end{equation}
(compare with \eqref{a2.6}). Here is our sufficient condition for 
$A^\infty$ absolute continuity,
which will be proved in Section \ref{SAinfty}.

\begin{theorem}\label{Itai1} 
Let $L_0=-\div A_0 \nabla$ be a degenerate elliptic operator (with real coefficients) 
on $\Omega_0 = \R^n \setminus \R^d$, and assume that $A_0$ satisfies 
\eqref{1.2.1a}  and \eqref{1.2.2a}. Also assume that there is a constant $C_L$ such that
for any Borel set $H\subset \Gamma_0 = \R^d$, the solution $u_H$ of \eqref{a2.24} is such that 
\begin{equation} \label{a2.26}
|t|\nabla u_H \in CM(C_L).
\end{equation}
Then the harmonic measure $\omega^X_{\Omega_0,L_0}$ is $A^\infty$ 
(with the definition of Theorem \ref{Main}) with respect to the Lebesgue measure on $\R^d$.
\end{theorem}

\ms
The next stage is to find reasonable conditions on $A_0$ that imply the estimate \eqref{a2.26},
and hence the $A^\infty$-absolute continuity of $\omega_{\Omega_0,L_0}$. 
Recall from the discussion above that even in co-dimension 1 one does not expect the absolute 
continuity of the harmonic measure with respect to the Lebesgue measure to hold for all elliptic 
operators, due to the counterexamples in \cite{CFK}.

We need more notation. 
In the sequel, $Q$ denotes a cube in $\R^d$ and $l(Q)$ its sidelength. 
The truncated cone of approach to $x\in \Gamma_0$ is
\begin{equation} \label{a2.27}
\gamma^Q(x)=\big\{(y,s)\in \R^n\setminus \R^d :\, |y-x|\leq a |s|
\text{ and } 0<|s|< l(Q) \big\},
\end{equation}
where $a > 0$ is any given constant, and often we just take $a=1$. 
The notation is slightly misleading, because $\gamma^Q(x)$ merely depends on
$l(Q)$ rather than $Q$, but it is convenient and often used. 

Then we define a localized  square function $S^Qu$ and a localized 
non-tangential maximal function ${\mathcal N}^Q$, both on $\Gamma_0$, by
\begin{equation} \label{a2.28}
S^Qu(x)
:=\left(\int_{(y,s)\in \gamma^Q (x)} |\nabla u(y,s)|^2 \, \frac{dyds}{|(y,s)-(x,0)|^{n-2}}\right)^{1/2}
\end{equation}
and
\begin{equation} \label{a2.29}
N^Q u(x)
:=\sup_{(y,s)\in \gamma^Q (x)} |u(y,s)| 
\end{equation}
for $x\in \Gamma_0$.
Our main theorem for elliptic operators on $\RR^n\setminus \RR^d$, which yields square 
function/non-tangential maximal function estimates and ultimately $A^\infty$ 
property of harmonic measure, is the following.

\begin{theorem}\label{Itsf1} 
Let $A_0$ be a degenerate elliptic matrix satisfying \eqref{1.2.1a} and \eqref{1.2.2a} in 
$\Omega_0 = \R^n \setminus \R^d$, and set then $L_0 = - \diver A_0 \nabla$ as above.
Define the rescaled matrix $\A$ by $\A = |t|^{n-d-1} A_0$, so that now 
$L_0 = - \diver |t|^{d+1-n}\A \nabla$, and assume that $\A$ has the following block structure: 
\begin{equation}\label{Ieqsf1}
\A(X)= \left( \begin{array}{cc}
\A^1(X) & \A^2(X) 
\\  \C^3(X) & b(X)I_{n-d}+  \C^4(X)  \end{array} \right), 
\end{equation}
where $\A^1(X)$ is a matrix in $M_{d\times d}$, $\A^2(X)$ is a matrix in $M_{d\times (n-d)}$,
$b$ is a function on $\Omega_0$, $I_{n-d}$ is the identity matrix in
$M_{(n-d)\times (n-d)}$, and in addition we can find constants $M \geq 0$ and $\lambda \geq 1$
such that
\begin{equation} \label{a2.32}
\lambda^{-1} \leq b \leq \lambda \ \text{ on } \Omega_0,
\end{equation}
\begin{equation} \label{a2.33}
|t|\nabla b \in CM(M),
\end{equation}
\begin{equation} \label{a2.34}
\C^3, \, \C^4 \in CM(M).
\end{equation}
Then there is a constant $K > 0$, that depends only on $n$, $d$, the elliptic parameter
in \eqref{1.2.1a} and \eqref{1.2.2a}, $\lambda$, $M$, and $a$ (the aperture in \eqref{a2.27}),
and a constant $k_0 > 1$ that depends only on the aperture $a$,
such that if $u$ is a weak solution to $L_0u=0$, then for every cube $Q \subset\R^d$,
\begin{equation}\label{Ieqsf2a} 
\|S^Q u\|_{L^2(Q)}^2 \leq K \|N^{2Q} u\|_{L^2(k_0Q)}^2.  
\end{equation}
Here $k_0Q$ stands for the cube with the same center as $Q$ and sidelength $k_0 l(Q)$.

Furthermore, under the same conditions on the matrix $\A$, the harmonic measure 
$\omega^X_{\Omega_0,L_0}$ is $A^\infty$ 
with respect to the Lebesgue measure on $\R^d$ (with the definition of Theorem \ref{Main}).
\end{theorem}

Here and throughout the paper, when we say in \eqref{a2.34} that $\C^j \in CM(M)$, 
we mean that each entry $\C^j_{i,k}$ of $\C^j$ lies in $CM(M)$; the fact that we take 
the sup norm on matrices rather than a more reasonable norm is irrelevant here.

\ms
The theorem will be proved in Sections \ref{SSQR} and \ref{SAinfty}, 
Theorem \ref{tsf1} and Corollary~\ref{cai5}.
Let us discuss some aspects of its statement. First, the finiteness of the quantities in \eqref{Ieqsf2a} is not guaranteed (but we do mean, as a part of the statement, that the finiteness of the right-hand side implies the finiteness of the left-hand side); but it holds in the following case. 
Let $H$ be any Borel set on $\RR^d$ and $u_H$ be the solution with data given by the characteristic function of $H$, defined in \eqref{a2.24}. Then (since $\omega$ is a probability measure) $N^{2Q}(u_H) \leq 1$, which by \eqref{Ieqsf2a} implies that for every cube $Q \subset\R^d$
\begin{equation} \label{a8.16}
\|S^Q u_H\|_{L^2(Q)}^2 \leq C |Q|.
\end{equation}
The latter, by Fubini's theorem, yields \eqref{a2.26} -- see Remark~\ref{rSfCm}.
This is the reason why, having proved \eqref{Ieqsf2a} and Theorem~\ref{Itai1}, we can conclude 
that for operators satisfying conditions of Theorem~\ref{Itsf1} the harmonic measure 
$\omega^X_{\Omega_0,L_0}$ is $A^\infty$ 
with respect to the Lebesgue measure on $\R^d$ (with the definition of Theorem \ref{Main}).

Turning to the conditions on $A_0$ proposed in Theorem \ref{tsf1}, we observe that
we did not impose any condition on the first $d$-lines of the elliptic and
bounded matrix $\A$. In the case of co-dimension 1, the reader should compare to \cite{KePiDrift} where the {\it full} matrix $\A$ is assumed to satisfy $|t|\nabla \A\in CM(M)$. One could always add to the latter a Carleson measure perturbation, that is, to treat  $\A+\C$, with $|t|\nabla \A\in CM(M)$ and $\C\in CM(M)$ due to \cite{FKP}, but still imposing these conditions on the full matrix of coefficients (as opposed to our statement appealing only to the last $n-d$ lines). On top of it, to be even more precise, both \cite{KePiDrift} and \cite{FKP} require slightly stronger Carleson measure conditions, dealing with the suprema of coefficients on Whitney cubes, that is, $(|t|\nabla \A)_W\in CM(M)$ and $(\C)_W\in CM(M)$ where $F_W(x,t)=\sup_{B((x, t), t/2)} |F|$, $(x, t)\in \RR^n$. In all those directions, even in co-dimension one the result of Theorem~\ref{Itsf1} is new. Observe, for instance, that it yields the $A^\infty$ property for the harmonic measure of an operator associated to a block matrix 
\begin{equation}\label{Ieqsf1b}
\A(X)= \left( \begin{array}{cc}
\A^1(X) & 0
\\ 0& I_{n-d} \end{array} \right), 
\end{equation}
and for its Carleson measure perturbations, $\A+\C$, {\it without any assumptions on $\A^1$.} In the case when $\A_1$ is $t$-independent, this is a consequence of the resolution of the famous Kato problem \cite{AHLMcT}, but the observation that one take any elliptic $\A^1$, to our knowledge, is new.

Note also that in the case of co-dimension 1 one could take ${\mathcal B}^3 + \C^3$ in place of $\C^3$, with $|t|\nabla {\mathcal B}\in CM(M)$, at the expense of a harmless drift term (this is the reason why our results include the aforementioned case $|t|\nabla \A\in CM(M)$). A version of this should be also possible in our context, but we choose not to develop it here as strictly speaking one would have to revisit  \cite{DFMprelim} for construction of solutions of operators with drift terms.

We remark, parenthetically, that in co-dimension 1 another important class of elliptic operators whose harmonic measure is absolutely continuous with respect to the Hausdorff measure on $\RR^d$ (or on a Lipschitz graph) is the class of  operators with $t$-independent coefficients \cite{HKMP}. Those do not make much sense in our context as the dependence on the distance to the boundary and hence, on $|t|$, is exactly the feature that allows us to access the higher codimension.

\ms
At this point, we see that in order to prove Theorem \ref{Main}, we now want to construct a change
of variables $\rho$ that transforms the operator $L$ in that theorem into an operator $L_0$ that satisfies the assumptions of Theorem \ref{Itsf1}.

\subsection{A bi-Lipschitz change of variables}
\label{Abil}

Let us now say a few words about the change of variables that we will construct.
This will be a bi-Lipschitz mapping $\rho : \R^n \to \R^n$, with 
\begin{equation} \label{a2.42}
\rho(\Gamma_0) = \Gamma \ \text{ and (hence) } \ \rho(\Omega_0) = \Omega.
\end{equation}
When we conjugate our operator $L$ by $\rho$, we obtain the operator 
$L_\rho = - \div A_\rho \nabla$, which satisfies the conditions \eqref{1.2.1}--\eqref{1.2.2} 
relatively to the boundary $\Gamma_0 = \R^d$ and the domain $\Omega_0 = \R^n \setminus \R^d$. Besides, the harmonic measure associated to $\Omega$ and $L$ is $A^\infty$ with respect to the measure $\sigma$ when the harmonic measure associated to $\Omega$ and $L$ is $A^\infty$ 
with respect to the Lebesgue measure $ \lambda^d$ on $\Gamma_0 = \R^d$,
where the $A^\infty$-absolute continuity of the harmonic measure is taken as in Theorem \ref{Main}. 
Indeed, since $\rho$ is bi-Lipschitz, it is easy to check that for a Borel set $H\subset \Gamma$ the Hausdorff measure of $H$ is equivalent to the Lebesgue measure of $\rho^{-1}(H)$. Moreover, if $X\in \Omega$, $\omega^{X}_{\Omega,L}(H) = \omega^{\rho^{-1}(X)}_{\Omega_0,L_\rho}(\rho^{-1}(H))$; and since again $\rho$ is bi-Lipschitz, the fact that $X = A_{x,r}(\Gamma)$ is a Corkscrew point implies that $\rho^{-1}(X)$ is Corkscrew point of the form $A_{\rho^{-1}(x),r}(\Gamma_0)$. 
Theorem \ref{Main} will be thus proven if $L_\rho$ satisfies the assumption of Theorem \ref{Itsf1}.

The change of variables $\rho$  will also be a differentiable mapping, with an invertible differential, denoted by $\Jac$, 
such that 
\begin{equation} \label{a2.43}
\frac12 |v| \leq |\Jac(X) \cdot v| \leq 2 |v| \ \text{ for $X\in \Omega_0$ and } v \in \R^n.
\end{equation}
For the moment, this is quite natural, and easy to obtain if the Lipschitz constant $C_0$
for the mapping $\varphi$ whose graph is $\Gamma$ is small enough.
Now the goal of the change of variables is to transform the operator $L$ defined on $\Omega$
by the matrix $\Ak$ defined by \eqref{defL}, with a function $D$ coming from \eqref{IdefD}
or perhaps \eqref{a2.12}, into an operator $L_0$ on $\Omega_0$ that 
satisfies the assumptions of Theorem \ref{Itsf1}. A computation, that will be done
in Lemma \ref{newLaplacianprop}, shows that $L_0$ is the operator associated to the normalized
matrix $\A$, where 
\begin{equation} \label{defAc}
\A(x,t) = \left(\frac{|t|}{D(\rho(x,t))}\right)^{n-d-1} |\det(\Jac(x,t))| (\Jac(x,t)^{-1})^T \Jac(x,t)^{-1},
\end{equation}
where $X = (x,t)$ is the running point of $\Omega_0$ and $(\Jac(x,t)^{-1})^T$ denotes 
the transpose of the inverse of the differential $\Jac(x,t)$. 

The general shape of the matrix, which is given by $(\Jac(x,t)^{-1})^T \Jac(x,t)^{-1}$, is 
symmetric, but since we want the block on the bottom right of $\A$ to be Carleson-close 
to a matrix $b(X) I_{n-d}$, we want the $t$-part of $\Jac(X)$ to be as close as possible 
to a multiple of an isometry.
In addition, \eqref{a2.33} tells us that we do not want to let $b(X)$ vary too fast,
which will force us to control the Jacobian determinants of the matrices.
Let us say how these constraints influence our definition of $\rho$. In this case the formula
tells a lot on what we want to do.

We shall use a smooth, nonnegative, radial function $\eta : \R^d \to \R$, compactly supported in
the unit ball and such that $\int \eta(x) dx = 1$, and its dilations $\eta_r$, $r > 0$, defined by
\begin{equation} \label{defeta}
\eta_r(x) = \frac1{r^d} \ \eta\left( \frac {x}{r}  \right).
\end{equation}
Then we define functions $\varphi_r : \R^d \to \R^{n-d}$ by
\begin{equation} \label{a2.46}
\varphi_r = \varphi \ast \eta_r,
\end{equation}
and then $\Phi$ and $\Phi_r$, with values in $\R^n$, by
\begin{equation} \label{a2.47}
\Phi(x) =  (x,\varphi(x)), \ \text{ and } \ 
\Phi_r(x) =  (x,\varphi_r(x)) = (\Phi \ast \eta_r)(x)
\end{equation}
for $x\in \R^d$. Thus $\Phi$ is the standard parameterization of $\Gamma$, and
$\Phi_r$ parameterizes a graph $\Gamma_r$ which is a nice approximation of
$\Gamma$ at the scale $r$. 

Denote by $X = (x,t)$ the running point of $\R^n$, and also set $r = |t|$. We will take
\begin{equation} \label{a2.48}
\rho(x,0) = \Phi(x) = (x,\varphi(x))
\ \text{ for } x\in \R^d
\end{equation}
and
\begin{equation} \label{defrhointro}
\rho(x,t) = \Phi_r(x) + h(x,t) R_{x,r}(0,t) = (x,\varphi_r(x)) + h(x,t) R_{x,r}(0,t)
\ \text{ for } (x,t) \in \Omega_0,
\end{equation} 
where $R_{x,r}$ is a linear isometry
of $\R^n$ and $h$ is a positive function on $\Omega_0$ such that 
$C^{-1} \leq h \leq C$ for some positive constant $C$.

We shall construct $R_{x,r}$ so that it maps $\R^d$ to the 
$d$-plane $P(x,r)$ tangent to $\Gamma_r$ at the point $\Phi_r(x)$.
In fact, we are even more interested by the fact that
$R_{x,r}$ maps $\R^{n-d} = (\R^d)^\perp$ to the orthogonal plane to $P(x,r)$
at $\Phi_r(x)$, but the two are are equivalent anyway. Our function $h$ will change slowly
(in a way that is controlled by Carleson measures), so \eqref{defrhointro} is a way to make
$\rho$ close to an isometry in the $t$-variables. It is also important that the variations
of $\rho$ in the $t$ variables are almost orthogonal to the approximate direction
of the tangent plane, which is a property that other standard changes of variables
do not have.

The role of $h$ is a little more subtle. It is connected to the fact that we want to control
the variations of the coefficient $b(X)$ above. We will choose it 
to control the $D$-distance from $\rho(X)$ to $\Gamma$, in such a way  
that 
\begin{equation} \label{DrhoisCM}
\frac{|t|}{D(\rho(y,t))} - 1 \text{ satisfies the Carleson measure condition}, 
\end{equation}
and this in turn will allow us to control $b$.

The idea of changing variables to get a slightly more complicated operator in
the simpler domain $\Omega_0$ is classical. There are, however, only two changes of variables 
that have been used in this context before. The first one, $\rho(x, t)=(x, t+\vp(x))$, $(x, t)\in \RR^n$, yields the operators with $t$-independent coefficients (see, e.g., \cite{JK}, \cite{KKPT}, \cite{HKMP}) 
and the second one, somewhat closer to ours, $\rho(x, t)=\Phi_r(x) +(0, t)$, $(x, t)\in \RR^n$, 
known as the Necas-Kenig-Stein change of variables (see, e.g., \cite{KePiDrift}), 
yields the operators with coefficients whose gradient is a Carleson measure.  
However, both of them move points vertically and lack some of the delicate 
properties that we now require, almost an isometry for $\rho$ and Carleson measure estimates 
for $D\circ \rho$ to mention only a few.
Even in co-dimension $1$, the present change of variables is new in the context of elliptic operators, 
and seems to be much more adapted because it tends to preserve the orthogonal direction.
On the other hand, in the realm of Reifenberg flat sets, the idea of trying to preserve the orthogonal
direction is not new; see for instance \cite{Toro}, and more closely \cite{DT} where somewhat  similar problems arise.

Unfortunately, our construction does not work when the Lipschitz constant $C_0$ is large
(because injectivity fails), and this does not seems easy to fix. 
So we may need to find other ways to treat general Lipschitz graphs and other nice sets.

\ms
At this point our meticulously chosen $\rho$ and $D$ are
good enough to apply Theorem \ref{Itsf1}
(we even get that the matrix $\A$ satisfies slightly stronger conditions, which will be described
in Lemma~\ref{MainS1}), and hence get Theorem~\ref{Main}. 

Let us also mention that some regularity of $D$ is needed to establish \eqref{DrhoisCM}.
For the soft distance functions of \eqref{IdefD}, we will get the desired control on $D$,
which in particular allows us to define $\rho$ so that \eqref{DrhoisCM} holds, from
a result of Tolsa \cite{Tolsa09} that controls some Wasserstein distances between $\sigma$
and Lebesgue measures on $d$-planes. These distances control both the flatness of $\Gamma$
and the repartition of $\sigma$ on $\Gamma$; that is, both the regularity of $\Gamma$ and $D$. 

In the case of the Euclidean distance in \eqref{a2.12}, we need the $L^\infty$-based
P. Jones numbers $\beta$, which we can only control well in dimension $d=1$; otherwise,
the distance does not seem to be regular enough for our method to apply.

\subsection{Organization of the paper}

In Section \ref{Sorthbasis}, we set the notation, and construct an orthonormal basis of the 
approximate tangent plane $P(x,r)$, that will be used to define the linear isometries $R_{x,r}$.
This completes the definition of $\rho$, modulo the choice of $h$ which is left open for the moment.

In Section \ref{Srho} we prove that if $C_0$ is small enough, $\rho$ is a locally smooth
and globally bi-Lipschitz change of variable, with \eqref{a2.42} and \eqref{a2.43}. The main point
is the injectivity of $\rho$, which we prove with a simple topological argument. We do not
give a formula for $\rho^{-1}$.

In Section \ref{SCMforJac} we write the matrix
$|\det(\Jac(x,t))| (\Jac(x,t)^{-1})^T \Jac(x,t)^{-1}$ that shows up in \eqref{defAc}
in an appropriate form, and prove Carleson measure estimates for some of its
coefficients. The estimates in this section come from elementary linear algebra 
(to decompose the matrix) and standard Littlewood-Paley theory for bounded functions
(for the Carleson bounds).

In Section \ref{SCMforD}, we use P. Jones $\beta$-number and localized 
Wasserstein
distances, together with the uniform rectifiability of $\Gamma$, to control the geometry
of $\Gamma$ and $\sigma$, choose the function $h$, and prove \eqref{DrhoisCM}
in particular. 

Once all this is done, we can check in Section \ref{Scollect} that the matrix of $L_0$
of the conjugated operator satisfy stronger assumptions than needed for 
Theorem \ref{Itsf1}. As was explained above, this completes the proof of 
Theorem \ref{Main}, modulo Theorems \ref{Itsf1} and \ref{Itai1}.

Theorem \ref{Itsf1} and Theorem \ref{Itai1} are of independent interest, and their
proofs, given in Section \ref{SSQR} and Section \ref{SAinfty}, can be read independently 
from the rest of this paper.

\ms
In the sequel, the letter $C$ denotes a positive constant whose dependence is either recalled or obvious from the context, and that may change from one inequality to another. 
The expression $A\lesssim B$, where $A$ and $B$ are two expressions depending on some parameters, means that there exists $C>0$ such that $A \leq CB$, and the dependence of $C$ 
on the parameters will be either given by the context or recalled. 
In addition, $A \approx B$ is used if $ A \lesssim B \lesssim A$.
  
\section{Construction of a field of orthonormal bases of $\R^n$}

\label{Sorthbasis}

In this section we use the Gram-Schmidt orthogonalization algorithm to construct an orthonormal 
basis of $\R^n$ that starts with an orthonormal basis of the tangent plane to $\Gamma_r$
at $\Phi_r(x)$.

For this section we only need to know that $\Gamma$ is a Lipschitz graph, not that 
the Lipschitz constant $C_0$ is small; the constants will be allowed to depend on $C_0$,
but if $C_0 \leq 1$ as in the next sections, they depend only on $n$ and $d$. 

We start with some amount of notation.
We have our Lipschitz function $\varphi$, $C_0$ as in \eqref{defC0}, the graph
$\Gamma$ and its parameterization $\Phi$, as in \eqref{defGamma} and \eqref{a2.47},
and their approximations $\varphi_r$, $\Gamma_r$, and $\Phi_r$.

\medskip

Here is the {\bf notation for derivatives} that we shall try to use systematically. 
If $f$ is a function defined on $\R^d$, on $\R^d\times (0,+\infty)$, or on $\Omega_0 = \R^d \times (\Rt)$), then, for any $i\in \{1,\dots,d\}$, the notation $\dr_{x_i} f$ denotes the derivative of $f$ with respect to the $i^{th}$ coordinate (of the first variable). We use $\dr_r f$ to denote the derivative of a function $f$ defined on $\R^d\times (0,+\infty)$ with respect to the second variable and, for $j\in \{d+1,\dots,n\}$, the function $\dr_{t_j} f$ is the derivative of a function $f$ defined on $\R^d \times (\Rt)$ with respect to the $(j-d)^{th}$ coordinate of the second variable (or the $j^{th}$ coordinate, if $f$ is seen as a function defined on a subset of $\R^n$).
When $f$ takes value in $\R^{n-d}$ (resp. in $\R^n$), then the quantities $\dr_{x_i} f$, $\dr_r f$ and $\dr_{t_j} f$ are lines vectors, and $|\dr_{x_i} f|$, $|\dr_r f|$ and $|\dr_{t_j} f|$ denotes their classical Euclidean norm in $\R^{n-d}$ (resp. in $\R^n$). The terms $\nabla_x f$, $\nabla_{x,r} f$ and $\nabla_{x,r} f$ are used for
\begin{equation} \label{defgradient}
\begin{pmatrix} \dr_{x_1} f \\ \vdots \\ \dr_{x_d} f \end{pmatrix}, \qquad 
\begin{pmatrix} \dr_{x_1} f \\ \vdots \\ \dr_{x_d} f \\ \dr_r f \end{pmatrix}, 
\quad  \text{ and } \quad 
\begin{pmatrix} \dr_{x_1} f \\ \vdots \\ \dr_{x_d} f \\ \dr_{t_{d+1}} f \\ \vdots \\ \dr_{t_n} f  \end{pmatrix}
\end{equation}
respectively. Note that the latter quantities are matrices when $f$ is vector valued. 
In addition, $|\nabla_{x} f|$, $|\nabla_{x,r} f|$ and $|\nabla_{x,t} f|$ are respectively the quantities
\begin{equation} \label{defnormgradient}
\left(\sum_{i=1}^d |\dr_{x_i} f|^2 \right)^\frac12, \quad 
\left(\sum_{i=1}^d |\dr_{x_i} f|^2 + |\dr_r f|^2 \right)^\frac12, \quad  \text{and} \quad \left(\sum_{i=1}^d |\dr_{x_i} f|^2 + \sum_{j=d+1}^n |\dr_{t_j} f|^2 \right)^\frac12.
\end{equation}
Finally, the set of 
second derivatives of $f$ with respect to the $x$-variables is 
written $\nabla^2_{x} f$. 
If $f$ takes values in $\R$, it corresponds to the $d\times d$ matrix
$\nabla_x (\nabla_x f)^T$; if $f = (f_1,\dots,f_k)$ is vector valued, 
we see $\nabla^2_{x} f$ as the collection 
of matrices $\{\nabla^2_{x} f_j\}_{1\leq j \leq k}$. 
Besides, the norm $|\nabla^2_x f|$ denotes the quantity
\begin{equation} \label{defnormgradient2}
\left(\sum_{j=1}^k |\nabla_x (\nabla_x f_j)^T|^2 \right)^\frac12,
\end{equation}
which correspond to the $\ell^2$-norm $(\sum_{i,\ell,j} |\dr_{x_i} \dr_{x_\ell} f_j|^2)^\frac{1}2$. The definition above is modified accordingly to give 
sense to $\nabla_{x,r}^2 f$, $\nabla_{x,t}^2 f$ or $\nabla_{x,r}\nabla_x f$.

\medskip

Recall that $\rho$ will have the form \eqref{defrhointro}, and $R_{x,r}$ will be defined as the 
linear isometry that maps the canonical orthonormal basis $\mathcal B_e$ of $\R^n$
to a new orthonormal basis $\mathcal B_{vw}$ that we construct now.

\begin{definition}[Coordinate basis]
For any $i\in \{1,\dots,n\}$, we denote by $e^i$ the unit vector in $\R^n$ that 
has $1$ in the $i^{th}$ coordinate and $0$ in the other coordinates.
The family $(e^1,\dots,e^n)$ forms an orthonormal basis which we call $\mathcal B_e $.

By notation abuse, when $i\in \{1,\dots,d\}$, we also use $e^i$ for the unit vector in $\R^d$ 
that has $1$ on the $i^{th}$ coordinate and $0$ on the other coordinates, 
and when $j\in \{d+1,\dots,n\}$, we use $e^j$ for the unit vector in $\R^{n-d}$ that has $1$ on the $(j-d)^{th}$ coordinate and $0$ on the other ones.
\end{definition}

Fix $x\in \R^d$ and $r > 0$, and denote by $P(x,r)$ the tangent $d$-plane to $\Gamma_r$ 
at $\Phi_r(x)$. Also call $P'(x,r)$ the vector $d$-plane parallel to $P(x,r)$, and 
$P'(x,r)^\perp$ its orthogonal complement. 

We start with a basis of $P'(x,r)$, which is given by the $d$ vectors $\hat v^i$,
$1 \leq i \leq d$ given by
\begin{equation} \label{defbarv}
\hat v^i(x,r) = \dr_{x_i}\Phi_r(x) = (e^i,\dr_{x_i} \varphi_r(x)).
\end{equation}
We also have a basis of $P'(x,r)^\perp$, which is composed of the normal vectors
$\hat w^j$, $d+1 \leq j \leq n$, defined as
\begin{equation} \label{defbarw}
\hat w^j(x,r) = ((-\nabla_x \varphi_r^j)^T,e^j),
\end{equation}
where $\varphi_r^j$ is the $j^{th}$ coordinate of $\varphi_r$. Thus the $d$ first
coordinates of $\hat w^j(x,r)$ are the partial derivatives $-\dr_{x_i} \varphi_r^j$.
These vectors are clearly independent (because of their $e^j$ part), and they are orthogonal
to $P'(x,r)$ because 
\begin{equation} \label{barviwjorth}
 \left< \hat v^i,\hat w^j \right> 
 = - \dr_{x_i}\varphi_r^j \left< e^i,e^i\right> +  \left<\dr_{x_i} \varphi_r(x), e^j \right> = 0.
\end{equation}
It will be useful to know that 
\begin{equation} \label{a3.8}
|\hat v^i| \leq \sqrt{1+C_0^2} \quad \text{ and }  \quad |\hat w^j| \leq \sqrt{1+C_0^2}
\end{equation}
for $1 \leq i \leq d$ and $d+1 \leq j \leq n$, just because $|\nabla_x\varphi| \leq C_0$.

Next we apply the Gram-Schmidt construction to replace the $\hat v^i$, $1 \leq i \leq d$, 
by an orthonormal basis $v^i$, $1 \leq i \leq d$, of $P'(x,r)$. Since we want to have estimates
on the various coefficients that arise, we recall how it goes.

We start with $v^1 = |\hat v^1|^{-1} \hat v^1$. Then, assuming that the $v^\ell$,
$\ell \leq i$ have already been chosen, we first replace $\hat v^{i+1}$ by
\begin{equation} \label{a3.9}
\wt v^{i+1}=\hat v^{i+1} - \sum_{\ell=1}^i \left< \hat v^{i+1},v^\ell \right> v^\ell
\end{equation}
and finally
\begin{equation} \label{a3.10}
v^{i+1} = |\wt v^{i+1}|^{-1} \ \wt v^{i+1}.
\end{equation}
It is well known that the procedure gives a new (and orthonormal) basis of $P'(x,r)$,
which we call $\mathcal B_v$. Notice that $\left< v^\ell, e^i \right> = 0$ for $\ell < i$ 
(because $v^\ell$ lies in the span of the $\hat v^m$, $m \leq \ell$), hence
\begin{equation} \label{a3.11}
\left< \wt v^{i+1},e^i \right> = \left< \hat v^{i+1},e^i \right>
- \sum_{\ell=1}^i \left< \hat v^{i+1},v^\ell \right> \left< v^\ell, e^i \right>
= \left< \hat v^{i+1},e^i \right> = 1,
\end{equation}
which in turn implies that $|\wt v^{i+1}| \geq 1$ and 
\begin{equation} \label{a3.12}
|\wt v^{i+1}| \geq 1.
\end{equation}
When $i=1$, we simply set $\wt v^1 = \hat v^1$ and get that $|\wt v^1| \geq 1$
as well. This is reassuring: we get confirmation that we never have to divide by $0$,
but we knew that because the $v^\ell$, $1 \leq \ell \leq i$, span the same space
as the $|\hat v^{l}|$, $1 \leq \ell \leq i$.

We now do the same thing with the $\hat w^j$: we apply the Gram-Schmidt process to 
construct an orthonormal basis $\mathcal B_w = (w^{d+1}, \ldots, w^n)$ of $P'(x,r)^\perp$,
with formulas like 
\begin{equation} \label{a3.13}
\wt w^{j+1}=\hat w^{j+1} 
- \sum_{\ell=d+1}^j \left< \hat w^{j+1},w^\ell \right> w^\ell
\end{equation}
for $d + 1 \leq  j < n$ and 
\begin{equation} \label{a3.14}
w^{j} = |\wt w^{j}|^{-1} \ \wt w^{j}
\end{equation}
for $d+1 \leq j \leq n$. The construction and estimates are exactly the same, except that
we exchange the first and last sets of coordinates.

Finally we put $\mathcal B_v$ and $\mathcal B_w$ together to get an orthonormal basis
$\mathcal B_{vw} = (v^1, \ldots v^d, w^{d+1}, \ldots, w^n)$ of $\R^n$. 
Finally, $R_{x,r}$ is the linear mapping that sends $\mathcal B_e$ to $\mathcal B_{vw}$.
That is,
\begin{equation} \label{a3.15}
R_{x,r}(y,u) = \sum_{i=1}^d y_i v^i(x,r) + \sum_{j= d+1}^n t_i w^j(x,r),
\end{equation}
where we now write the dependence on $x$ and $r$ because we shall start worrying
on the variations of all these functions. Notice that $R_{x,r}$ is a linear isometry because
it maps an orthonormal basis to another one, and by construction and with a slight abuse of notation
\begin{equation} \label{a3.16}
R_{x,r}(\R^d) = P'(x,r) \ \text{ and } \  R_{x,r}(\R^{n-d}) = P'(x,r)^\perp.
\end{equation}

Now we worry about the smoothness of the vectors $v^i$ and $w^j$.
We start with the easy soft result.

\begin{lemma} \label{viCinfty}
The vector fields $v^i$, $1 \leq i \leq d$, and $w^j$, $d+1 \leq j \leq n$, are $C^\infty$
functions on $\R^d \times (0,+\infty)$.
\end{lemma}

\begin{proof}
We start with the function $(x,r) \to \varphi_r(x) = \varphi \ast \eta_r(x)$,
which is $C^\infty$ on $\R^d \times (0,+\infty)$ by standard results (and we shall
have ample opportunities to compute some of its derivatives). 
Then all the $\hat v^i$, $\wt v^i$, $v^i$, and their $w$-counterparts are smooth too,
by \eqref{a3.9}, \eqref{a3.10}, \eqref{a3.13}, \eqref{a3.14}, and an easy induction argument.
The fact that by \eqref{a3.12} and its analogue for $\wt w^j$ we never divide by $0$ helps here.
\end{proof}

In the last result of the section, we want to prove that the derivatives of $v^i$ and $w^j$ are controlled by the second derivatives of $\varphi_r$. These bounds will be important in 
later sections, to prove size estimates on $\rho$ (that lead to bi-Lipschitz estimates), 
and then Carleson measure estimates on some of its derivatives.

\begin{lemma} \label{vrbyphirr}
For any $i\in \{1,\dots,d\}$, there holds
\begin{equation} \label{a3.19}
|\nabla_{x,r} v^i| \leq C |\nabla_{x,r} \nabla_x \varphi_r |,
\end{equation}
and for any $j\in \{d+1,\dots,d\}$, one has
\begin{equation} \label{a3.20}
|\nabla_{x,r} w^j| \leq C |\nabla_{x,r} \nabla_x \varphi_r |,
\end{equation}
where in both cases the constant $C>0$ depends only on
$n$ and the Lipschitz constant $C_0$. If we assume that $C_0 \leq 1$,
$C$ depends only on $n$.
\end{lemma}

\begin{proof}
We start with the vectors 
$\hat v^i = (e^i, \dr_{x_i} \varphi_r)$ and $\hat w^j = ((-\nabla_x \varphi_r^j)^T, e^j)$,
for which the desired result holds because their coordinates are directly written in terms of 
$\nabla_x \varphi_r$. 

Then we follow the Gram-Schmidt algorithm and observe that the coefficients of the $v^i$
and $w^j$ are obtained from those of the $\hat v^i$ and $\hat w^j$ by a bounded number
of algebraic computations involving taking sums, products, and inverses of functions that are 
never smaller than $1$ (see \eqref{a3.12}). That is, each $v^i$ or $w^j$ has a simple
algebraic expression in terms of $\nabla_x \varphi_r$. We compute the derivatives of 
these expressions and get \eqref{a3.19} and \eqref{a3.20}. A more detailed proof would
work by induction on $i$ or $j-d$ and give precise bounds, but we decided not to do these
computations.
\end{proof}

\section{The change of variables $\rho$}
\label{Srho}

The goal of this section is to prove that the map $\rho$ defined by \eqref{defrhointro}
is a bi-Lipschitz change of variable, and at the same time prove some estimates on its 
derivative $\Jac(x,t)$.

Recall that in addition to $R_{x,r}$, there is an auxiliary function $h$ that appears in the
definition of $\rho$. In this section, we only need to assume that 
there exist constants $C_{h1} \geq 1$ and $C_{h0} \geq 0$ such that 
\begin{equation} \label{H1}
C_{h1}^{-1} \leq h(x,t) \leq C_{h1}\ \text{ for }  (x,t) \in \Omega_0, 
\end{equation}
$h$ is continuously differentiable on $\Omega_0$, 
and  
\begin{equation} \label{H2}
r|\nabla_{x,t} 
h(x,t)| \leq C_{h0} \ \text{ for } (x,t) \in \Omega_0.
\end{equation}
Here we allow $h$ to be a function of $(x,t)$, but for our final choice, $h$ will depend only on
$x$ and $r=|t|$. In later sections, some additional Carleson bounds on $h$ will be required, 
and then specific choices of $h$ will be taken, but not yet. 

For this section to work all the way to the injectivity of $\rho$, we will need to assume that 
$C_0$ and $C_{h0}$ are small enough, depending on $n$, $d$, our choice of bump function $\eta$,
and $C_{h1}$. We will denote by $C$ any constant that depends only on 
$n$, $d$, $\eta$, $C_{h1}$, and an upper bound for $C_0$ and $C_{h0}$.
This last dependance is not important, because anyway we shall rapidly assume that
$C_0 + C_{h0} \leq 1$, say.

Recall from \eqref{defrhointro} and \eqref{a3.15} that $\rho$ is given by
\begin{equation} \label{defrhobis}
\rho(x,t) = \Phi_r(x) +  h(x,t) R_{x,r}(0,t)
= (x,\varphi_r(x)) + h(x,t) \sum_{j=d+1}^n t_j w^j(x,r)
\end{equation}
for $(x,t) \in \Omega_0$, and where we systematically let $r = |t| > 0$
and $t = (t_{d+1},\dots,t_{n})$.
We will worry about the definition of $\rho$ on $\Gamma_0 = \R^d$
later; for the moment let us work on $\Omega_0$.

Because of Lemma \ref{viCinfty}, $\rho$ is smooth on $\Omega_0$, and \eqref{defrhobis}
gives the following formulas for the derivatives of $\rho$. 
For $i\in \{1,\dots,d\}$ and $(x,t) \in \Omega_0$,
\begin{equation} \label{a4.4}
\dr_{x_i} \rho(x,t) = \hat v^i(x,r) + \dr_{x_i}h(x,t) \sum_{j=d+1}^n t_j w^j(x,r) 
+  h(x,t) \sum_{j=d+1}^n t_j \dr_{x_i} w^j(x,r),
\end{equation}
(recall \eqref{defbarv}) and for $j\in \{d+1,\dots,n\}$ and $(x,t) \in \Omega_0$,
\begin{equation} \label{a4.5}
\begin{split}
\dr_{t_j} \rho = \frac{t_j}{r} \dr_r \Phi_r(x) 
& + h(x,t) w^j(x,r) \\
& +  \dr_{t_j} h(x,t)
\sum_{k=d+1}^n t_k w^k(x,r) 
+  \frac{t_j}{r} h(x,t) \sum_{k=d+1}^n t_k \dr_r w^k(x,r).
\end{split}
\end{equation}

The Jacobian matrix of $\rho$ is written $\Jac$. 
Note that $\Jac$ depends on $(x,t)\in \Omega_0$, but we shall not always write the argument.
With the notation of the beginning of Section \ref{Sorthbasis}, the coefficients of $\Jac$ are
given by 
\begin{equation} \label{a4.6}
\Jac_{k\ell} = \left<\dr_{x_k}\rho,e^\ell \right> 
\ \text{ for } 1\leq k \leq d \text{ and } 1\leq \ell \leq n,
\end{equation}
\begin{equation} \label{a4.7}
\Jac_{k\ell} = \left<\dr_{t_k}\rho,e^\ell \right> 
\ \text{ for } d+1\leq k \leq n \text{ and } 1\leq \ell \leq n.
\end{equation}

For the computations that follow, it will be useful to transform functions defined 
on $\R^d\times (0,+\infty)$ into functions defined in $\Omega_0$, and still give the 
same name to the new functions. That is, if $f$ is defined on $\R^d\times (0,+\infty)$,
we define $\wt f$ on $\Omega_0$ by $\wt f(x,t) = f(x,|t|)$, and then simply write
$f$ instead of $\wt f$, like physicists.

We shall use, as a first approximation of $\Jac$, the ``simpler'' matrix $J = J(x,t)$ where
in \eqref{a4.6} and \eqref{a4.7} we replace $e^\ell$ by $v^\ell$ when $\ell \leq d$,
and by $w^\ell$ when $\ell > d$. This gives the following matrix.

\begin{definition}[The matrix $J$] \label{defJ}
The matrix $J = (J_{k\ell})_{1\leq k,\ell \leq n}$ is defined on $\Omega_0$ as
\[J_{k\ell} := \left<\dr_{x_k}\rho,v^\ell\right> = \left< \hat v^k,v^\ell\right> + h \sum_{j=d+1}^n t_j \left<\dr_{x_k} w^j, v^\ell \right>\]
when $1\leq k \leq d$ and $1\leq \ell \leq d$, 
\[J_{k\ell} := \left<\dr_{x_k}\rho,w^\ell\right> = t_\ell \dr_{x_k} h + h \sum_{j=d+1}^n t_j \left<\dr_{x_k} w^j, w^\ell \right>\]
when $1\leq k \leq d$ and $d+1\leq \ell \leq n$, 
\[J_{k\ell} := \left<\dr_{t_k}\rho,v^\ell\right> =\frac{t_k}{r} \left<\dr_r \Phi_r 
,v^\ell\right> +   \frac{t_k}{r} h \sum_{j=d+1}^n t_j \left<\dr_r w^j,v^\ell\right>\]
when $d+1\leq k \leq n$ and $1\leq \ell \leq d$, and, setting $\delta_{k\ell} = 1$ if $k=\ell$ 
and $I_{k\ell} = 0$ otherwise,
\[J_{k\ell} := \left<\dr_{t_k}\rho,w^\ell\right> = h \, \delta_{k\ell} 
+ \frac{t_k}{r} \left<\dr_r \Phi_r ,w^\ell\right> + t_\ell \, \dr_{t_k} h 
+ \frac{t_k}{r} h \sum_{j=d+1}^n t_j \left<\dr_r w^j,w^\ell\right>,\]
when $d+1\leq k \leq n$ and $d+1\leq \ell \leq n$.

As previously, in the above expressions, $t=(t_{d+1},\dots,t_n) \in \Rt$ is the value of the second variable where $J$ is evaluated, and $r = |t|$.
\end{definition}

Let $Q = Q(x,t)$ denote the matrix of the change of basis from $\mathcal B_{vw}$ to 
$\mathcal B_e$; this is also the matrix of our isometry $R_{x,r}$ - in the sense that $R_{x,t}(y,u) = (y,u)Q(x,t)$ (recall that $(y,u)$ is an horizontal vector in $\R^n$) - and its coefficients are given by 
\begin{equation} \label{a4.9}
Q_{k\ell} 
= \left<e^k,v^\ell\right> 
\ \text{ for } 1\leq k \leq n \text{ and } 1\leq \ell \leq d,
\end{equation}
\begin{equation} \label{a4.10}
Q_{k\ell} 
= \left<e^k,w^\ell\right>
\ \text{ for } 1\leq k \leq n \text{ and } d+1\leq \ell \leq n.
\end{equation}
Observe that (just from this and the initial definition of the $J_{k\ell}$ as scalar products),
\begin{equation} \label{JisJacQ}
J = \Jac Q
\end{equation}

\medskip
We now decompose $J$ into a
sum of three matrices $J = J'+ H+ M$ where $J'$, $H$ and $M$ are defined 
as follows. All these matrices depend on $(x,t) \in \Omega_0$. 

\begin{definition}[The matrices $J'$, $H$ and $M$] \label{defJprimeandM}
Let $J'_1$ be 
the $d\times d$ matrix
with coefficients 
$(J'_1)_{k\ell } = \left<\hat v^k, v^\ell \right>$.
Then define the $n\times n$ matrix $J'$ by
\begin{equation} \label{a4.13}
J' = \begin{pmatrix} J'_1 & 0 \\ 0 & h\, I_{n-d}
\end{pmatrix}.
\end{equation}
Let $H$ be
the $n\times n$ matrix
with coefficients 
\renewcommand{\arraystretch}{1.5}
\begin{equation} \label{a4.14}
\begin{array}{l@{\hspace{2pt}}l@{\hspace{3pt}}l} 
H_{k\ell} &= 0  &\text{ for } \ell \leq d, \\
H_{k\ell} &= t_\ell \, \dr_{x_k} h \ &\text{ for } 1 \leq k \leq d < \ell \leq n, \\
H_{k\ell} &= t_\ell \, \dr_{t_k} h 
\ &\text{ for } k,\ell \in \llbracket d+1, n\rrbracket.
\end{array}
\end{equation}
\renewcommand{\arraystretch}{1}
Finally set $M= J-J'-H$, that is
\[
\begin{aligned}
M_{k\ell} &:= h \sum_{j=d+1}^n t_j \left<\dr_{x_k} w^j, v^\ell \right>
\qquad\text{ for $1\leq k \leq d$ and } 1\leq \ell \leq d, \\
M_{k\ell} &:= h \sum_{j=d+1}^n t_j \left<\dr_{x_k} w^j, w^\ell \right>
\qquad\text{ for $1\leq k \leq d$ and } d+1\leq \ell \leq n, \\
M_{k\ell} &:=\frac{t_k}{r} \left<\dr_r \Phi_r
,v^\ell\right> +   \frac{t_k}{r} h \sum_{j=d+1}^n t_j \left<\dr_r w^j,v^\ell\right>
\ \text{ for $d+1\leq k \leq n$ and } 1\leq \ell \leq d \\
M_{k\ell} &:= \frac{t_k}{r} \left<\dr_r \Phi_r 
,w^\ell\right>  + \frac{t_k}{r} h \sum_{j=d+1}^n t_j \left<\dr_r w^j,w^\ell\right>
\ \text{ for $d+1\leq k \leq n$ and } d+1\leq \ell \leq n.
\end{aligned}
\]
\end{definition}

In the decomposition $J = J'+H+M$, we will see that the main term is
$J'$, which is close to the identity, and $H+M$ is a small perturbation.
We start with the size of $M$ and recall our convention that in the estimates
that follow, $C$ is allowed to depend on $n$, $d$, $\eta$, $C_{h1}$,
and not $C_{h0}$ or $C_0$ as long as they stay bounded.

\begin{lemma} \label{defM}
There exists $C>0$ 
such that for 
$1\leq k,\ell\leq n$
\begin{equation} \label{a4.16}
|M_{k\ell}| \leq C \left( |\dr_r \varphi_r| + r|\nabla_{x,r}\nabla_x \varphi_r| \right).
\end{equation}
As previously, $r$ denotes the norm of the second variable where $M_{kl}$ is evaluated.
\end{lemma}

\begin{proof}
Recall that the vectors $v^i$, $1\leq i\leq d$, and the vectors $w^j$, $d+1\leq j\leq n$, 
are unit vectors. So with the definition of $M_{k\ell}$ given in Definition \ref{defJprimeandM}, 
one immediately gets that
\[|M_{k\ell}| \leq C \left( |\dr_r \varphi_r| + \sup_{d<j\leq n} rh |\nabla_{x,r} w^j| \right).\]
We now use the assumption \eqref{H1} on $h$ and Lemma \ref{vrbyphirr} to conclude.
\end{proof}

The coefficients of $M$ in Lemma \ref{defM} can be controlled by the Lipschitz constant $C_0$ 
with the help of the following result. 

\begin{lemma} \label{rgradientphilem}
There exists $C>0$ 
such that for 
$(x,r)\in \R^d \times (0,+\infty)$,
\begin{equation} \label{rgradientvestimate}
|\dr_r \varphi_r(x)| + r|\nabla_{x,r} \nabla_x \varphi_r(x)| \leq C C_0.
\end{equation}
\end{lemma}

\ms
\begin{proof}
The lemma will be proven as soon as we establish the bounds
\begin{equation} \label{rphirxbyC0}
r|\dr_{r}\dr_{x_k} \varphi_r(x)| \leq CC_0
\end{equation}
for $k \in \{1,\dots,d\}$, 
\begin{equation} \label{rphixxbyC0}
r|\dr_{x_\ell}\dr_{x_k} \varphi_r(x)| \leq CC_0
\end{equation}
for $\ell,k \in \{1,\dots,d\}$, and 
\begin{equation} \label{phirbyC0}
|\dr_{r} \varphi_r(x)| \leq CC_0.
\end{equation}

Let $k\in \{1,\dots,d\}$. One has the relation 
\begin{equation} \label{rphirx=cv}
r\dr_{r}\dr_{x_k} \varphi_r(x) = (r\dr_r  \eta_r) * (\dr_{x_k} \varphi)(x),
\end{equation}
 and $r\dr_r  \eta_r$ can be rewritten as
\begin{equation} \label{retar=eta2}
r\dr_r \eta_r(x) = r\dr_r \left[\frac1{r^d} \eta\left( \frac xr\right)\right] 
= - \frac{1}{r^d} \left[ d \eta\left( \frac xr\right) 
+ \frac{x}r \cdot \nabla \eta\left( \frac xr\right) \right] = \hat \eta_r(x)
\end{equation}
where $\hat \eta(x) = -d\eta(x) - x\cdot \nabla \eta(x)$ and 
$\hat \eta_r(x) = r^{-d} \hat\eta(\frac xr)$. Hence
\[\begin{split}
|r\dr_{r}\dr_{x_k} \varphi_r(x)| & = |\hat \eta_r * (\dr_{x_k} \varphi)(x)| \leq \sup_{y\in \R^d} |\dr_{x_k} \varphi(y)| \int_{\R^d} |\hat \eta_r(y)| dy \\
& \leq C_0 \int_{\R^d} |\hat \eta(y)| dy = C_\eta C_0.
\end{split}\]
The inequality \eqref{rphirxbyC0} follows.
Next, let $1\leq i,k\leq d$. In the same way as before,
\[|r\dr_{x_i}\dr_{x_k} \varphi_r(x)| 
\leq  |(r\dr_{x_i} \eta_r)*(\dr_{x_k} \varphi)(x)| \leq C_0 \int_{\R^n} |r\dr_{x_i} \eta_r|.\]
Yet, for any $r>0$, by a simple change of variable, 
\[\int_{\R^d} |r\dr_{x_i} \eta_r| = \int_{\R^d} |(\dr_{x_i} \eta)_r| 
= \int_{\R^d} |\dr_{x_i} \eta| 
= C_\eta.\]
The inequality \eqref{rphixxbyC0} follows.
Now we prove
\eqref{phirbyC0}. We have
\[\dr_r \varphi_r(x) = (\dr_r \eta_r) * \varphi(x),\]
and, with the help of \eqref{retar=eta2}, $\dr_r \eta_r$ can be rewritten as
\[\dr_r \eta_r(x) = - \frac{1}{r^{d+1}} \left[ d \eta\left( \frac xr\right) 
+ \frac{x}r \cdot \nabla \eta\left( \frac xr\right) \right] = \diver_x \wt \eta_r(x)\]
where $\wt \eta(x) = -x \eta\left( x \right)$ and $\wt\eta_r(x) = r^{-d} \wt \eta(\frac xr)$. 
As a consequence, if $\varphi = (\varphi^{d+1},\dots,\varphi^n)$, 
we obtain that for $j\in \{d+1,n\}$,
\[\begin{split}
\dr_r \varphi_r^j(x) & = (\diver_x \wt \eta_r) * \varphi^j(x) 
= \int_{\R^d} \varphi^j(x-y)  \diver_x \wt \eta_r(y) dy \\
& = \int_{\R^d}  \left< \wt \eta_r(y), \nabla_x \varphi^j(x-y)\right> dy 
\end{split}\]
and then
\[\begin{split}
|\dr_r \varphi_r(x)| & \leq \sup_{y\in \R^d} |\nabla_x \varphi(y)| \int_{\R^d} |\wt \eta_r (y)| dy 
= C_0 \int_{\R^d} |\wt \eta (y)| dy  = C_\eta C_0.
\end{split}\]
Lemma \ref{rgradientphilem} follows.
\end{proof}

The combination of Lemma \ref{defM} and Lemma \ref{rgradientphilem} gives that 
\begin{equation} \label{MbyC0}
|M_{k\ell}(x,t)| \leq C C_0
\end{equation}
whenever $1\leq k,\ell \leq n$ and $(x,t) \in \Omega_0$. 
Here and in the next estimates, the constant $C$ in \eqref{MbyC0} depends on $C_0$, 
but this is all right because $C$ stays bounded when $C_0 \leq 1$.

The coefficients for $H$ are easier to control, since \eqref{a4.14} and \eqref{H2} yield
\begin{equation} \label{a4.25}
|H_{k\ell}| \leq C_{h0}
\ \text{ for } 1 \leq k, \ell \leq n.
\end{equation}

\medskip

We shall now estimate the determinant of $\Jac$ to prove the local invertibility of $\rho$.
We start with $J$ and $J'$.

\begin{lemma} \label{Jacobianneq0}
Recall that $J'$ is defined in Definition \ref{defJprimeandM}.  Then
\begin{equation} \label{detJ-detJp}
|\det(J) - \det(J')| 
\leq C  \left( |\dr_r \varphi_r| + r|\nabla_{x,r} \nabla_x \varphi_r| 
+ r|\nabla_{x,t} h| \right)
\end{equation}
and 
\begin{equation} \label{detJp-1}
\sup_{1\leq i \leq d} \left| \left< \hat v^i,v^i \right> -1 \right| + |\det(J')- h^{n-d}| \leq CC_0^2.
\end{equation}

As a consequence, for any $\epsilon>0$, there exists $c_\epsilon>0$, depending only on 
$C_{h1}$, $\eta$ and $\epsilon$, such that if  $C_0+C_{h0}\leq c_\epsilon$, then
\begin{equation} \label{detJsim1}
(1-\epsilon) h^{n-d} \leq \det(J) \leq (1+\epsilon) h^{n-d}
\end{equation}
and
\begin{equation} \label{barvvsim1}
\sup_{1\leq i \leq d} \left| \left< \hat v^i,v^i \right> -1 \right| \leq \epsilon.
\end{equation}
\end{lemma}

\begin{proof}
We will use the matrices $J'$, $H$ and $M$ introduced in Definition \ref{defJprimeandM}. We have 
\[\det(J) = \det(J'+ H + M) = \sum_{\sigma \in \mathfrak S_n} \sgn(\sigma) \prod_{i=1}^n (J'_{i,\sigma(i)} + H_{i,\sigma(i)} + M_{i,\sigma(i)}).\]
We develop the above formula and we decompose the sum into two parts: the terms that are products of coefficients of $J'$ and the terms that contain at least one coefficient of $M$ or $H$ and we have

\begin{eqnarray} 
|\det(J)-\det(J')| & = & \left|\sum_{\sigma \in \mathfrak S_n} \sgn(\sigma) 
\prod_{i=1}^n (J'_{i,\sigma(i)} + H_{i,\sigma(i)} + M_{i,\sigma(i)}) 
- \sum_{\sigma \in \mathfrak S_n} \sgn(\sigma) \prod_{i=1}^n J'_{i,\sigma(i)} \right| 
\nn\\
&\leq& C \left(\sup_{1\leq k,\ell\leq n} |M_{k\ell}| + |H_{k\ell}| \right) \left(\sup_{1\leq k,\ell\leq n}  |J'_{k\ell}| + |H_{k\ell}|+ |M_{k\ell}| \right)^{n-1}.
\end{eqnarray}
Recall from \eqref{MbyC0} and \eqref{a4.25} that $|M_{k\ell}| + |H_{k \ell}|\leq CC_0+ C_{h0} \leq C$.
Moreover,  
\[  \sup_{1\leq k,l\leq n} |J'_{kl}| 
\leq  \sup_{1\leq i\leq d} |\hat v^i| \leq \sqrt{1+C_0^2} \leq C\]
by Definition \ref{defJprimeandM} and \eqref{a3.8}.
These two latter facts prove that 
\[\Big(\sup_{1\leq k,\ell\leq n}  |J'_{k\ell}| + |H_{k\ell}|+ |M_{k\ell}| \Big)^{n-1} \leq C, \]
and hence
\[|\det(J)-\det(J')| \leq C \sup_{1\leq k,\ell\leq n} |M_{k\ell}| + |H_{k\ell}|. \]
Lemma \ref{defM} and the definition of $H$ allows us to conclude
\[|\det(J)-\det(J')| \leq C \left( |\dr_r \varphi_r| + r|\nabla_{x,r} \nabla_x \varphi_r| 
+ r|\nabla_{x,t} h| \right), \] 
which is exactly \eqref{detJ-detJp}.

\medskip
Let us turn to the proof of \eqref{detJp-1}.
Let $i\in \{1,\dots,d\}$. Since $|v^i|= 1$, \eqref{a3.8} yields
\begin{equation} \label{UBforbarvivi} \left|\left<\hat v^i,v^i\right>\right| \leq |\hat v^i| \leq \sqrt{1+C_0^2}.\end{equation}
We also want a a lower bound on $\left<\hat v^i,v^i\right>$. Since 
$v^{i} = |\wt v^{i}|^{-1} \ \wt v^{i}$ (by \eqref{a3.10} for $i > 1$ and 
by convention for $i=1$),
\begin{equation} \label{LBforbarvivi}
\left<\hat v^i,v^i\right> = |\wt v^i|^{-1} \, \left<\wt v^i,\hat v^i \right>
\end{equation}
so we want  
a lower bound for $\left<\wt v^i,\hat v^i \right>$ and an upper bound for $|\wt v^i|$. 
The latter is
\[|\wt v^i|^2 = |\hat v^i|^2 - \sum_{\ell=1}^{i-1} |\left< \hat v^i,v^\ell \right>|^2 \leq |\hat v^i|^2 \leq 1+C_0^2,\]
which follows from \eqref{a3.9} because
$\{v^\ell\}_{1\leq \ell \leq i-1}$ is an orthonormal family of vectors. 
Moreover, 
$\{v^1,\dots,v^{i-1},e^i\}$ is also an orthonormal family 
(because all the $\hat v^\ell$, $\ell < i$, are orthogonal to $e_i$), and so 
\[|\hat v^i|^2 \geq \left< \hat v^i,e^i\right>^2 + \sum_{k=1}^{i-1} \left< \hat v^i,v^k\right>^2\]
and thus by \eqref{a3.9} again
\[\begin{split}
\left<\wt v^i,\hat v^i \right> & = |\hat v^i|^2 - \sum_{k=1}^{i-1} \left< \hat v^i,v^k\right>^2 \geq \left< \hat v^i,e^i\right>^2 + \sum_{k=1}^{i-1} \left< \hat v^i,v^k\right>^2 - \sum_{k=1}^{i-1} \left< \hat v^i,v^k\right>^2 = 1.
\end{split}\]
The previous inequalities prove that 
\[\frac{1}{\sqrt{1+C_0^2}}\leq \left<\hat v^i,v^i\right> \leq \sqrt{1+C_0^2}\]
and hence, subtracting $1$, 
\begin{equation} \label{barvv-1}
 -\frac{C_0^2}{2} \leq \frac{1}{\sqrt{1+C_0^2}}-1 \leq \left<\hat v^i,v^i\right> - 1 \leq \sqrt{1+C_0^2} - 1 \leq \frac{C_0^2}{2}.\end{equation}
This proves the first half of \eqref{detJp-1}, which will now
be proved fully as soon as we show that $|\det(J')-h^{n-d}|
\leq CC_0^2$. 

By \eqref{a4.13}, $\det(J') =h^{n-d} \det(J'_1)$. Then recall that
$(J'_1)_{k\ell} = \left<\hat v^k,v^\ell\right>$ by Definition \ref{defJprimeandM}.
By construction of $v^\ell$ (see \eqref{a3.9}), this vanishes when $1\leq k<\ell\leq d$.
Therefore, $J'_1$ is a lower triangular matrix and 
\begin{equation} \label{a4.35}
\det(J') =h^{n-d} \det(J'_1)= h^{n-d} \prod_{i=1}^d \left<\hat v^i,v^i\right>.
\end{equation}
Then by \eqref{barvv-1}
\begin{eqnarray} \label{a4.36}
|\det(J')-h^{n-d}| &=& h^{n-d}\left|1 - \prod_{i=1}^d \left<\hat v^i,v^i\right>\right| 
\nn\\
&\leq& C h^{n-d} \left(\sup_{1\leq i \leq d} |\left<\hat v^i,v^i\right> - 1|\right) 
\left(\sup_{1\leq i \leq d} 1+ |\left<\hat v^i,v^i\right>|\right)^{d-1} \leq CC_0^2.
\end{eqnarray}
This completes the proof of \eqref{detJp-1}.
It remains to prove the second part of the lemma. 
The estimate \eqref{barvvsim1} is immediate from \eqref{detJp-1}.
Next 
\[|\det(J) - \det(J')| 
\leq C  \left( |\dr_r \varphi_r| + r|\nabla_{x,r} \nabla_x \varphi_r| + r|\nabla_{x,t} 
h| \right)
\leq C(C_0+C_{h0})\]
by \eqref{detJ-detJp}, Lemma~\ref{rgradientphilem}, and \eqref{H2}; and
where $C>0$ is independent of $C_0$ and $C_{h_0}$ provided that, for instance, 
$C_0 + C_{h0}\leq 1$. Together with\eqref{detJp-1}, this yields
\[|\det(J) - h^{n-d}| \leq C(C_0+C_{h0}) \leq Ch^{n-d} (C_0+C_{h0}),\]
whenever $C_0+C_{h0}\leq 1$. The bound \eqref{detJsim1} follows.
 \end{proof}
 
 \ms
 From now on we assume that $C_0+C_{h0}$ is small enough (depending on 
 $n$, $d$, $\eta$, and $C_{h1}$), so that 
\eqref{barvvsim1} and \eqref{detJsim1} hold for some $\epsilon <1$.
 Similar additional conditions will be given soon.
 
Already with these conditions, \eqref{detJsim1} says that $\det(J(x,t)) \neq 0$
on $\Omega_0$. Then, since \eqref{JisJacQ} says that $J = \Jac Q$ for some orthogonal matrix $Q$,
\begin{equation} \label{Jacneq0}
\det(\Jac(x,t)) = \det(J(x,t)) \neq 0.
\end{equation}
That is, $\Jac(x,t)$ is invertible. We also know from Lemma \ref{viCinfty} and the continuous differentiability of $h$ that
\begin{equation} \label{a4.38}
\rho \text{ is continuously differentiable in } \Omega_0, 
\end{equation}
so the inverse function theorem says that there is a small neighborhood of $(x,t)$
where $\rho$ is a diffeomorphism. In particular, $\rho$ is an open mapping.
In the next result, we prove - if $C_0+C_{h0}$ is small - that 
\begin{equation} \label{a4.39}
\rho(\Omega_0) \subset \Omega = \R^n \sm \Gamma.
\end{equation}
Recall that $\Phi : \R^d \to \Gamma$ is the Lipschitz function defined by $\Phi(x) := (x,\varphi(x))$.

\begin{lemma} \label{controldistprop}
Let $\epsilon \in (0,1)$ be given.
If $C_0+C_{h0}$ is small enough (depending on $\epsilon$, $n$, $d$, $\eta$ and $C_{h1}$),
then for any $(x,t)\in \Omega_0$, 
\begin{equation} \label{controlofdist}
(1-\epsilon) r  h(x,t)  \leq |\rho(x,t)-\Phi(x)| \leq (1+\epsilon)r  h(x,t) 
\end{equation}
and
\begin{equation} \label{controlofdist2}
\dist(\rho(x,t),\Gamma) \geq (1-\epsilon) r  h(x,t) > 0,
\end{equation}
where as usual $r=|t|$.
\end{lemma}

\begin{proof}
Notice before we start that \eqref{controlofdist2} implies \eqref{a4.39}.
Now let $(x,t)\in \Omega_0$ be given, and write $t=(t_{d+1},\dots,t_n)$ as usual.
By \eqref{defrhobis},
\begin{equation} \label{defrho2}
\rho(x,t) - \Phi(x) = \Phi_r(x) - \Phi(x) 
+ h(x,t) R_{x,r}(0,t).  
\end{equation}
Notice that since $\Phi_r(x) - \Phi(x) = (0, \varphi_r(x)-\varphi(x))$, 
\begin{equation} \label{calculusCeta}\begin{split}
|\Phi_r(x) - \Phi(x)| & =
|\varphi_r(x) - \varphi(x)|  \leq  \int_{\R^d} \eta_r(y) |\varphi(x-y)-\varphi(x)| dy \\
& \leq C_0 \int_{\R^d} |y|\eta_r(y) dy \leq C_0 r \int_{\R^d} \eta_r (y) dy = C_0 r
\end{split}\end{equation}
because $\eta_r$ is supported in $B(0,r)$ and $\int \eta_r (y) dy = 1$.

Recall that $R_{x,r}$ is an isometry, so $|h(x,t) R_{x,r}(0,t)| = r h(x,t)$, and by \eqref{defrho2}
\begin{equation} \label{a4.45}
\big| |\rho(x,t) - \Phi(x)|- r h(x,t) \big| \leq |\Phi_r(x) - \Phi(x)| \leq C_0 r.
\end{equation}
Thus \eqref{controlofdist} as soon as $C_0 < C_{h1}^{-1} \epsilon$ in \eqref{a4.45}.

It remains to prove  
\eqref{controlofdist2}, which is stronger than the lower bound 
in \eqref{controlofdist}. We take any point $\Phi(y)$, $y\in \R^d$, of $\Gamma$, and we want
to show that 
\begin{equation} \label{a4.46}
|\rho(x,t) - \Phi(y)| \geq (1-\epsilon) r  h(x,t) > 0.
\end{equation}
The idea is roughly as follows.
We know that $h(x,t) R_{x,r}(0,t)$ lies in $P'(x,r)^\perp$ and that $P'(x,r)$ is almost parallel to 
$\R^d$, so $h(x,t) R_{x,r}(0,t)$ should be almost orthogonal to $\R^d$. On the other
hand, $\Phi(x)-\Phi(y)$ is almost parallel to $\R^d$, so the sum
$h(x,t) R_{x,r}(t) + \Phi(x)-\Phi(y)$ should be larger or almost as large as $h(x,t) R_{x,r}(t)$.
Finally, we add $\Phi_r(x) - \Phi(x)$ which we know is much smaller, and we should get \eqref{a4.46}.

We first estimate the direction of the $w^j$, $j \geq d+1$. Observe that 
$\left< \wt w^j,e^j\right> = \left< \hat w^j,e^j\right> = 1$ by the proof of \eqref{a3.11}
(but with the $w^j$). Besides, $|\wt w^j| \leq |\hat w^j|$ because $\wt w^j$
is an orthogonal projection of $\hat w^j$ (see \eqref{a3.13}), and 
$|\hat w^j| \leq (1+C_0^2)^{1/2}$ by \eqref{a3.8}. Hence
\begin{equation} \label{wjejbfb}
\left< w^j,e^j\right> = |\wt w^j|^{-1} \left< \wt w^j,e^j\right> = |\wt w^j|^{-1}
\geq (1+C_0^2)^{-1/2}
\end{equation}
by \eqref{a3.14}. Hence, since $w^j$ and $e^j$ are both unit vectors,
\[|w^j-e^j|^2 = 2\left(1- \left< w^j,e^j\right>\right) 
\leq 2\left(1- (1+C_0^2)^{-1/2} \right) \leq C_0^2.\]
Since $R_{x,r}(0,t) = \sum_{j=d+1}^n t_j w^j$, 
\begin{equation} \label{wt-et}
\left| R_{x,r}(0,t) - \sum_{j=d+1}^n t_j e^j \right| 
\leq \left( \sum_{j=d+1}^n |t_j|^2 \right)^\frac12 
\left( \sum_{j=d+1}^n |w^j-e^j|^2 \right)^\frac12 \leq C_0 r.
\end{equation}
Set $e(t) = \sum_{j=d+1}^n t^j e^j$. The definition \eqref{defrho2} gives 
\begin{equation} \label{controlofdist3} \begin{split}
|\rho(x,t)-\Phi(x) - h(x,t) e(t)| 
&\leq |\Phi_r(x)-\Phi(x)| + h(x,t) |R_{x,r}(0,t) - e(t)|  \\
&\leq  C_0 r + C_0 r h(x,t) \leq (1+ C_{h1}) C_0r
\end{split} \end{equation}
by \eqref{calculusCeta}, \eqref{wt-et}, and \eqref{H1}.

Now we control the vertical part of $\Phi(y)-\Phi(x)$, i.e., $\varphi(y)-\varphi(x)$. 
Denote by $\pi^\perp$ the orthogonal projection on $(\R^d)^\perp$. Then 
$|\pi^\perp(\Phi(y)-\Phi(x))| \leq C_0 |y-x|$, just because $\varphi$ is $C_0$-Lipschitz.
If $|y-x| \leq 2C_{h1}r$, then 
\begin{eqnarray} \label{a4.50}
|\rho(x,t) - \Phi(y)| &\geq& |\pi^\perp(\rho(x,t) - \Phi(y))|
\nn\\
&\geq& |\pi^\perp(h(x,t) e(t))| - |\pi^\perp(\rho(x,t) - h(x,t) e(t) - \Phi(y))|
\nn\\
&=& h(x,t) - |\pi^\perp(\rho(x,t) - h(x,t) e(t) - \Phi(y))|
\nn\\
&\geq& h(x,t) - |\rho(x,t)-\Phi(x) - h(x,t) e(t)| - |\pi^\perp(\Phi(y)-\Phi(x))|
\\
&\geq& h(x,t) - (1+ C_{h1}) C_0 r - C_0 |y-x| \geq h(x,t) - CC_0 r,
\nn
\end{eqnarray}
which implies \eqref{a4.46} when $C_0$ is small enough.
If instead $|y-x| \geq 2C_{h1}r$, then
\begin{eqnarray} \label{a4.51}
|\rho(x,t) - \Phi(y)| &\geq& |\Phi(x) - \Phi(y)| - |\rho(x,t) - \Phi(x)|
\nn\\
&\geq& |x-y| - (1+\epsilon) r h(x,t)
\geq 2C_{h1} r - (1+\epsilon) r h(x,t)  
\nn\\
&\geq& (2C_{h1} - (1+\epsilon)h(x,t)) r 
\geq (1-\epsilon)C_{h1}r \geq (1-\epsilon)h(x,t) r
\end{eqnarray}
by \eqref{controlofdist} and \eqref{H1}.
This establishes \eqref{a4.46} in our second case;
\eqref{controlofdist2} and Lemma \ref{controldistprop} follow.
\end{proof}

\ms 
Lemma \ref{controldistprop} is useful, in particular because it allows us to control the inverse 
images of sets. Suppose from now on that $C_0$ is chosen so small that the conclusion of
the lemma holds with $\epsilon = 1/2$. Then for $(x,t) \in \Omega_0$,
\begin{equation} \label{a4.52}
\dist(\rho(x,t), \Gamma) \leq 
|\rho(x,t) - \Phi(x)| \leq 2 r h(x,t) \leq 2 C_{h1} r \end{equation}
and 
\begin{equation} \label{a4.53}
\dist(\rho(x,t), \Gamma) \geq \frac{1}{2} \, r h(x,t) \geq \frac{r}{2 C_{h1}}.
\end{equation}
Thus, if $Z = \rho(x,t)$ 
for some $(x,t) \in \Omega_0$, we get that 
\begin{equation} \label{a4.54}
(2 C_{h1})^{-1} \dist(Z,\Gamma) \leq |t| \leq 2 C_{h1} \dist(Z,\Gamma)
\end{equation}
and then, writing $Z = (y,s) \in \R^{d} \times \R^{n-d}$, 
\begin{eqnarray} \label{a4.55}
|x-y| &\leq& |(x,\varphi(x)) - (y,s)| = |\Phi(x)-\rho(x,t)|
\leq 2 C_{h1} |t| 
\end{eqnarray}
by \eqref{a4.52}.

Thus, given $Z = \rho(x,t)$ we get a good idea of where $(x,t)$ lies. In fancy terms,
$\rho$ is proper, which means that for every compact set $K \subset \R^n \sm \Gamma$,
\begin{equation} \label{a4.56}
\rho^{-1}(K) = \big\{ (x,t) \in \Omega_0 \, ; \, \rho(x,t) \in K \big\}
\ \text{ is a compact subset of } \Omega_0.
\end{equation}
Indeed, the estimates above imply that $\rho^{-1}(K)$ lies in a bounded subset of 
$\R^n$, and at distance at least $(2 C_{h1})^{-1} \dist(K,\Gamma) > 0$
from $\R^d$. Thus $\rho^{-1}(K)$ is relatively compact in $\Omega_0$,
and the fact that it is compact (or closed) follows because $\rho$ is continuous.
We are ready for the fun part of the argument.

\begin{theorem} \label{Tbijective}
Assume that $C_0$ and $C_{h0}$ are small enough, as before. Then
$\rho$ is a bijection from $\Omega_0$ to $\Omega = \R^n \sm \Gamma$.
Since it is also continuously differentiable and its Jacobian is invertible at every point
(by \eqref{Jacneq0}) of $\Omega$ it is also a diffeomorphism.
\end{theorem}

\noindent{\it Proof.} 
We will use a little bit of topology but, even though we are thinking about the degree
of our function $\rho$, we shall not need more than connectedness and the inverse function
theorem. 
Also, to make the argument more pleasant to read, elements of $\Omega$ or $\Omega_0$
will be called $z$ or $y$, rather than $Z$ or $Y$ with our earlier convention.

Set $\rho^{-1}(z) = \big\{ y\in \Omega_0 \, ; \, \rho_{\varphi,h}(y) = z \big\}$ 
for $z \in \Omega$. 
We introduce the function $N$ defined on $\Omega$ by
\begin{equation} \label{defNphi}
N(z) = \left\{\begin{array}{ll} 
0 & \text{ if $\rho^{-1}(z)$ is empty,} \\
1 & \text{ if $\rho^{-1}(z)$ contains exactly one point,} \\
2 & \text{ if $\rho^{-1}(z)$ contains at least two points.}
\end{array}\right.
\end{equation}
The quantity $N$ is clearly inspired by the notion of degree of a map, but
we shall not need to know more about that notion.
We aim to prove that $N$ is constant equal to 1, which is equivalent to the fact that 
$\rho$ is bijective, and we shall proceed by (similar) steps.

The general idea is to show that $N$ is constant, and later on we will compute its value. 
Consider the set 
\begin{equation} \label{a4.59}
R_2 : = \{z \in \Omega, \, N(z) =2\};
\end{equation}
we want to show that $R_2$ is both open and closed in $\Omega$. Since $\Omega$
is connected (we are in codimension $>1$), we will deduce that 
\begin{equation} \label{a4.60}
\text{$R_2 = \emptyset$ or $R_2 = \Omega$.} 
\end{equation}

First, $R_2$ is open. Indeed choose $z\in R_2$ and take $y_1,y_2\in \Omega_0$, 
with $y_1 \neq y_2$, such that $\rho(y_1) = \rho(y_2) = z$. 
Define $B_1 = B(y_1,|y_1-y_2|/3)$ and $B_2 = B(y_2,|y_1-y_2|/3)$. 
Since both $\Jac(y_1)$ and $\Jac(y_2)$ are invertible (see \eqref{Jacneq0}), 
the inverse function theorem proves that there exist neighborhoods 
$U_1 \subset B_1 \cap \Omega_0$, $U_2 \subset B_2 \cap \Omega_0$ and 
$V_1, V_2 \subset \Omega$ of  $y_1$, $y_2$, and $z$ respectively, 
such that $\rho$  is a bijection
from $U_1$ to $V_1$ and from $U_2$ to $V_2$. Since 
$U_1 \cap U_2 \subset B_1 \cap B_2 = \emptyset$ by construction, 
we deduce $N(w) = 2$ for all $w\in V_1 \cap V_2$, hence 
$V_1 \cap V_2 \subset R_2$. 

The set $R_2$ is closed. Indeed, choose a sequence $(z^j)_{j\in \bN}$ of values in $R_2$ 
that converges to some $z\in \Omega$. Take
two sequences $(y^j_1)_{j\in \bN}$ and $(y^j_2)_{j\in \bN}$ 
in $\Omega_0$, such that $y^j_1 \neq y^j_2$ and 
$\rho(y^j_1) = \rho(y^j_2) = z^j$ for each $j$.
Set $K = \{z^j\}_{j\in \bN} \cup \{z\}$; this is compact set in $\Omega$,
so by \eqref{a4.56} $K_0 = \rho^{-1}(K)$ is a compact subset of $\Omega_0$. 
The $y^j_1$ lie in $K_0$, so there is a subsequence of $(y^j_1)$ 
(that we still will denote by $(y^j_1)$) that converges to some limit $y_1 \in K_0 \subset \Omega_0$.
Since $\rho$ is continuous and the $\rho(y^j_1) = z^j$ converge to $z$, $\rho(y_1) = z$. 
Let us extract a further subsequence (still denoted the same way), so that
the $y^j_2$ converge to some $y_2 \in K_0$. Observe that $\rho(y_2) = z$
too.

If $y_1 \neq y_2$, $z \in R_2$ and we are happy. Otherwise, observe that since
$\Jac(y_1)$ is invertible, the inverse function theorem shows that there is a small ball
centered on $y_1$ where $\rho$ is injective. This contradicts the fact
that $y^j_1 \neq y^j_2$ and both sequences converge to $y_1 = y_2$.
So $R_2$ is both open and closed in $\Omega$ and \eqref{a4.60} holds.

\medskip
Next we want to show that 
\begin{equation} \label{a4.61}
R_1 : = \{z \in \Omega, \, N(z) =1 \}
\ \text{ is either empty or equal to } \Omega.
\end{equation}
Since this is trivial if $R_2 = \Omega_0$, we may assume that $R_2 = \emptyset$.

We proceed as for $R_2$. First observe that $R_1$ is open, because if
$z\in R_1$ and $y \in \rho^{-1}(z)$, then $\Jac(y)$ is invertible,
so we can apply the inverse function theorem near $y$ and find solutions
of $\rho^{-1}(x) = z'$ for $z'$ near $z$. Since none of these $z'$
lies in $R_2$ (which is empty), they lie in $R_1$, as needed.

The set $R_2$ is closed. Indeed, let $z^j = \rho(y^j)$ be a sequence 
in $R_1$ that converges to some $z \in \Omega$. As before, 
$K = \{z^j\}_{j\in \bN} \cup \{z\}$ is compact set in $\Omega$,
$K_0 = \rho^{-1}(K)$ is a compact subset of $\Omega_0$, 
there is a subsequence of $(y^j_1)$ that converges to some limit $y_1 \in K_0 \subset \Omega_0$,
and  since $\rho$ is continuous, $\rho(y_1) = z$. 
Thus $z \in \rho(\Omega_0)$, hence $z\in R_1$ (because $R_2 = \emptyset$),
$R_1$ is closed, and since $\Omega$ is connected, \eqref{a4.61} holds.

Of course, if both $R_1$ and $R_2$ are empty, then 
$R_0: = \{z\in \Omega, \, N(z) = 0\} = \Omega$. 
So we proved that
\begin{equation} \label{a4.62}
N \text{ is constant on } \Omega. 
\end{equation}
We still need to compute its value; obviously Theorem \ref{Tbijective} will follow
as soon as we prove the next lemma.

\begin{lemma} \label{onetooneprop}
Let $\varphi$, $h$, and $\rho$ be as in Theorem \ref{Tbijective}.
Then $N_{\varphi,h}$ is constant equal to 1.
\end{lemma}

\begin{proof}
We decided to put this as a separate statement to cut the proof, 
and because the argument is of a different nature. We are now going to play
with the way $N$ depends on the functions $\varphi$ and $h$, so
we now write $\rho_{\varphi, h}$ for the mapping $\rho$ that was constructed above,
and similarly denote by $N_{\varphi,h}(z)$ the counting function $N$ above.
Also denote by $N(\varphi,h)$ denote the constant value of $N_{\varphi,h}(z)$ on $\Omega$.
With all this notation, we want to prove that $N(\varphi,h)=1$.

\ms
\noindent {\bf Step 1.}
We claim that if the pairs $(\varphi_1,h_1)$ and $(\varphi_2, h_2)$ both satisfy the assumptions of the theorem 
(with the same $C_{h1}$ in particular) and if $\varphi_1 = \varphi_2$ on 
$B_0 = B(0, 20 C_{h1}^2) \subset \R^d$ and  $h_1 = h_2$ on $B'_0 = B_0 \times B(0,20C_{h1}^2) \subset \R^n$,  then $N(\varphi_1,h_1) = N(\varphi_2,h_2)$.

This will be convenient, to replace $(\varphi,h)$ by a pair for which we can compute
the degree $N$ more easily, but let us first prove our claim.
Let $(\varphi_1,h_1)$ and $(\varphi_2, h_2)$ be as in the claim. For $i=1,2$, denote
by $\Gamma_i$ the graph of $\varphi_i$, and set $\rho_i = \rho_{\varphi_i, h_i}$.
Consider the point $z = (0,s)$, where $s \in \R^{n-d}$ is chosen such that $|s-\varphi_1(0)| = 1$.
By assumption, both functions $\varphi_i$ are $1$-Lipschitz, so 
$\frac12 \leq \dist(z,\Gamma_i) \leq 1$.

Next let $(x_i,t_i) \in \Omega_0$ be any solution of $\rho_i(x_i,t_i) = z$. 
We know from \eqref{a4.54} that $|t_i| \leq 2C_{h1} \dist(z,\Gamma_i) \leq 2C_{h1}$,
and then \eqref{a4.55} says that $|x_i| \leq  2  C_{h1} |t_i| \leq 4 C_{h1}^2$. 
Now the value of $\rho_i(x_i,t_i)$ depends only on the values of $\varphi_i$ and $h_i$
in $B(x_i,t_i)$ (check with \eqref{defrhobis}, and recall that $\eta$ is supported in the unit ball).
Since these functions coincide on $B_0$ and $B'_0$ respectively, 
 $\rho_1(x_i,t_i) = \rho_2(x_i,t_i) = z$.
In other words, the equations $\rho_i(x,t) = z$ have the same solutions, 
$N(\varphi_1,h_1) = N_{\varphi_1,h_1}(z) = N_{\varphi_2,h_2}(z) = N(\varphi_2,h_2)$,
and our claim follows.

In fact, by translation invariance (or with the same proof), the claim is still valid if
$B_0$ is replaced with any ball of radius $20 C_{h1}^2$.

\ms
\noindent {\bf Step 2.}
We modify $\varphi$ in a faraway ball to make it simpler.
Let $\varphi$ and $h$ be as in the theorem. We can find $\varphi_1$, which 
is also $C_0$-Lipschitz, coincides with $\varphi$ in $B_0$, while 
$\varphi_1 = \varphi(0)$ in some other ball $B_1 = B(x_1,20 C_{h1}^2)$. 

Indeed, choose $x_1\in \R^d$ at distance $100 C_{h1}^2$ from the origin,
and decide that $\varphi_1 = \varphi$ on $B_0$ and $\varphi_1 = \varphi(0)$
on $B_1$; it is easy to check that $\varphi_1$ is $C_0$-Lipschitz on $B_0 \cup B_1$,
because 
$$|\varphi_1(y)-\varphi_1(z)| = |\varphi(y)-\varphi(0)| \leq 20 C_{h1}^2C_0 \leq |z-y| C_0$$
whenever  $y \in B_0$ and $z\in B_1$; then we can use
the Kirszbraun extension theorem (see \cite[Theorem 2.10.43]{Federer}) to define
$\varphi_1$ on the rest of $\R^d$. Of course, we do not need to be that fancy; a
Whitney-type extension theorem, or just setting 
$\varphi_1(x) = \psi(x) \varphi(x) + (1-\psi(x)) \varphi(x_0)$, with a smooth radial bump
function $\psi$ such that $\psi = 1$ on a neighborhood of $B_0$ and $\psi = 0$
on $\R^d \sm 3B_0$ would work as well, except that maybe $\varphi_1$ is
$2C_0$-Lipschitz. We can fix this problem by requiring in advance that $C_0$ to be 
twice smaller in the statement of the theorem. 
By Step 1, $N(\varphi,h) = N(\varphi_1,h)$. In addition, let $\varphi_2$ denote
the constant function $\varphi(0)$; then $\varphi_2 = \varphi_1$ on $B_1$,
so $N(\varphi_1,h) = N(\varphi_2,h)$ by the small extension of Case 1.
Thus it is enough to check that $N(\varphi_2,h) = 1$ when our function $\varphi_2$ is constant.

\ms
\noindent {\bf Step 3.}
At this point we could actually compute because \eqref{defrhobis} becomes much simpler,
but let us cheat again and modify $h$ now. Set $R = 100 C_{h1}$ and let $\tau > 0$
be small, to be chosen soon. Then find a smooth compactly supported function $\psi$ on $\R$,
such that $0 \leq \psi \leq 1$ everywhere, $\psi(r) = 1$ for $r \leq R$, and 
$\psi'(r) \leq \tau r^{-1}$. This is easy to find, because $\int_R^{+\infty} \frac{1}{r}$ diverges. 
Then set $h_1(x,t) = \psi(r) h(x,t) + (1-\psi(r))$. Observe that $h_1$ satisfies \eqref{H1} 
trivially, is continuously differentiable, and satisfies \eqref{H2} because 
\begin{equation} \label{a4.65}
r |\nabla_{x,t} h_1| \leq r |\nabla_{x,t} h| \psi +  r |h-1| |\dr_r \psi|
\leq r|\nabla_{x,t} h| + C_{h1} \|r \psi'\|_\infty \leq C_{h0} + \tau C_{h1} \leq 2 C_{h0}
\end{equation}
if $\tau$ is chosen small enough. The pair $(\varphi_2,h_1)$ satisfies the assumptions
of the theorem (with the constant $2C_{h0}$, but this is not a worry, we just need to 
require $C_{h0}$ to be twice smaller in that theorem), so we can use Step 1 again
and we get that $N(\varphi_2,h) = N(\varphi_2, h_1)$.

Finally we can compute $N(\varphi_2, h_1)$ by computing $N_{\varphi_2, h_1}(z)$
for points that are very far. Let $R_1$ be such that $\psi(r) = 0$ for $r \geq R_1$.
Since $\varphi_2(x) = \varphi(0)$ everywhere, we get that for $(x,t) \in \Omega_0$,
$\varphi_r(x) = \varphi(0)$, $\hat v^i = v^i = e^i$ for $1 \leq i \leq d$,
$\hat w^j = w^j = e^j$ for $d+1 \leq j \leq n$, $R_{x,r}$ is the identity, 
and \eqref{defrhobis} yields
\begin{equation} \label{a4.66}
\rho_{\varphi_2,h}(x,t) = \Phi_r(x) + h(x,t) R_{x,r}(t) = (x,\varphi(0)) + h(x,t) t.
\end{equation}
If in addition $r = |t| \geq R_1$, we get 
$\rho_{\varphi_2,h}(x,t) = (x,\varphi(0)) + t = (x,t) + (0,\varphi(0))$.

Now we compute $N(\varphi_2, h_1) = N_{\varphi_2, h_1}(z)$ at a point
$z$ which we take at distance larger than $2C_{h_1}R_1$ from the graph $\Gamma'$ of
$\varphi_2$.  If $(x,t) \in \Omega_0$ is any solution of 
$\rho_{\varphi_2, h_1}(x,t) =z$, then by \eqref{a4.52},
$\dist(z,\Gamma')  \leq 2 C_{h_1} |t|$, which forces $|t| \geq R_1$.
Thus $\rho_{\varphi_2, h_1}(x,t) = (x,t) + (0,\varphi(0))$, and it is easy to see
that there is at most one solution $(x,t)$.

This proves that $N(\varphi_2, h_1) = N_{\varphi_2, h_1} \leq 1$. 
But $N(\varphi,h)=0$ is impossible, because $\rho_{\varphi_2,h_1}$ has to take 
{\it some} values somewhere. Altogether $N(\varphi, h) = N(\varphi_2, h_1) = 1$,
and this competes the proof of Lemma \ref{onetooneprop}, and at the same time
Theorem \ref{Tbijective}.
\end{proof}

We end this section with a remark which may not be useful for the rest of the paper, but is 
nonetheless reassuring. We complete our definition of $\rho$ by taking
\begin{equation} \label{a2.48bis}
\rho(x,0) = \Phi(x) = (x,\varphi(x)) \ \text{ for } x\in \R^d,
\end{equation}
as announced in \eqref{a2.48}. Because of \eqref{controlofdist}, $\rho$ is continuous across
$\R^d$ too. Since we have an upper bound on $\Jac$, $\rho$ is now globally
Lipschitz on $\R^n$. Since $\rho$ defines a bijection from $\R^d$ to $\Gamma$,
we get that $\rho$ is a Lipschitz bijection on $\R^n$. Finally, since the inverse of $\rho$
is continuous (including across $\Gamma$, by \eqref{a4.52} and \eqref{a4.53})
and Lipschitz on $\Omega$ (because we have bounds on $\Jac^{-1}$), we see that $\rho$
is bi-Lipschitz on $\R^n$, as promised in the introduction.

\section{The Carleson measure condition for the Jacobian} 
\label{SCMforJac}

We continue with the conventions and assumptions of the previous section, 
where we assume that $C_0 + C_{h0}$ is small enough, depending on $n$, $d$, $\eta$,
and $C_{h1}$, and $C$ denotes any constant that depend on these parameters.

Recall from \eqref{defAc} that (we plan to show that) the normalized matrix $\A$ 
that is associated to our conjugate operator $L_0$ is given by  
\begin{equation} \label{a5.1}
\begin{aligned}
\A(x,t) &= \left(\frac{|t|}{D(\rho(x,t))}\right)^{n-d-1} 
|\det(\Jac(x,t))| (\Jac(x,t)^{-1})^T \Jac(x,t)^{-1} \\
& = \left(\frac{|t|}{D(\rho(x,t))}\right)^{n-d-1} 
|\det(J(x,t))| (J(x,t)^{-1})^T J(x,t)^{-1}
\end{aligned}
\end{equation}
where $(\Jac(x,t)^{-1})^T$ denotes the transpose $\Jac(x,t)^{-1}$,
$J(x,t)$ is defined in Definition \ref{defJ},
and the second line comes from \eqref{JisJacQ}, which says that
$J(x,t) = \Jac(x,t) Q(x,t)$ for some orthogonal matrix $Q(x,t)$.

In this section we take care of 
\begin{equation} \label{a5.2}
\Ab(x,t) = |\det(J(x,t))| (J(x,t)^{-1})^T J(x,t)^{-1}
\end{equation}
which we try to put in a nice form (a nice block matrix, plus a perturbation that satisfies Carleson
estimates). The additional multiplicative term $\left(\frac{|t|}{D(\rho(x,t))}\right)^{n-d-1}$ 
will only  be treated in the next section, because this will involve estimates on the regularity of 
$\Gamma$ and $\sigma$. 

We do not reveal yet what is our choice of function $h$ in the definition of $\rho$,
but in addition to \eqref{H1} and \eqref{H2}, we now assume that 
$r\nabla_{x,t} h$ satisfies the Carleson measure condition, as in Definition \ref{defCMI},
with some constant $C_{h2}$. In short, 
\begin{equation} \label{H3}
r\nabla_{x,t} h \in CM(C_{h2}).
\end{equation}

The main result of this section is the following description of $\Ab$.

\begin{lemma} \label{main1.3} 
There exists $c >0$, that depends only on $n$, $d$, $\eta$, and $C_{h1}$,
such that if $C_0+C_{h0}\leq c$, then
there is a decomposition 
$\Ab = \Ab^1+\Ab^2$, where
\begin{enumerate}
\item both $\Ab$ and $\Ab^1$ are uniformly elliptic and bounded, 
\item both $r\nabla_{x,t} \Ab^1$ and $\Ab^2$ satisfy the Carleson measure condition, 
\item 
$\Ab^1 = \begin{pmatrix} \Ab^1_1 & 0 \\ 0 & \bb I_{n-d} \end{pmatrix}$,
where $\Ab^1_1$ is a $d\times d$ (elliptic) matrix and $\bb$ is a positive function, with 
\begin{equation} \label{a5.5}
C^{-1} \leq \bb (x,t) \leq C \ \text{ for } (x,t) \in \Omega_0.
\end{equation}
\end{enumerate}
\end{lemma}

In this context, bounded elliptic means that there exists $C >0$ such that 
$|\Ab_{k,\ell}| \leq C$ for $1 \leq k, \ell \leq n$ and
\begin{equation} \label{a5.6}
\Ab(x,t)\xi \cdot \xi \geq C^{-1}  |\xi|^2
\ \text{ for } (x,t)\in \Omega_0 \text{ and } \xi \in \R^n,
\end{equation}
and similarly for $\Ab^1$ (in $\R^n$).

The constant $C$ above depends only on $n$, $d$, $\eta$, and $C_{h1}$,
and the Carleson constants for $\Ab$ and $\Ab^1$ depend on $C_{h2}$ too.
Let us also observe that we will take
\begin{equation} \label{a5.7}
\bb = h^{-2} \det(J') = h^{n-d-2} \det(J'_1),
\end{equation} 
where $J'$ and $J'_1$ are defined in Definition \ref{defJprimeandM}.
We put the second formula here because $\det(J'_1)$ does not depend on $h$. 

\ms
The proof will make us busy for the rest of this section. 
We start with simple remarks on the Carleson condition.
Recall from Definition \ref{defCMI} that we say that the function $a(x,t)$, defined on $\Omega_0$, 
satisfies a Carleson measure condition when $|a(x,t)|^2 \frac{dx dt}{|t|^{n-d}}$ is a Carleson measure
on $\Omega_0$. 

When $a$ is defined on $\R^d \times (0,+\infty)$, we say that it satisfies a Carleson measure 
condition when $|a(x,|t|)|^2 \frac{dx dt}{|t|^{n-d}}$ is a Carleson measure on $\Omega_0$;
it is easy to check that this happens if and only if $|a(x,r)|^2 \frac{dx dr}{r}$ is a Carleson measure,
in the usual sense, on $\R^d \times (0,+\infty)$. That is, if there is a constant 
$C \geq 0$ such that
\begin{equation} \label{CMdef2}
\int_{0}^r \int_{B(x,r)} |a(y,s)|^2 dy \frac{ds}{s} \leq C r^d
\ \text{ for $x\in \R^d$ and $r > 0$.} 
\end{equation}

In both cases, when $a$ is matrix-valued, let us simply say that it satisfies a Carleson measure condition
if each of its entries $a_{k,\ell}$ satisfies a Carleson measure condition; we do not care about which norm
we would take on the spaces of matrices.

A very useful way to obtain Carleson measures is via the following result.

\begin{lemma} \label{lemLPL2}
Let $m=1$ or $m=d$. Let $\theta$ be in the Schwartz space $\mathcal S(\R^d,\R^m)$ and 
such that $\int_{\R^d} \theta = 0$. Define $$\theta_r (x) : = \frac1{r^d} \theta\left(\frac x r\right).$$
Then there exists $C>0$ such that for any $f\in L^2(\R^d,\R^m)$,
$$\int_0^\infty \|\theta_r * f\|_{L^2}^2 \frac{dr}{r} \leq C\|f\|_{L^2}^2,$$
where $\theta_r * f(x) = \int_{\R} \theta_r(y) f(x-y) dy$ if $m=1$ and $\theta_r * f(x) = \int_{\R^d} \left<\theta_r(y), f(x-y)\right> dy$ if $m=d$.

Moreover, if $\theta$ is supported in $B(0,2)$ and $f\in L^\infty(\R^d,\R^m)$, 
then $(x,r) \mapsto \theta_r * f(x)$ satisfies the Carleson measure condition.
\end{lemma}

\begin{proof}
The first part of the lemma (on the $L^2$-boundedness of $\theta_r * f$) can be found 
in \cite[Section I.6.3]{Stein93}. 

Let us turn to the proof of the second part. Let $f\in L^\infty(\R^d,\R^m)$. 
Let $x\in \R^d$ and $r>0$. The function $y\mapsto f(y)\1_{B(x,3r)}(y)$ lies
in $L^2(\R^d,\R^m)$ and thus the first part
implies that
\begin{equation} \label{LPL21}
\int_0^\infty \|\theta_s * [f\1_{B(x,3r)}]\|_{L^2}^2 \frac{ds}{s} 
\leq C \|f\|_{L^2(B(x,3r))}^2 \leq C\|f\|_\infty^2 r^d.
\end{equation}
Observe now that $\theta$ is supported in $B(0,2)$, 
thus $\theta_r$ is supported in $B(0,2r)$, 
and then $\theta_s * f(y) = \theta_s * [f\1_{B(x,3r)}](y)$ for any $(y,s) \in B(x,r)$. 
As a consequence,
\[\begin{split}
\int_0^r \int_{B(x,r)} |\theta_s * f(y)|^2 dy \frac{ds}{s}  & \leq \int_0^r \int_{B(x,r)} |\theta_s * [f\1_{B(x,3r)}](y)|^2 dy \frac{ds}{s} \\
& \leq \int_0^\infty \|\theta_s * [f\1_{B(x,3r)}]\|_{L^2}^2 \frac{ds}{s}  \leq C \|f\|_\infty^2 r^d,
\end{split}\]
where the last inequality is due to \eqref{LPL21}. Thus \eqref{CMdef2} holds for 
$\theta_s * f(y)$; the lemma follows.
\end{proof}

\ms
We are now ready to prove various Carleson measure estimates for $\rho$ and its components.

\begin{lemma} \label{lemCM1}
The quantities $|\dr_r \varphi_r|$ and  $r|\nabla_{x,r}\nabla_x \varphi_r|$ satisfy the Carleson measure condition, with a constant
$C>0$ that depends 
only upon $n$, $d$, and $\eta$. 
\end{lemma}

\begin{proof} According to Lemma \ref{lemLPL2}, it suffices to prove that  $\dr_r \varphi_r$,  $r\dr_{r}\dr_{x_k} \varphi_r$ and  $r\dr_{x_i}\dr_{x_k} \varphi_r$ can be written as $\theta_r*\nabla \varphi$ with $\theta \in \mathcal S(\R^d)$, $\supp \, \theta \subset B(0,1)$ and $\int_{\R^d} \theta = 0$.

\medskip
First, we have that $\dr_r \varphi_r = (\dr_r \eta_r)*\varphi$ where
\[\begin{split}
 \dr_r \eta_r(x) & = -\frac d{r^{d+1}} \eta\left(\frac xr \right) - \frac x{r^{d+2}} \cdot  \nabla \eta \left(\frac xr \right)  = \diver_x \cdot \left[ - \frac{x}{r^{d+1}} \eta\left(\frac xr \right) \right] \\
& = \diver_x \theta^1_r(x)
\end{split}\]
if $\theta^1(x) := x\eta(x)$. Thus 
Since $\eta$ is in  $\mathcal S(\R)$ and supported in $B(0,1)$, so is $\theta^1$. Besides, $\eta$ is even implies that $\theta^1$ is odd and thus $\int \theta^1 = 0$.

Second, $r\dr_{r}\dr_{x_k} \varphi_r = (r\dr_r \eta_r)*\dr_{x_k} \varphi = (r\dr_r \eta_r)*(\dr_{x_k} \varphi)$. In the same way, $r\dr_r \eta_r = \theta^2_r$ with $\theta^2 := \diver_x \theta^1$. Since $\theta^1$ is in  $\mathcal S(\R)$ and supported in $B(0,1)$, so is $\theta^2$. Moreover, $\theta^2$ is a derivative so $\int_{\R^d} \theta^2 = 0$.

Third, $r\dr_{x_i} \dr_{x_k} \varphi_r = (r\dr_{x_i} \eta_r)*\dr_{x_k} \varphi = \theta^3_r* \dr_{x_k} \varphi$ with $\theta^3(x) := \dr_{x_i} \eta (x)$. Since $\eta$ is in  $\mathcal S(\R)$ and supported in $B(0,1)$, so is $\theta^3$. Moreover, $\theta^3$ is a derivative so $\int_{\R^d} \theta^3 = 0$.
\end{proof}

\begin{lemma} \label{lemCM2}
Let $\epsilon\in (0,1)$. There exists $c_\epsilon > 0$, that depends on 
$\epsilon$, $n$, $d$, $\eta$ and $C_{h1}$, such that if $C_0+C_{h0} \leq c_\epsilon$, then
\begin{enumerate}[(i)]
\item the matrices $J'_1-I$, $H$ and $M$ have coefficients bounded by $\epsilon$,
\item $H \in CM(CC_{h2})$ and $M \in CM(CC_0^2)$,
\item 
$r \nabla_{x,t} 
J'_1 \in CM(CC_0^2)$.
\end{enumerate}
\end{lemma}

\begin{proof}
Let us start with (i). 
By \eqref{a4.25}, $H_{k\ell} \leq C C_{h0} \leq \epsilon$ if $C_{h0}$ is small enough.
Similarly, $M_{k\ell} \leq CC_0 \leq \epsilon$ by \eqref{MbyC0} if $C_{0}$ is small enough.
Recall from Definition \ref{defJprimeandM} that the coefficients of $J'_1$ are 
$(J'_1)_{k\ell} = \left< \hat v^k, v^\ell \right>$. So the fact that 
$|(J'_1)_{kk}-1| \leq \epsilon$ for $1\leq k\leq d$ follows from \eqref{barvvsim1}.
It remains to check that $|(J'_1)_{k\ell}| \leq \epsilon$ for $k \neq \ell$.
Observe that for any $k\in \{1,\dots,d\}$, 
\[0 \leq \sum_{\ell=1}^d \left<\hat v^k,v^\ell \right>^2 = |\hat v_k|^2 \leq 1+ C_0^2\]
since $\{v^\ell\}_{1\leq \ell\leq d}$ is an orthonormal basis and by \eqref{a3.8}.
As a consequence, if $\ell \neq i$, we have
\[ 
|\left<\hat v^i,v^\ell \right>|^2  \leq \sum_{\begin{subarray}{c} 1 \leq k \leq d  \\ k \neq \ell \end{subarray}} \left<\hat v^k,v^\ell \right>^2 \leq 1+C_0^2 - \left<\hat v^\ell,v^\ell \right>^2 \\
 \leq 2C_0^2, 
\]
where we used \eqref{barvv-1} for the last inequality. Point (i) follows.

\smallskip

The proof of (ii) is immediate from \eqref{H3}, Lemma \ref{defM} and Lemma \ref{lemCM1}.

\smallskip

To prove the last point, recall that $(J'_1)_{k\ell} = \left< \hat v^k,v^\ell \right>$.
Notice that 
\[r|\nabla_{x,r} \left< \hat v^k,v^\ell \right>| \leq r|\nabla_{x,r} \hat v^k| + r\sqrt{1+C_0^2} |\nabla_{x,r} {v^\ell}|\]
by \eqref{a3.8}.
Thus $r \nabla (J'_1)_{k\ell}$ satisfies the Carleson measure condition if $r|\nabla_{x,r} \hat v^k| = r|\nabla_{x,r} \dr_{x_k} \varphi_r|$ and $r|\nabla_{x,r} {v^\ell}|$ do. Yet, the latter fact is a consequence of Lemmas \ref{vrbyphirr} and \ref{lemCM1}. The lemma follows.
\end{proof}

\begin{lemma} \label{lemCM2b}
For any $\epsilon\in (0,1)$, there exists $c_\epsilon >0$ 
(depending on $n$, $d$, $\epsilon$, $\eta$ and $C_{h1}$) such that if $C_0+C_{h0}\leq c_\epsilon$, then
\begin{enumerate}[(i)]
\item 
$|\det(J)-\det(J')| \leq \epsilon$ and $|\det(J')-h^{n-d}| \leq \epsilon$,
\item $\det(J)-\det(J') \in CM(CC_0^2+CC_{h2})$,
\item $r\nabla_{x,t} \det(J'_1) \in CM(CC_0^2)$ and $r\nabla_{x,t} \det(J') \in CM(CC_0^2 + CC_{h2})$.
\end{enumerate}
\end{lemma}

\begin{proof}
The fact that $|\det(J)-\det(J')| \leq \epsilon$ if $C_0+C_{h0}$ is small comes from 
\eqref{detJ-detJp}, Lemma~\ref{rgradientphilem} and \eqref{H2}.
The fact that $\det(J)-\det(J')$ satisfies the Carleson measure condition is immediate 
from \eqref{detJ-detJp}, Lemma \ref{lemCM1} and \eqref{H3}.
Similarly $|\det(J')-h^{n-d}| \leq \epsilon$ when $C_0$ is small, 
by \eqref{detJp-1}. It remains to prove (iii). 

Recall that $\displaystyle\det(J'_1) = \sum_{\sigma \in \mathfrak S_d} \sgn(\sigma) 
\prod_{i=1}^d (J'_1)_{i,\sigma(i)}$. Then
$$
r\nabla_{x,t}
\det(J'_1) = \sum_{\sigma \in \mathfrak S_d} \sgn(\sigma) 
\sum_{j=1}^d (r\nabla_{x,t}
(J'_1)_{j,\sigma(j)}) \prod_{\begin{subarray}{c}i=1\\ i\neq j \end{subarray}}^d (J'_1)_{i,\sigma(i)}.
$$
Since the coefficients of $J'_1$ are bounded (see (i) in Lemma \ref{lemCM2}) and 
$r \nabla_{x,t} (J'_1)_{kl}$ satisfies the Carleson measure condition for any $1\leq k,l\leq d$ 
(see (iii) in Lemma \ref{lemCM2}), $r\nabla_{x,t} 
\det(J'_1)$ satisfies the Carleson measure condition.

Finally observe that $\det(J') = h^{n-d}\det(J'_1)$, hence 
\[
r\nabla_{x,t} \det(J') = (n-d)h^{n-d-1} r\nabla_{x,t} h \, \det(J'_1)
+ h^{n-d}  r \, \nabla_{x,t} \det(J'_1).
\]
Now $r\nabla_{x,t} \det(J') \in CM(CC_0^2+CC_{h2})$ because $r\nabla_{x,t} \det(J'_1) \in CM(CC_0^2)$
and \eqref{H1}, \eqref{H3}.
\end{proof}

In the sequel, we denote by $J''(x,t)$, $(x,t) \in \Omega_0$,  the diagonal matrix defined by
\begin{equation} \label{defJpp} (J'')_{k\ell}(x,t) 
: = \left\{\begin{array}{ll}
0 & \text{ if } k\neq \ell \\
1 & \text{ if } 1 \leq k=\ell \leq d \\
h(x,t) & \text{ if } d+1 \leq k=\ell \leq n.
\end{array}\right.
\end{equation}
Thus $\det(J'') = h^{n-d}$.

\begin{lemma} \label{lemCM3}
For any $\epsilon\in (0,1)$, there exists $c_\epsilon >0$ 
(depending on $n$, $d$, $\epsilon$, $\eta$ and $C_{h1}$) such that if 
$C_0\leq c_\epsilon$, then the matrices $J$ and $J'$ are invertible and
\begin{enumerate}[(i)]
\item $J^{-1}-(J')^{-1}$ and $(J')^{-1}-(J'')^{-1}$ have coefficients bounded by $\epsilon$,
\item $J^{-1}-(J')^{-1}\in CM(CC_0^2 + CC_{h2})$,
\item $r \nabla_{x,t} (J')^{-1}\in CM(CC_0^2 + CC_{h2})$. 
\end{enumerate}
\end{lemma}

\begin{proof}
First notice that by Proposition \ref{Jacobianneq0} 
(see in particular \eqref{detJsim1}), there exists $c_\epsilon > 0$
(depending on $n$, $d$, $\eta$ and $C_{h1}$) such that when $C_0+C_{h0} < c_\epsilon$,
\begin{equation} \label{detJ>0}
\text{$\det(J)$ and $\det(J')$ are both greater than $2^{-1} C_{h1}^{d-n}$.}
\end{equation}
In the sequel of the proof, we systematically assume that $C_0+C_{h0}$
is small enough for all the conditions above to be satisfied;  
in particular, the matrices $J$ and $J'$ are invertible.

\smallskip
Let prove (i). Cramer's rule yields for $k,\ell \in \{1,\dots,n\}$
\begin{equation} \label{invJ'formula} 
(J'^{-1})_{k\ell}  = \det(J')^{-1}
\Big(
\sum_{\begin{subarray}{c} \sigma \in \mathfrak S_n \\ \sigma(k) = \ell \end{subarray}} \sgn(\sigma) \prod_{\begin{subarray}{c} i=1 \\ i\neq k\end{subarray}}^{n} J'_{\sigma(i),i} \Big)
\end{equation}
and
\begin{eqnarray} \label{invJformula}
(J^{-1})_{k\ell} & = &\det(J')^{-1} 
\Big(\sum_{\begin{subarray}{c} \sigma \in \mathfrak S_n \\ \sigma(k) = \ell \end{subarray}} \sgn(\sigma) \prod_{\begin{subarray}{c} i=1 \\ i\neq k\end{subarray}}^{n} J_{\sigma(i),i} \Big)
\nn \\
& = & \big(\det(J') + (\det(J) - \det(J'))\big)^{-1}
\Big( \sum_{\begin{subarray}{c} \sigma \in \mathfrak S_n \\ \sigma(k) = \ell \end{subarray}} \sgn(\sigma) \prod_{\begin{subarray}{c} i=1 \\ i\neq k\end{subarray}}^{n} (J'_{\sigma(i),i} + H_{\sigma(i),i} + M_{\sigma(i),i}) \Big).
\end{eqnarray}
So the difference is bounded by (recall that \eqref{detJ>0} holds)
\begin{eqnarray}\label{invJminusinvJprime}
|(J^{-1})_{k\ell} - (J'^{-1})_{k\ell}| & \leq & C |\det(J')-\det(J)|  \left(\sup_{1\leq i,j\leq n} |J_{ij}| \right)^n 
\nn \\
&\,&  + C \left(\sup_{1\leq i,j\leq n} |H_{ij}| + |M_{ij}| \right)  \left(\sup_{1\leq i,j\leq n} |H_{ij}| + |M_{ij}| + |J'_{ij}| \right)^{n-1} ,
\end{eqnarray}
where $C>0$ depends only on $C_{h1}$.
Thanks to (i) of Lemma \ref{lemCM2} and (i) of Lemma \ref{lemCM2b}, for all $\epsilon\in (0,1)$, there exists $c_\epsilon$ such that if $C_0 + C_{h0} \leq c_\epsilon$,
\[|(J^{-1})_{k\ell} - (J'^{-1})_{k\ell}| \leq \epsilon.\]

The same argument can be repeated to prove that for all $\epsilon\in (0,1)$, there exists $c_\epsilon$ such that if $C_0 + C_{h0} \leq c_\epsilon$,
\[|(J'^{-1})_{k\ell} - (J''^{-1})_{k\ell}| \leq \epsilon.\]

The bound \eqref{invJminusinvJprime}, together with (ii) of Lemma \ref{lemCM2} 
and (ii) of Lemma \ref{lemCM2b}, implies in the same way that $J^{-1}-(J')^{-1}$ 
satisfies the Carleson measure condition.

In order to prove (iii), observe first that $J'$ is diagonal by blocs, 
hence $(J')^{-1}$ is also diagonal by blocs. 
In particular, we have for $k,\ell \in \{1,\dots,d\}$,
\[\begin{split}
(J'^{-1})_{k\ell} &  = \det(J'_1)^{-1} \Big(
 \sum_{\begin{subarray}{c} \sigma \in \mathfrak S_d \\ \sigma(k) = \ell \end{subarray}} \sgn(\sigma) \prod_{\begin{subarray}{c} i=1 \\ i\neq k\end{subarray}}^{d} (J'_1)_{\sigma(i),i} \Big).
\end{split}\]
It follows that for $k,\ell \in \{1,\dots,d\}$
\[\begin{split}
\nabla_{x,t} 
(J'^{-1})_{k\ell} & = \det(J'_1)^{-1}
\Big(\sum_{\begin{subarray}{c} \sigma \in \mathfrak S_d \\ \sigma(k) = \ell \end{subarray}} \sgn(\sigma) \sum_{j=1}^d \nabla (J'_1)_{\sigma(j),j} \prod_{\begin{subarray}{c} i=1 \\ i\neq j,k\end{subarray}}^{d} (J'_1)_{\sigma(i),i}  
\\ &\qquad\qquad\qquad
- \dfrac{\nabla_{x,t} 
\det(J'_1)}{\det(J'_1)^2} \Bigg(\sum_{\begin{subarray}{c} \sigma \in \mathfrak S_d \\ \sigma(k) = \ell \end{subarray}} \sgn(\sigma) \prod_{\begin{subarray}{c} i=1 \\ i\neq k\end{subarray}}^{d} J'_{\sigma(i),i}\Bigg).
\end{split}\]
The property (iii) of Lemmata \ref{lemCM2} and \ref{lemCM2b} implies that $\nabla (J'^{-1})_{k\ell}$ satisfies the Carleson measure condition for any $k,\ell \in \{1,\dots,d\}$. The Carleson measure condition for the gradient of the other coefficients of $J'^{-1}$ is either trivial or given by \eqref{H3}.
\end{proof}

In the following, if $A$ is a matrix, we use the notation $A^{-T}$ for the transpose inverse matrix of A,
that is, $(A^{-1})^T$ or equivalently $(A^{T})^{-1}$.

\begin{lemma} \label{lemCM3b}
For any $\epsilon\in (0,1)$, there exists $c_\epsilon >0$ 
(that depends on $\epsilon$, $\eta$ and $C_{h1}$) 
such that for $C_0 + C_{h0} \leq c_\epsilon$,
\begin{enumerate}[(i)]
\item the matrices $J^{-T}J^{-1} -(J')^{-T}(J')^{-1}$ and $(J')^{-T}(J')^{-1}-(J'')^{-2}$ have coefficients bounded by $\epsilon$,
\item $J^{-T}J^{-1} -(J')^{-T}(J')^{-1} \in CM(CC_0^2 + CC_{h2})$,
\item $r \nabla_{x,t} [(J')^{-T}(J')^{-1}] \in CM(CC_0^2 + CC_{h2})$.
\end{enumerate}
\end{lemma}

\begin{proof}
Just for this proof, let us write $A$ for $J^{-1}$, $A'$ for $J'^{-1}$ and $A''$ for $J''^{-1}$. One has for $1\leq k,\ell\leq n$
\begin{equation} \label{productmatrixATA}\begin{split}
(A^T A)_{k\ell} - (A'^T A')_{k\ell} & = \sum_{i=1}^n A_{ik} A_{i\ell} - \sum_{i=1}^n A'_{ik} A'_{i\ell} = \sum_{i=1}^n (A_{ik} A_{i\ell} - A'_{ik} A'_{i\ell}) \\
& = \sum_{i=1}^n A_{ik} (A_{i\ell} - A'_{i\ell}) + (A_{ik} - A'_{ik})A'_{i\ell}.
\end{split}\end{equation}
Then, thanks to (i) of Lemma \ref{lemCM3} (and assuming as usual that $C_0 + C_{h0}$ is small),
$(A^T A)_{k\ell} - (A'^T A')_{k\ell} \leq \epsilon$. 
In the same way, we can prove that $(A'^T A')_{k\ell} - (A''^T A'')_{k\ell} \leq \epsilon$ if 
$C_0 + C_{h0}$ is small enough.

The identity \eqref{productmatrixATA} and (ii) of Lemma \ref{lemCM3} entail the Carleson measure condition for $(A^T A)_{k\ell} - (A'^T A')_{k\ell}$.
Finally, 
\[\begin{split}
\nabla (A'^T A')_{k\ell} = \sum_{i=1}^n A'_{i\ell} \nabla A'_{ik} + A'_{ik} \nabla A'_{i\ell},
\end{split}\]
which implies, by (iii) of Lemma \ref{lemCM3}, that $\nabla (A'^T A')$ satisfies the Carleson measure condition.
\end{proof}

We are now ready to complete the proof of Lemma \ref{main1.3}.
Assume $C_0 + C_{h0}$ is small enough, so that we can apply the previous lemmas for
values of $\epsilon$ that we can decide along the way.

Thanks to Lemma \ref{lemCM2b}, $J'$ and $J$ are both invertible,
and furthermore the coefficients of our main matrix $\Ab =  |\det(J)| J^{-T}J^{-1}$ 
of \eqref{a5.2} are bounded. Observe also that $\det(J)$ and $\det(J')$ are positive 
by Lemma \ref{Jacobianneq0} and so $\Ab =  \det(J) J^{-T}J^{-1}$. 
Set
\[\Ab^1 = \det(J') \, J'^{-T}J'^{-1},\]
 \[\begin{split} 
 \Ab^2 & =  \det(J) J^{-T}J^{-1} -  \det(J') J'^{-T}J'^{-1} \\
 & = (\det(J)-\det(J')) J^{-T}J^{-1} + \det(J') (J^{-T}J^{-1} - J'^{-T}J'^{-1}),
 \end{split} \]
 and finally
 \[\begin{split} 
 \Ab^3 & =  \det(J') J'^{-T}J'^{-1} -  \det(J'') J''^{-2} \\
 & = (\det(J')-\det(J'')) J'^{-T}J'^{-1} + \det(J'') (J'^{-T}J'^{-1} - J''^{-2}).
 \end{split}\]
 By Point (i) in Lemmas \ref{lemCM2b} and \ref{lemCM3b} (plus the fact that 
 $\det(J'') = h^{n-d}$), the coefficients of $\Ab^2$ and $\Ab^3$ are as small as we want
 in $L^\infty$-norm. Since $\Ab^1 = |\det(J'')| J''^{-2} - \Ab^3$
 and $\Ab = |\det(J'')| J''^{-2} - \Ab^2 - \Ab^3$, $\Ab^1$ and $\Ab$ are both small
 perturbations of $|\det(J'')| J''^{-2}$. But the diagonal matrix $|\det(J'')| J''^{-2}$,
 where $J''$ comes from \eqref{defJpp}, is clearly bounded and elliptic
 (which means that \eqref{a5.6} holds), with an ellipticity constant $C$ in \eqref{a5.6}
 bounded by $C_{h1}^{n-d-2}$.
 It easily follows that, if the perturbations are small enough, $\Ab^1$ and $\Ab$ are 
 (uniformly) bounded and elliptic. This takes care of Point (1) of Lemma \ref{main1.3}.
 
It is clear that $\Ab = \Ab^1 + \Ab^2$. Concerning Point (2), notice that
$$\nabla_{x,t} \Ab^1 = \det(J') \nabla_{x,t} (J'^{-T}J'^{-1}) + (J'^{-T}J'^{-1}) \nabla_{x,t} \det(J')$$ 
and 
$$\Ab^2 = \det(J) (J^{-T}J^{-1} -  J'^{-T}J'^{-1}) + (\det(J)-\det(J'))  J'^{-T}J'^{-1}.$$ 
Lemmata \ref{lemCM2b} and \ref{lemCM3b} allow us to conclude that both $r\nabla_{x,t} \Ab^1$ and $\Ab^2$ satisfy the Carleson measure condition.

Now we check (3).
Recall from Definition \ref{defJprimeandM} that $J'$ has the form 
$\begin{pmatrix} J'_1 & 0 \\ 0 & hI_{n-d} \end{pmatrix}$. Thus 
\begin{equation} \label{defAb1}
\Ab^1 = \det(J') \begin{pmatrix} (J'_1)^{-T}(J'_1)^{-1} & 0 \\ 0 & h^{-2}I_{n-d} \end{pmatrix}.
\end{equation} 
Let $\bb = h^{-2} \det(J')$; then \eqref{defAb1} is the same as the representation of (3),
and \eqref{a5.7} holds too. 

Lemma \ref{lemCM2b}, point (i) entails that $|\det(J') - h^{n-d}|\leq \epsilon$, with 
$\epsilon$ as small as we want. Since $h$ is bounded and bounded from below, we get 
\eqref{a5.5}. Lemma \ref{main1.3} follows.
\qed

\section{The $\alpha$-numbers and the regularity of the soft distance $D$}

\label{SCMforD} 

In this section we use the $\alpha$-numbers introduced by X. Tolsa 
to give some control on the geometry of our graph $\Gamma$, 
an Ahlfors-regular measure $\sigma$ supported on $\Gamma$,
and eventually the distance function $D = D_\alpha$ defined by \eqref{IdefD}.

The main estimate is Lemma \ref{L6D}, which gives some control on $D$
in terms of an average density of the measure $\sigma$, and with errors controlled by
a Carleson measure. This will lead to the estimate \eqref{DrhoisCM}, for a suitable choice of
auxiliary function $h$. 

Here $\Gamma$ is our Lipschitz graph, and for $\sigma$ we start with any positive 
measure on $\Gamma$ that is Ahlfors-regular, i.e., satisfies \eqref{defADR} for some 
constant $C_\sigma \geq 1$. For this section we need some known geometric estimates 
on $\Gamma$ and $\sigma$, which happen to hold for any uniformly rectifiable set, 
but which we may write slightly differently in the context of small Lipschitz graphs 
to simplify the exposition.

Later on, we shall need to restrict to measures $\sigma$ that lie close enough to
the surface measure, because otherwise the function $h$ that corresponds to $D$
may not satisfy the slow variation condition \eqref{H2} with a small enough constant.
But we shall only worry about this in the next section.

The generic constant $C$ in this section is allowed to depend on $n$, $d$, 
$C_\sigma$, $\eta$, and another bump function $\theta$ that will be chosen later,
and (later in the section) the exponent $\alpha$ in the definition of $D$.

We start with the proof of \eqref{equivD-bis} promised in the Introduction.

\begin{lemma}\label{lequivD}
If $\Gamma$ is Ahlfors-regular and $\sigma$ satisfies \eqref{defADR},
\begin{equation} \label{equivD}
C^{-1} \dist(X,\Gamma) \leq D_\alpha(X) \leq C\dist(X,\Gamma)
\end{equation}
where $C>0$ is a constant that depends only on $n$, $d$, $\alpha$, and $C_\sigma$. 
\end{lemma}

\bp Proving \eqref{equivD} is the same as proving that $D(X)^{-\alpha}$ is equivalent to 
$\dist(X,\Gamma)^{-\alpha}$. Let us prove the latter fact. 
Let $X\in \Omega$ be given, set $r = \dist(X,\Gamma)$, 
pick $x\in \Gamma$ such that $|X-x| = r$, and decompose then $D_\alpha(X)^{-\alpha}$
into contributions of annuli as
\[\begin{split}
D_\alpha(X)^{-\alpha} 
& = \int_{|y-x| < 2r} |X-y|^{-d-\alpha} d\sigma(y) 
+ \sum_{k\geq 1} \int_{2^k r\leq |y-x| <2^{k+1}r} |X-y|^{-d-\alpha} d\sigma(y) \\
& \leq r^{-d-\alpha} \sigma(B(x,2r)) + \sum_{k\geq 1} ((2^k-1)r)^{-d-\alpha} \sigma(B(x,2^{k+1}r)) \\
&\leq C_\sigma r^{-d-\alpha}(2r)^d + C_\sigma \sum_{k=1}^\infty (2^{k}-1)^{-d-\alpha} 
r^{-d-\alpha} (2^{k+1}r)^d \leq C^{-1} r^{-\alpha}
\end{split}\]
by \eqref{defADR}. The reverse inequality is easy too, since
\[D_\alpha(X)^{-\alpha} \geq \int_{|y-x| < r} |X-y|^{-d-\alpha} d\sigma(y) \geq (2r)^{-d-\alpha} \sigma(B(x,r)) \geq C^{-1} r^{-\alpha},\]
where the last inequality uses the lower bound in \eqref{defADR}.
\ep

\subsection{Wasserstein distances}
Most of our estimates will be based on a result of X. Tolsa, 
\cite[Theorem 1.1]{Tolsa09}, which gives a good control on sums of squares of local Wasserstein 
distances to flat measures, for every Ahlfors-regular measure on a uniformly rectifiable set.
Here our set is a small Lipschitz graph, which makes the verification of Tolsa's theorem
easier, but we will not really need the Lipschitz character of $\Gamma$.

We first define flat measures and local Wasserstein distances. 
Denote by $\Xi$ the set of affine $d$-planes in $\R^n$, and for each
plane $P \in \Xi$, denote by $\mu_P$ the restriction of $\H^d$ to $P$
(in other words, the Lebesgue measure on $P$). By {\bf flat measure}, we shall 
mean simply mean a measure $\mu = c \mu_P$, with $c > 0$ and $P\in \Xi$.
We shall denote by $\cF$ the set of flat measures.

Next we measure the distance between our measure $\sigma$ and flat measures,
locally in a ball $B(z,r)$, which we shall often take centered on $\Gamma$ because 
this way we know that $\sigma(B(z,r))$ is fairly large.

\begin{definition} \label{D6.1}
For $z\in \R^n$ and $r > 0$, denote by $Lip(z,r)$ the set of Lipschitz functions 
$f : \R^n \to \R$ such that $f(y)=0$ for $y\in \R^n \sm B(z,r)$ and $|f(y)-f(w)|\leq |y-w|$
for $y,w\in \R^n$. Then define the normalized Wasserstein distance between two measures
$\sigma$ and $\mu$ by
\begin{equation} \label{a6.2}
\dist_{z,r}(\mu,\sigma) = r^{-d-1} \sup_{f\in Lip(z,r)} \Big|\int f d\sigma - \int f d\mu\Big|.
\end{equation}
Then define the distance to flat measures by
\begin{equation} \label{a6.3}
\wt\alpha_\sigma(z,r) = \inf_{\mu \in \cF}\dist_{z,r}(\mu,\sigma).
\end{equation}
\end{definition}

\ms
We normalized $\dist_{z,r}(\mu,\sigma)$ with $r^{-d-1}$ because this way,
if $\mu(B(z,r)) \leq C r^d$ and $\sigma(B(z,r)) \leq C r^d$, then 
$\dist_{z,r}(\mu,\sigma) \leq 2C$ because
\begin{equation} \label{a6.4}
\|f\|_\infty \leq r \ \text{ for } f\in Lip(z,r).
\end{equation}
Also observe that if $B(y,s) \subset B(z,r)$, then $Lip(y,s) \subset Lip(z,r)$; it follows that
$\dist_{y,s}(\mu,\sigma) \leq (r/s)^{d+1} \dist_{z,r}(\mu,\sigma)$, and hence
\begin{equation} \label{a6.5}
\wt\alpha_\sigma(y,s) \leq (r/s)^{d+1} \wt\alpha_\sigma(z,r) 
\ \text{ when } B(y,s) \subset B(z,r).
\end{equation}

In Theorem 1.1 of \cite{Tolsa09}, the author uses slightly different numbers, defined as follows.
For every dyadic cube $Q$ in $\R^d$, denote by $z_Q$ the center of $Q$, $d(Q)$ its
diameter, and set
\[
\wt \alpha_\sigma^d(Q) = \wt\alpha_\sigma(z_Q,3d(Q)).
\]
Then he proves the following Carleson measure estimate, valid when $\Gamma$ is a Lipschitz
graph (maybe with a large constant $C_0$) and $\sigma = g \H^{d}_{\vert \Gamma}$
for some bounded function $g$: for every dyadic cube $R \subset \R^n$ that meets the support
of $\sigma$, 
\begin{equation} \label{a6.6}
\sum_{Q \in {\mathcal D}(R)} \wt \alpha_\sigma^d(Q)^2 \sigma(Q) \leq C d(R)^d,
\end{equation}
where ${\mathcal D}(R)$ denotes the collection of dyadic cubes contained in $R$, and
the constant $C$ depends on $n$, $d$, $C_0$, and $\|g\|_\infty$.

The same statement, modulo cosmetic changes, stays true when $\Gamma$ is
a uniformly rectifiable set of dimension $d$ and $\sigma$ an Ahlfors-regular measure
on $\Gamma$. This is even a characterization of uniform rectifiability. See Theorem 1.2 in 
\cite{Tolsa09}. We shall not need this fact here.

We want to turn \eqref{a6.6} into a Carleson estimate of the usual type, with a function
of $x\in \R^d$ and $r>0$.

\begin{lemma} \label{L6.7} For every Ahlfors-regular measure $\sigma$ with support $\Gamma$,
the function $(x,r) \to \wt\alpha_\sigma(\Phi(x),r)$ (defined on $\R^d \times (0,+\infty)$)
satisfies the Carleson measure condition.
\end{lemma}

\begin{proof}
See \eqref{CMdef2} (or just \eqref{a6.11} below) for the definition of the Carleson condition.
Let $B(x,r) \subset \R^d$ be given, let $k_0 \in \mathbb Z$ be such that
$2r \leq 2^{k_0} < 4r$, and the cover $B(\Phi(x),2r)$ by less than $C$ 
disjoint dyadic cubes $R_j$ of sidelength $2^{k_0}$. For each $j$ and each dyadic
subcube $Q \in {\mathcal D}(R_j)$, denote by $H(Q)$ the set of pairs $(y,s)$
such that $\Phi(y)\in Q$ and $d(Q) < s \leq 2 d(Q)$. 
Notice that each pair $(y,s)$, with $y\in B(x,r)$ and $0 < s \leq r$, lies in one of 
these $H(Q)$ (because $\Phi(y) \in B(\Phi(x),2r) \subset \cup_j R_j$). In addition, 
if $(y,r) \in H(Q)$, $B(y,s) \leq B(z_Q,3d(Q))$, and hence by \eqref{a6.5}
\begin{equation} \label{a6.8}
\wt\alpha_\sigma(y,s) \leq (3d(Q)/s)^{d+1}\wt\alpha_\sigma(z_Q,3d(Q))
= (3d(Q)/s)^{d+1} \wt\alpha_\sigma^d(Q).
\end{equation}
Hence
\begin{eqnarray} \label{a6.9}
\int_0^r \int_{\R^d \cap B(x,r)} \wt\alpha_\sigma(y,s)^2 \,\frac{dxds}{s}
&\leq& \sum_j \sum_{Q \in {\mathcal D}(R_j)} \int_{(y,s)\in H(Q)} 
\wt\alpha_\sigma(y,s)^2 \,\frac{dxds}{s}
\nn\\
&\leq& C \sum_j \sum_{Q \in {\mathcal D}(R_j)} \wt\alpha_\sigma^d(Q)^2
\int_{(y,s)\in H(Q)} \frac{dxds}{s}
\end{eqnarray}
by \eqref{a6.8}. Next 
\begin{equation} \label{a6.10}
\int_{(y,s)\in H(Q)} \frac{dxds}{s}
= |\{ y \in \R^d \, ; \, \Phi(y) \in Q \}| \int_{d(Q) < s \leq 2 d(Q)} \frac{ds}{s}
\leq C \ln(2)  \sigma(Q)
\end{equation}
because the pushforward by $\Phi$ of the Lebesgue measure on $\R^d$ is less than 
$C \sigma$.
Altogether
\begin{eqnarray} \label{a6.11}
\int_0^r \int_{\R^d \cap B(x,r)} \wt\alpha_\sigma(y,s)^2 \frac{dxds}{s}
&\leq& C \sum_j \sum_{Q \in {\mathcal D}(R_j)} \wt\alpha_\sigma^d(Q)^2 \sigma(Q)
\leq C \sum_j d(R_j)^d \leq C r^d
\end{eqnarray}
by \eqref{a6.6}. This is the desired Carleson measure estimate (compare with
\eqref{CMdef2}).
\end{proof}

Here we shall work with a fixed measure $\sigma$ on $\Gamma$, which we may drop 
from the notation, and it will be simpler to work with the function $\alpha$ defined by
\begin{equation} \label{a6.12}
\alpha(x,r) = \wt\alpha_\sigma(\Phi(x),4r)
\ \text{ for $x\in \R^d$ and $0 < r < +\infty$.}
\end{equation}
It easily follows from Lemma \ref{L6.7} that 
\begin{equation} \label{a6.13}
(x,r) \to \alpha(x,r)\text{ satisfies the Carleson measure condition.} 
\end{equation}

\subsection{The ensuing geometric control}
Next we want to show that the numbers $\alpha(x,t)$ control the geometry
of $\Gamma$ and $\sigma$. In particular, the $d$-planes $P(x,r)$ essentially
minimize $\alpha(x,r)$ (see Lemma \ref{Lreplace}), and the normalized average density
$\lambda(x,r)$ defined by \eqref{a6.18} is reasonably smooth (see Lemma \ref{CMfornablah}).
 
For $x\in \R^d$ and $r >0$, choose a flat measure $\wt\mu_{x,r}$ such that 
\begin{equation} \label{a6.14}
\dist_{\Phi(x),4r}(\wt\mu_{x,r},\sigma) \leq 2 \wt\alpha_\sigma(\Phi(x),4r) = 2\alpha(x,r).
\end{equation}
Recall from \eqref{a6.2} that this means that 
\begin{equation} \label{a6.15}
\Big|\int f d\wt\mu_{x,r}-\int f d\sigma \Big| \leq 2 \alpha(x,r) (4r)^{d+1}
\ \text{ for } f\in Lip(\Phi(x),4r).
\end{equation}
Since $\wt\mu_{x,r}$ is a flat measure, 
\begin{equation} \label{a6.16}
\wt\mu_{x,r} = c(x,r) \mu_{\Lambda(x,r)}
\end{equation}
for some choice of $c > 0$ and affine $d$-plane $\Lambda(x,r)$, but
we would prefer to use constants and $d$-planes that we control better.

So our next task will be to replace $\Lambda(x,r)$ with the approximating
$d$-plane $P(x,r)$ of Section \ref{Sorthbasis}, and $c(x,r)$ with a number
$\lambda(x,r)$ that we define now.

Fix $\theta \in C^\infty(\R^n)$ such that $0\leq \theta \leq 1$, 
$\supp \, \theta \subset B(0,1)$ and $\theta=1$ on $B(0,\frac12)$. 
Then set
\begin{equation} \label{a6.17}
\theta_{x,r}(y) = \theta\Big(\frac{\Phi_r(x)-y}{r} \Big),
\end{equation}
where $\Phi_r(x) = (x,\varphi_r(x))$ is as in Section \ref{Sorthbasis},
and finally
\begin{equation} \label{a6.18}
\lambda(x,r) = \dfrac{\int_\Gamma \theta_{x,r}(y) d\sigma(y)}
{\int_{P(x,r)} \theta_{x,r}(y) d\mu_{P(x,r)} }
= \dfrac{\int_\Gamma \theta_{x,r}(y) d\sigma(y)}
{r^d\int_{\R^d} \theta(y) d\H^{d}(y) },
\end{equation}
where the second identity uses the fact that we centered $\theta_{x,r}(y)$
at $\Phi_r(x) \in P(x,r)$. In other words, we normalize the measure
\begin{equation} \label{a6.19}
\mu_{x,r} = \lambda(x,r) \mu_{P(x,r)}
\end{equation}
by the fact that it has the same effect on a normalized bump $\theta_{x,r}$
as $\sigma$. Here is our replacement lemma.

\begin{lemma} \label{Lreplace}
With the definitions above (and if $C_0$ is small enough),
\begin{equation} \label{a6.21}
\dist_{\Phi(x),r}(\mu_{x,r},\sigma) \leq C\alpha(x,r).
\end{equation}
\end{lemma}

Once this is proved, we will be able to forget about $\Lambda(x,r)$ and $c(x,r)$,
because we have more explicit and convenient choice that does nearly as well.
\begin{proof}
Let $x\in \R^d$ and $r > 0$ be given. We may as well assume that
\begin{equation} \label{a6.22}
\alpha(x,r) \leq c_1,
\end{equation}
where $c_1 > 0$ will be chosen soon (depending on $n$, $d$, and $C_\sigma$)
will be chosen later. This information will be used from time to time, to eliminate strange cases.

\ms
We first control the average distances from points of $\Gamma$ to $\Lambda(x,r)$.
We claim that 
\begin{equation} \label{a6.23}
\int_{\Gamma \cap B(\Phi(x),2r)}\dist(z,\Lambda(x,r)) d\sigma(z)
\leq C r^{d+1}\alpha(x,r).
\end{equation}
Let us get rid of a minor potential problem first.
It could happen \em a priori \em 
that the $d$-plane $\Lambda(x,r)$ 
does not meet $B(\Phi(x),4r)$. In fact this is impossible when \eqref{a6.22}
holds with a small $c_1$, but let us not bother to prove this.
When this happens, notice that $\int f d\wt \mu_{x,r} = 0$ 
 for every $f\in Lip(\Phi(x),4r)$ 
(because $f=0$ on $\Lambda(x,r)$), so $\dist_{\Phi(x),4r}(\wt\mu_{x,r},\sigma)$ 
does not change if we replace $\Lambda(x,r)$ with any $d$-plane $\wt\Lambda(x,r)$
that does not meet $B(\Phi(x),r)$ and $\wt\mu(x,r)$ with any flat measure on 
$\wt\Lambda(x,r)$. For instance, we may choose $\wt\Lambda(x,r)$ so that
it touches $\dr B(\Phi(x), 4r)$.
Thus we may assume that $\Lambda(x,r)$ was always chosen so that
\begin{equation} \label{a6.24}
\dist(\Phi(x),\Lambda(x,r)) \leq 4r.
\end{equation}
In order to prove \eqref{a6.23}, we shall construct a function $f\in Lip(\Phi(x),4r)$ and 
test it against \eqref{a6.15}. First define a Lipschitz function $\psi$ by 
$\psi(z) = \frac18$ for $z\in B(\Phi(x),2r)$ and 
\begin{equation} \label{a6.25}
\psi(z) = \frac1{16r} (4r - |z-\Phi(x)|)_+ = \frac1{16r} \max(0, 4r - |z-\Phi(x)|)
\end{equation}
otherwise.  Then set  $f(z) = \psi(z) \dist(z,\Lambda(x,r))$. Notice that
$$
|\nabla f(z)|  \leq \psi(z) + \dist(z,\Lambda(x,r)) |\nabla \psi(z)|
\leq \frac14 + 8r |\nabla \psi(z)| \leq 1,
$$ 
by \eqref{a6.24}, and since $\psi(x) = 0$ on $\R^n \sm B(\Phi(x),4r)$, we get that $f\in Lip(\Phi(x),4r)$ and \eqref{a6.15} applies. This yields
\begin{eqnarray} \label{a6.26}
\int_{\Gamma \cap B(\Phi(z),2r)}\dist(z,\Lambda(x,r)) d\sigma 
&\leq& 8\int_{\Gamma \cap B(\Phi(z),2r)} \psi(z) \dist(z,\Lambda(x,r)) d\sigma
\nn \\
&=& 8 \int f(z) (d\sigma - d\mu_{x,r}) \leq C r^{d+1} \alpha(x,r)
\end{eqnarray}
since $\int f d\mu_{x,r}= 0$ because $f$ vanishes on its support $\Lambda(x,r)$.
This proves our claim \eqref{a6.23}.

\ms
Next we estimate the distance from $P(x,r)$ to $\Lambda(x,r)$. 
We start with the position of the base point 
$\Phi_r(x) = \eta_r \ast \Phi(x) = \int_{y\in B(x,r)} \eta_r(x-y) \Phi(y) dy$. 
Simply observe that
\begin{eqnarray} \label{a6.27}
\dist(\Phi_r(x),\Lambda(x,r))
&\leq& \int_{y\in B(x,r)} \eta_r(x-y) \dist(\Phi(y),\Lambda(x,r)) dy
\nn\\
&\leq& r^{-d} \|\eta\|_\infty \int_{y\in B(x,r)} \dist(\Phi(y),\Lambda(x,r)) dy
\nn\\
&\leq& C r^{-d} \|\eta\|_\infty \int_{z\in \Gamma \cap B(\Phi(x),2r)} 
\dist(z,\Lambda(x,r)) d\sigma(z)
\leq C  r \alpha(x,r)
\end{eqnarray}
because the pushforward by $\Phi$ of the Lebesgue measure $dy$ is less than 
$C d\sigma$, 
 and by \eqref{a6.23}.

Then we control the direction $P'(x,r)$ of $P(x,r)$.
Recall from the discussion near \eqref{defbarv}
that $P'(x,r)$ is the $d$-dimensional vector space spanned by the 
$\hat v^i(x,r) = \dr_{x_i} \Phi_r(x)$, $1 \leq i \leq d$. 
Denote by $\Pi^\perp = \Pi^\perp_{x,r}$ the orthogonal projection 
on the orthogonal complement of (the direction of) $\Lambda(x,r)$. 
Thus $\Pi^\perp(\Lambda(x,r))$ is a single point $\xi$. For $1 \leq i \leq d$, 
$\dr_{x_i} \Phi_r = \Phi \ast \dr_{x_i}(\eta_r) = r^{-1} \Phi \ast \psi^i_r$,
where we set $\psi^i = \dr_{x_i} \eta$ and as usual $\psi^i_r(y) = r^{-d}\psi^i(y/r)$.
Then, by the same sort of computation as before,
\begin{eqnarray} \label{a6.28}
|\Pi^\perp(\hat v^i(x,r))|
&=& |\Pi^\perp(\dr_{x_i} \Phi_r(x))| 
=  r^{-1} \big| \Pi^\perp [\Phi \ast \psi^i_r(x)] \big|
=  r^{-1}\Big|\int_{y\in B(x,r)} \psi^i_r(x-y) \Pi^\perp (\Phi_r(y)) dy \Big|
\nn\\
&=&   r^{-1}\Big|\int_{y\in B(x,r)} \psi^i_r(x-y) [\Pi^\perp ( \Phi_r(y)) - \xi] dy \Big|
\nn\\
&\leq& r^{-d-1} \|\psi^i\|_\infty \int_{y\in B(x,r)} \dist(\Phi_r(y),\Lambda(x,r)) dy
\\
&\leq& C r^{-d-1} \|\psi^i\|_\infty \int_{z\in \Gamma \cap B(\Phi(x),2r)} 
\dist(z,\Lambda(x,r)) d\sigma(z) 
\leq C \alpha(x,r)
\nn
\end{eqnarray}
where the second line comes from the fact that $\int_{\R^d} \psi^i = 0$
because $\psi^i = \dr_{x_i} \eta$, the third line from the fact 
$|\Pi^\perp(z)-\xi| = \dist(z, \Lambda(x,r))$, and then we continue as before.
Now the $\hat v^i(x,r)$ are a nearly orthonormal basis of $P'(x,r)$ 
(because they are as close as we want to the $e^i$
), so, if $C_0$ is small enough,
\eqref{a6.28} actually implies that
\begin{equation} \label{a6.29}
|\Pi^\perp(v)| \leq C \alpha(x,r) |v| \ \text{ for } v\in P'(x,r).
\end{equation}

Denote by $\Lambda'(x,r)$ the vector $d$-space parallel to $\Lambda(x,r)$;
we just showed that
\begin{equation} \label{a6.30}
\dist(v,\Lambda'(x,r)) \leq  C \alpha(x,r) |v| \ \text{ for } v\in P'(x,r).
\end{equation}
It now follows from \eqref{a6.30} and \eqref{a6.27} that 
\begin{equation} \label{a6.32}
\dist(y,\Lambda(x,r)) \leq C \alpha(x,r) [r+|y-\Phi_r(x)|] 
\ \text{ for } y \in P(x,r).
\end{equation} 

\ms
Now we worry about the constants $c(x,r)$ and $\lambda(x,r)$.
Notice that for $C_0$ small
\begin{equation} \label{a6.31}
|\Phi_r(x)-\Phi(x)| = |\varphi_r(x)-\varphi(x)| \leq C_0 r \leq \frac r{10},
\end{equation}
because $\varphi_r(x)$ is an average of $\varphi(y)$, $y\in B(x,r)$. 
In particular, $\Phi_r(x) \in B(\Phi(x),r)$ and hence $B(\Phi_r(x),r)\leq B(\Phi(x),2r)$.
Let us apply now \eqref{a6.15} to the function $\theta_{x,r}$ of \eqref{a6.17}.
The support is right, because of \eqref{a6.31}, but the Lipschitz norm is $r^{-1} \|\theta\|_{lip}$, 
so we divide
$\theta_{x,r}$ by that number, apply \eqref{a6.15}, and get that 
\begin{equation} \label{a6.34}
\Big|\int \theta_{x,r} d\wt\mu_{x,r} - \int \theta_{x,r} d\sigma] \Big| \leq C r^d \alpha(x,r),
\end{equation}
where $C$ depends on $\theta$ but this is all right. Set
\begin{equation} \label{a6.35}
a_0 = \int_{\R^d} \theta(y) d\H^{d}(y);
\end{equation}
One of the two terms of \eqref{a6.34} is
\begin{equation} \label{a6.36}
A_\sigma = \int \theta_{x,r} d\sigma = a_0 r^d \lambda(x,r).
\end{equation}
The other term is
\begin{equation} \label{a6.37}
A_\mu = \int\theta_{x,r} d\wt\mu_{x,r} = c(x,r) \int\theta_{x,r} d\mu_{\Lambda(x,r)},
\end{equation}
which we shall write as an integral on $P(x,r)$.
Denote by $\pi$ the orthogonal projection from $P(x,r)$ to $\Lambda(x,r)$; 
by \eqref{a6.29} this is an affine bijection, and a very brutal estimate shows that its constant
Jacobian determinant $J$ is such that $|J-1| \leq C \alpha(x,r)$ (we could even get a square). 
We do the change of variables $z = \pi(y)$ and find that
\begin{equation} \label{a6.38}
A_\mu =  c(x,r)\int_{\Lambda(x,r)}\theta_{x,r}(z) d\mu_{\Lambda(x,r)}(z)
= c(x,r) J \int_{P(x,r)} \theta_{x,r}(\pi(y)) d\mu_{P(x,r)}(y)
\end{equation}
By \eqref{a6.32}, $|\theta_{x,r}(\pi(y)) - \theta_{x,r}(y)| 
\leq r^{-1} \|\theta\|_{lip} |\pi(y)-y|  \1_{B(\Phi_r(x),2r)}(y) \leq C \alpha(x,r)$; 
it is useful to observe that since $c_1$ is small, then  by \eqref{a6.32} again, 
$\theta(\pi(y)) \neq 0$ implies that $y\in B(\Phi_r(x),2r)$. 
But $\int_{P(x,r)} \theta_{x,r}(y) d\mu_{P(x,r)}(y) = a_0 r^d$
(because $\theta_{x,r}$ is centered on $P(x,r)$; we already did this computation in \eqref{a6.18}), 
so 
\begin{equation} \label{a6.39}
|A_\mu - c(x,r) J a_0 r^d| \leq C c(x,r) J \int_{P(x,r) \cap B(\Phi(x),2r)} \alpha(x,r) d\mu_{P(x,r)}
\leq C c(x,r) \alpha(x,r) r^d.
\end{equation}
Since $J$ is so close to $1$, we also get that
$|A_\mu - c(x,r) a_0 r^d| \leq C c(x,r) \alpha(x,r) r^d$. We now compare to $A_\sigma$,
use \eqref{a6.34}, and get that
\begin{equation} \label{a6.40}
|a_0 r^d \lambda(x,r) - c(x,r) a_0 r^d|
\leq |A_\sigma - A_\mu| + |A_\mu - c(x,r) a_0 r^d|
\leq C (1+c(x,r)) \alpha(x,r) r^d.
\end{equation}
Notice that if $c_1$ in \eqref{a6.22} is small enough,
\begin{equation} \label{a6.40b}
|\lambda(x,r) - c(x,r)| \leq \frac12 (1+c(x,r)).
\end{equation}
Observe also that since $\1_{B(\Phi(x),r/2)} \leq \theta_{x,r} \leq \1_{B(\Phi(x),r)}$,
\eqref{defADR} implies that $(r/2)^d C_\sigma^{-1} r^d \leq \int \theta_{x,r} d\sigma
\leq C_\sigma r^d$, and hence 
\begin{equation} \label{a6.41}
C^{-1} \leq \lambda(x,r) \leq C 
\ \text{ for $x\in \R^d$ and } r > 0.
\end{equation}
Thus \eqref{a6.40b} and \eqref{a6.41} forbid $c(x,r)$ to be to large, 
we may simplify \eqref{a6.40}, and we get that
\begin{equation} \label{a6.42}
|\lambda(x,r) - c(x,r)| \leq C \alpha(x,r).
\end{equation}
In particular, if $c_1$ is small enough (our second and last
condition on $c_1$), we also get that $C^{-1} \leq c(x,r) \leq C$.

We are now ready to prove \eqref{a6.21}. Let $f \in Lip(\Phi(x,r)$ be given.
By \eqref{a6.16} and the same change of variables as for \eqref{a6.38},
\begin{equation} \label{a6.43}
\int f(z) d\wt\mu_{x,r}(z)
= c(x,r) \int _{\Lambda(x,r)} f(z) d\mu_{\Lambda(x,r)}(z)
= c(x,r) J \int_{P(x,r)} f(\pi(y)) d\mu_{P(x,r)}(y),
\end{equation}
and since $|f(\pi(y))-f(y)| \leq |\pi(y)-y| \leq C r \alpha(x,r)$ by \eqref{a6.32}
and $|J-1| \leq C \alpha(x,r)$, we get that
\begin{equation} \label{a6.44}
\Big|\int f(z) d\wt\mu_{x,r}(z)
- c(x,r) \int_{P(x,r)} f(y) d\mu_{P(x,r)}(y)\Big|
\leq C r^{d+1} \alpha(x,r).
\end{equation}
Then \eqref{a6.42} allows us to replace $c(x,r)$ with $\lambda(x,r)$,
and since $\mu_{x,r} = \lambda(x,r) \mu_{P(x,r)}$ by \eqref{a6.19}, we are left with 
\begin{equation} \label{a6.45}
\Big|\int f d\wt\mu_{x,r} - \int_{P(x,r)} f d\mu_{x,r}\Big| \leq C r^{d+1} \alpha(x,r).
\end{equation}
We add this to \eqref{a6.15} and get that 
\begin{equation} \label{a6.46}
\Big|\int f d\sigma - \int_{P(x,r)} f d\mu_{x,r}\Big| \leq C r^{d+1} \alpha(x,r).
\end{equation}
Finally we take the supremum over $f \in Lip(\Phi(x),r)$ and get \eqref{a6.21}.
Lemma \ref{Lreplace} follows.
\end{proof}

So we decided to work with the measures $\mu_{x,r} = \lambda(x,r) \mu_{P(x,r)}$.
Some regularity of the coefficients $\lambda(x,r)$ will be helpful.

\begin{lemma} \label{CMfornablah}
There is a constant $C \geq 0$ such that
$|r\nabla_{x,r} \lambda(x,r)| \leq C \alpha(x,r)$ for $x\in \R^d$ and $r > 0$.
Hence $r\nabla_{x,r} \lambda(x,r)$ is uniformly bounded and satisfies the Carleson 
measure condition. 
\end{lemma}

\begin{proof}
Recall from \eqref{a6.18}, \eqref{a6.35}, and \eqref{a6.17} that
\begin{equation} \label{a6.48}
\lambda(x,r) = (a_0 r^d)^{-1} \int_\Gamma \theta_{x,r}(y) d\sigma(y)
= a_0^{-1} r^{-d} \int_\Gamma \theta\Big(\frac{\Phi_r(x)-y}{r} \Big) d\sigma(y).
\end{equation}
We start with the derivatives in the $x_i$ variables, which are easier to treat.
Denote by $\hat v^i_j$ the $j^{th}$ coordinate of $\hat v^i = \dr_{x_i} \Phi_r(x)$, and 
set $\theta_j = \frac{\dr \theta}{\dr j}$; then
\begin{equation} \label{a6.49} 
r \dr_{x_i} \lambda(x,r) = a_0^{-1} r^{-d} 
\int_\Gamma \sum_{j=1}^n \hat v^i_j \theta_j \Big(\frac{\Phi_r(x)-y}{r} \Big) d\sigma(y).
\end{equation}
It is rather easy to see that this is uniformly bounded, but what we really want
is the Carleson measure estimate.
Let us not stare at this formula for too long. If we started with 
\begin{equation} \label{a6.50}
\wt\lambda(x,r) = (a_0 r^d)^{-1} \int_{P(x,r)} \theta_{x,r}(y) d\mu_{x,r}(y)
= a_0^{-1} r^{-d} \int_{P(x,r)} \theta\Big(\frac{\Phi_r(x)-y}{r} \Big) d\mu_{x,r}(y),
\end{equation}
with just a change of measure, the same computation would yield
\begin{equation} \label{a6.51}
r \dr_{x_i} \wt\lambda(x,r) = a_0^{-1} r^{-d} 
\int_{P(x,r)} \sum_{j=1}^n \hat v^i_j \theta_j \Big(\frac{\Phi_r(x)-y}{r} \Big) d\mu_{x,r}(y).
\end{equation}
But $\wt\lambda(x,r)$ is a constant (by the usual dilation invariance computation, using
the fact that $\Phi_r(x) \in P(x,r)$), so the expression in \eqref{a6.51} vanishes,
and we can replace $d\sigma(y)$ with $d\sigma(y) - d\mu_{x,r}(y)$ in \eqref{a6.49}.

Set $\wt f(y) = \sum_{j=1}^n \hat v^i_j \theta_j \Big(\frac{\Phi_r(x)-y}{r} \Big)$, and 
notice that $\wt f$ is $Cr^{-1}$-Lipschitz and supported in $B(\Phi(x),2r)$ (by \eqref{a6.31}). 
Thus we may apply \eqref{a6.21} (or directly \eqref{a6.44}) to $f = C^{-1} r \wt f$,
and get that
\begin{equation} \label{a6.52}
\Big|\int f [d\sigma - d\mu_{x,r}]\Big| \leq C r^{d+1} \alpha(x,r).
\end{equation}
We multiply this by $Cr^{-1}$, replace in \eqref{a6.49}, and get that 
\begin{equation} \label{a6.53}
|r \dr_{x_i} \lambda(x,r)|  = a_0^{-1} r^{-d} 
\Big|\int \wt f [d\sigma-d\mu_{x,r}](y)\Big| \leq C  \alpha(x,r).
\end{equation}

\ms

The radial derivative is treated in the same way, except that the formula is more complicated.
That is,
\[
\begin{split} 
r \, &\frac{\d \lambda(x,r)}{\d r} 
= - d a_0^{-1} r^{-d} \int \theta\Big(\frac{\Phi_r(x)-y}{r} \Big) d\sigma(y)
\\
&\ \  + a_0^{-1} r^{-d} \int \Big\{\sum_j (\d_{r}\Phi_r(x))_j \theta_j \Big(\frac{\Phi_r(x)-y}{r} \Big)
- \frac{r}{r^2} \sum_j \theta_j \Big(\frac{\Phi_r(x)-y}{r} \Big) (\Phi_r(x)-y)_j
\Big\} d\sigma(y).
\end{split}
\]
This looks ugly, but all we need to know is that we can write this as
\begin{equation} \label{a6.54}
r \,\frac{\d \lambda(x,r)}{\d r}  = a_0^{-1} r^{-d} \int F\Big(\frac{\Phi_r(x)-y}{r} \Big) d\sigma(y),
\end{equation}
where here it happens that 
$F(w) = -d\theta(w) + \sum_j (\d_{r}\Phi_r(x))_j \theta_j(w) - \sum_j \theta_j(w) w_j$,
but the main point is that $F$ is $C$-Lipschitz 
(use Lemma \ref{rgradientphilem} to bound $\d_{r}\Phi_r(x)$) 
and supported in the unit ball.

As before, $r \, \frac{\d \lambda(x,r)}{\d r}$ is bounded by inspection,
and since $\wt \lambda(x,r)$ is a constant and we can do the same computation for it,
with $\sigma$ replaced by $\mu_{x,r}$, this allows us to replace $d\sigma$ by
$[d\sigma - d\mu_{x,r}]$ in the big integral, use Lemma \ref{Lreplace}, and get the 
same estimate as before.

Thus  $|r\nabla_{x,r} \lambda(x,r)| \leq C \alpha(x,r)$, and the desired Carleson measure estimate follows from \eqref{a6.13}.
\end{proof}

\subsection{The soft distance function $D$} 

Fix $\alpha > 0$ and let $D$ be the distance function defined by \eqref{IdefD}; 
we want to use the $\alpha(x,r)$ to control the variations of $D$.
Of course it will be simpler to study 
\begin{equation} \label{a6.55}
D(z)^{-\alpha} = \int_\Gamma |z-y|^{-d-\alpha} d\sigma(y),
\end{equation}
and we shall need the normalizing constant 
\begin{equation} \label{defcalpha}
c_\alpha = \int_{\R^d} (1+|x|^2)^{-\frac{d+\alpha}{2}} dx.
\end{equation}

\begin{lemma} \label{L6D}
 Let $C_0$ be small. 
 For each constant $C_2 \geq 1$, we can find $C_D \geq 1$, that depends on
$n$, $d$, $\eta$, $\theta$, $C_\sigma$, $\alpha$, and $C_2$, such that if
$x\in \R^d$, $r > 0$, and $z \in \R^n$ is such that 
\begin{equation} \label{a6.58}
|z-\Phi(x)| \leq C_2 r, \ \text{ and} \  \dist(z, \Gamma \cup P(x,r)) \geq C_2^{-1} r,
\end{equation}
then 
\begin{equation} \label{a6.59}
\big| D(z)^{-\alpha} - c_{\alpha} \lambda(x,r) \dist(z,P(x,r))^{-\alpha} \big|
\leq C_D  r^{-\alpha} a(x,r), 
\end{equation}
where we set $a(x,r) = \sum_{k \geq 0}  
2^{-\alpha k} \alpha(x,2^k r)$.
\end{lemma}

This is a good enough control, since we shall see in Lemma \ref{L6DC} that 
$a(x,r)$ satisfies the Carleson condition.
We will try not to create too much confusion between the exponent $\alpha$ and the
Tolsa numbers $\alpha(x,2^k r)$. 

\begin{proof}
Let $x$, $r$, and $z$ be given. We intend to cut the integral of \eqref{a6.55}
into pieces that live in annuli, so we introduce cut-off functions. We start
with $\theta_0$, which is defined on $\R^n$, radial, smooth, supported in 
$B(0,r/2)$, such that $0 \leq \theta_0$ everywhere and $\theta_0 = 1$ on
$B(0,r/4)$, and finally is $3r^{-1}$-Lipschitz. Then we set
\begin{equation} \label{a6.60}
\theta_k(y) = \theta_0(2^{-k}y) - \theta_0(2^{-k+1}y)
\end{equation}
for $k \geq 1$ and $y\in \R^n$, and translate all these functions by setting
$\wt \theta_k(y) = \theta_k(y-\Phi_r(x))$. Notice that
\begin{equation} \label{a6.61}
\sum_{k \geq 0} \wt \theta_k = 1
\end{equation}
(it is a telescopic sum), 
\begin{equation} \label{a6.62}
\wt\theta_k \text{ is supported in } B_k = B(\Phi_r(x),2^{k-1}r),
\end{equation}
and for $k \geq 1$
\begin{equation} \label{a6.63}
\wt\theta_k = 0 \text{ on } B_{k-2}.
\end{equation}
We use \eqref{a6.61} to write $D(z)^{-\alpha} = \sum_{k\geq0} I_k$, with 
\begin{equation} \label{a6.64}
I_k = \int_\Gamma |z-y|^{-d-\alpha} \wt\theta_k(y) d\sigma(y)
= \int_\Gamma \wt f_k(y) d\sigma(y),
\end{equation}
with 
\begin{equation} \label{a6.65}
\wt f_k(y) = |z-y|^{-d-\alpha} \wt\theta_k(y).
\end{equation}
Our intention is to approximate $I_k$ by
\begin{equation} \label{a6.66}
J_k = \int_{P(x,r)} \wt f_k(y) d\mu_0(y),
\end{equation}
where for convenience we put
\begin{equation} \label{a6.67}
\mu_k = \mu_{x,2^kr} \ \text{ for } k \geq 0.
\end{equation}
There is no problem with the definition, because $\dist(z,\Gamma \cup P(x,r)) > 0$.
It turns out that the sum is easy to compute. Indeed
\begin{equation} \label{a6.68}
\sum_{k\geq0} J_k = \int_{P(x,r)} |z-y|^{-d-\alpha} \big(\sum_k \wt\theta_k(y)\big) d\mu_0(y)
= \int_{P(x,r)} |z-y|^{-d-\alpha} d\mu_0(y);
\end{equation}
we can send $P(x,r)$ to $\R^d$ by an isometry $F$, so that $F(z) = (0,t)$ for some
$t\in \R^{n-d}$ such that $|t| = \dist(z,P(x,r))$; the image of $\mu_0$ is $\lambda(x,r)$ 
times the Lebesgue measure, so
\begin{eqnarray} \label{a6.69}
\sum_{k\geq0} J_k 
&=& \lambda(x,r)\int_{\R^d} |F(z)-y|^{-d-\alpha} dy
=  \lambda(x,r)\int_{\R^d} (|t|^2 + |y|^2)^{-(d+\alpha)/2} dy
\nn\\
&=& \lambda(x,r) |t|^{-\alpha} \int_{\R^d} (1 + |u|^2)^{-(d+\alpha)/2} du
= c_{\alpha} \lambda(x,r)  \dist(z,P(x,r))^{-\alpha}
\end{eqnarray}
by the change of variables $y = |t| u$. We recognize the same expression as in \eqref{a6.59};
thus we will just need to estimate $|I_k-J_k|$.

\ms
We intend to use Lemma \ref{Lreplace} to play with measures, so we are interested in
the Lipschitz properties of $\wt f_k$, and thus want lower bounds on $|y-z|$ when
$\wt\theta_k(y) \neq 0$. We claim that
\begin{equation} \label{a6.70}
|y-z| \geq c_3 2^{k} r
\ \text{ when $y\in \Gamma \cup P(x,r)$ is such that } \wt\theta_k(y) \neq 0,
\end{equation}
with $c_3 = (16(C_2+1)C_2)^{-1}$.

We start with the case $2^{k-4} \geq C_2+1$.
If $\wt\theta_k(y) \neq 0$, then \eqref{a6.63} says that $y\notin B_{k-2}$,
hence
\begin{eqnarray} \label{a6.71}
|y-z| &\geq& |y-\Phi_r(x)| -|\Phi_r(x)-\Phi(x)| - |\Phi(x)-z|
\nn\\
&\geq& 2^{k-3} r - C_0 r - C_2 r \geq 2^{k-3} r - (1+C_0) r \geq 2^{k-4} r
\end{eqnarray}
by \eqref{a6.31}, \eqref{a6.58}, and our assumption on $k$.

Now assume that $2^{k-4} < C_2+1$, and use the lower bound in \eqref{a6.58}
to get that
\begin{equation} \label{a6.72}
|y-z| \geq \dist(y,\Gamma \cup P(x,r)) \geq C_2^{-1} r 
\geq (16(C_2+1)C_2)^{-1} 2^k r,
\end{equation}
as needed for \eqref{a6.70}. Now set
\begin{equation} \label{a6.73}
f_k(y) = \max(|z-y|, c_3 2^k r)^{-d-\alpha} \wt\theta_k(y).
\end{equation}
We just proved that $f_k(y) = \wt f_k(y)$ for every $y\in \Gamma \cup P(x,r)$.
Since $\sigma$ and $\mu_0$ are supported in $\Gamma$ and $P(x,r)$ respectively,
we can replace $\wt f_k$ with $f_k$ in the definition of $I_k$ and $J_k$, and this
will be convenient because $f_k$ is Lipschitz. Indeed the gradient of 
$\max(|z-y|, c_3 2^k r)^{-d-\alpha}$ vanishes, unless $|z-y| \geq c_3 2^k r$
and then the gradient is the same as for $|z-y| ^{-d-\alpha}$. Thus
\begin{equation} \label{a6.74}
|\nabla f_k(y)| \leq (c_3 2^k r)^{-d-\alpha} |\nabla \wt\theta_k(y)|
+ (d+\alpha) (c_3 2^k r)^{-d-\alpha-1} |\wt\theta_k(y)|
\leq C (2^k r)^{-d-\alpha-1}
\end{equation}
because $|\nabla \wt\theta_k(y)| \leq 6 (2^k r)^{-1}$ by \eqref{a6.60}.
Notice also that
\begin{equation} \label{a6.75}
\|f_k\|_\infty \leq C (2^k r)^{-d-\alpha}
\end{equation}
directly by \eqref{a6.73}. Set
\begin{equation} \label{a6.76}
J'_k = \int_{P(x,2^kr)} f_k(y) d\mu_k(y).
\end{equation}
Notice that $f_k$ is supported in $B_k \subset B(\Phi(x),2^kr)$ (by \eqref{a6.31}), 
so $[C (2^k r)^{-d-\alpha-1}]^{-1} f_k \in Lip(\Phi(x),2^kr)$, and 
Lemma \ref{Lreplace} yields
\begin{eqnarray} \label{a6.77}
|I_k-J'_k| &=& \Big| \int f_k(y) [d\sigma - d\mu_k](y) \Big|
\leq [C (2^k r)^{-d-\alpha-1}] (2^kr)^{d+1} \dist_{\Phi(x),2^kr}(\sigma,\mu_k)
\nn\\
&=& C 2^{-k\alpha} r^{-\alpha} \dist_{\Phi(x),2^kr}(\sigma,\mu_k)
\leq C 2^{-k\alpha} r^{-\alpha} \alpha(x,2^kr)
\end{eqnarray}
by \eqref{a6.2}, Lemma \ref{Lreplace}, and \eqref{a6.67}. This is still compatible
with the right-hand side of \eqref{a6.59}, so we are left with 
$J'_k-J_k = \int f_k [d\mu_k - d\mu_0]$ to estimate, obviously only for $k \geq 1$ 
We write $J'_k-J_k = \sum_{1 \leq j \leq k} \delta_{j,k}$, with
\begin{equation} \label{a6.78}
\delta_{j,k} = \int f_k [d\mu_j - d\mu_{j-1}],
\end{equation}
and recall that $\mu_j = \mu_{x,2^j r} = \lambda(x,2^jr) \H_j^d$,
where we denote by $\H_j^d$ the restriction of $\H^d$ to the plane
$P_j = P(x,2^jr)$ to simplify the notation. We want to know that $P_{j-1}$
lies close to $P_j$, so we return to the distance estimate \eqref{a6.32}
in the proof of Lemma \ref{Lreplace}. We claim that the same proof also shows that
\begin{equation} \label{a6.79}
\dist(y,P_j) \leq C \alpha(x,2^jr) [2^jr+|y-\Phi_{2^jr}(x)|] 
\ \text{ for } y \in P_{j-1},
\end{equation}
at least if $\alpha(x,2^jr)$ is small enough, as in \eqref{a6.22}.
The logical way to see this is to observe that the proof of \eqref{a6.32}
for the position of $P_j$ also gives the same estimates for the position
of $P_{j-1}$, and then we compare the two. Another way is to observe that by 
Lemma \ref{Lreplace}, we can actually replace $\Lambda(x,2^jr)$ with $P_j$
in the whole proof of Lemma \ref{Lreplace}, and in particular of of \eqref{a6.32}.
This is formally more easy, except for the fact that it exactly works for $P_{j-2}$
rather than $P_{j-1}$ (because we lose a factor $4$ in the scales when we prove 
Lemma \ref{Lreplace}). This would be easy to fix, for instance by setting
$\alpha(x,r) = \wt\alpha_\sigma(\Phi(x),8r)$ in \eqref{a6.12} and proving a 
Lemma \ref{Lreplace} with $\dist(\Phi(x),2r)$. We leave the details and consider 
\eqref{a6.79} established. 

Let us continue the estimate when $\alpha(x,2^jr)$ is small enough, so that
$P_{j-1}$ and $P_j$ make a small angle, and we estimate
$\int f_k [d\H^d_j - d\H^d_{j-1}]$ with the same projection trick as 
near \eqref{a6.37}. Denote by $\pi$ the orthogonal projection from
$P_{j-1}$ to $P_j$; if $\alpha(x,2^jr)$ is small enough, this is an affine bijection,
and the determinant $J$ of its jacobian is such that $|J-1| \leq C \alpha(x,2^jr)$.
The change of variables $z = \pi(y)$ yields
\begin{equation} \label{a6.80}
\int_{P_j} f_k(z) d\H^d_j(z) = J \int_{P_{j-1}} f_k(\pi(y)) d\H^d_{j-1}(y)
\end{equation}
and hence, since $f_k$ is supported in $B_k$ by \eqref{a6.62},
\begin{eqnarray} \label{a6.81}
\Big| \int f_k [d\H^d_j - d\H^d_{j-1}]\Big|
&\leq& |J-1| \Big| \int_{P_{j-1}} f_k(\pi(y)) d\H^d_{j-1}(y)\Big|
+ \int_{P_{j-1}} |f_k(\pi(y))-f_k(y)| d\H^d_{j-1}(y)
\nn\\
&\leq& C \|f_k\|_\infty \H^d_{j-1}(B_k) \alpha(x,2^jr) + \|f_k\|_{Lip} 
\int_{P_{j-1} \cap B_{k+1}} |\pi(y)-y| d\H^d_{j-1}(y) 
\nn\\
&\leq& C \|f_k\|_\infty (2^kr)^d \alpha(x,2^jr) + C \|f_k\|_{Lip} (2^kr)^{d+1} \alpha(x,2^jr) 
\\
&\leq& C(2^kr)^{-\alpha} \alpha(x,2^jr)
\nn
\end{eqnarray}
by \eqref{a6.79}, \eqref{a6.74}, and \eqref{a6.75}.

When $\alpha(x,2^jr)$ is not small, simply observe that 
\begin{equation} \label{a6.82}
\Big| \int f_k [d\H^d_j - d\H^d_{j-1}]\Big| 
\leq \|f_k\|_\infty (\H^d_{j-1}(B_k) + \H^d_{j}(B_k))
\leq C (2^kr)^{-\alpha}
\end{equation}
by \eqref{a6.75} and with no gain, but \eqref{a6.81}
also holds in this case because $\alpha(x,2^jr) \geq C^{-1}$.

We also need to show that
\begin{equation} \label{a6.83}
|\lambda(x,2^jr) - \lambda(x,2^{j-1}r)| \leq C \alpha(x,2^jr),
\end{equation}
and the simplest at this point is to observe that for $2^{j-1}r \leq \rho \leq 2^jr$, 
$\rho |\frac{\d \lambda(x,\rho)}{\d \rho}| \leq C \alpha(x,\rho) \leq C 2^{d+1} \alpha(x,2^jr)$
by Lemma \ref{CMfornablah}, 
\eqref{a6.12}, and \eqref{a6.5}. We integrate and get \eqref{a6.83}. 
We are now ready for our last estimate. Indeed
\begin{eqnarray} \label{a6.84}
\delta_{j,k} &=& \Big|\int f_k [d\mu_j - d\mu_{j-1}] \Big|
= \Big|\int f_k [\lambda(x,2^jr) d\H_j^d-\lambda(x,2^{j-1}r) d\H_{j-1}^d] \Big|
\nn\\
&\leq& \lambda(x,2^jr) \Big|\int f_k [\H_j^d- \H_{j-1}^d]\Big|
+ |\lambda(x,2^jr)-\lambda(x,2^{j-1}r)| \Big|\int f_k d\H_{j-1}^d\Big|
\nn\\
&\leq& C (2^kr)^{-\alpha} \alpha(x,2^jr) +  C \alpha(x,2^jr) \|f_k\|_\infty \H_{j-1}(B_k)
\leq C (2^kr)^{-\alpha} \alpha(x,2^jr)
\end{eqnarray}
by \eqref{a6.82}, \eqref{a6.83}, \eqref{a6.41}, and \eqref{a6.75}.
We now sum over $j$ and get that 
\begin{equation} \label{a6.85}
|J'_k-J_k| \leq \sum_{1 \leq j \leq k} |\delta_{j,k}| 
\leq C \sum_{1 \leq j \leq k} (2^kr)^{-\alpha} \alpha(x,2^jr).
\end{equation}
Then we sum over $k$ and by Fubini's lemma,
\begin{eqnarray} \label{a6.86}
\sum_k |J'_k-J_k| &\leq& C \sum_k\sum_{j\leq k} (2^kr)^{-\alpha} \alpha(x,2^jr) \nn\\
&\leq& C r^{-\alpha} \sum_{j} \alpha(x,2^jr) \sum_{k\geq j} 2^{-\alpha k}  \leq C r^{-\alpha} \sum_{j} 2^{-\alpha j} \alpha(x,2^jr) ,
\end{eqnarray} 
which is compatible with \eqref{a6.59}. This last estimate completes our proof of 
Lemma \ref{L6D}.
\end{proof}

We complete Lemma \ref{L6D} with a Carleson control on the right-hand side.

\begin{lemma} \label{L6DC}
Let $a(x,r) = \sum_{k \geq 0} 2^{-\alpha k} \alpha(x,2^k r)$ 
be the same function as Lemma \ref{L6D}.
Then $a$ satisfies the Carleson condition.
\end{lemma}

\begin{proof}
We first apply Cauchy-Schwarz to estimate
\[
a(x,r)^2 \leq 
\Big\{\sum_{k \geq 0}  2^{-\alpha k} \alpha^2(x,2^k r) \Big\}
\Big\{\sum_{k \geq 0}   2^{-\alpha k} \Big\}
\leq C \sum_{k \geq 0}  2^{-\alpha k} \alpha^2(x,2^k r).
\]
We need to compute 
\begin{eqnarray} \label{a6.88}
\int_{0}^r \int_{B(x,r)} a(y,s)^2 \,\frac{dyds}{s}
&\leq& C \int_{0}^r \int_{B(x,r)} \sum_{k \geq 0}  2^{-\alpha k} 
\alpha^2(y,2^k s) \frac{dyds}{s}
\nn\\
&=& C  \sum_{k \geq 0}  2^{-\alpha k} \int_{0}^r\int_{B(x,r)} \alpha^2(y,2^k s) \frac{dyds}{s}.
\end{eqnarray}
Write
\[
\int_{0}^r\int_{B(x,r)} \alpha^2(y,2^k s) \frac{dyds}{s} = I_1 + I_2,
\]
where $I_1$ is the part of the integral where $0 < s < 2^{-k}r$, and $I_2$
is the rest. Notice that $\alpha^2(y,\rho) \leq C$ for $y\in \R^d$ and $\rho>0$,
by the definition \eqref{a6.12} and the remark above \eqref{a6.4}. Thus
\[
I_2 = \int_{2^{-k}r}^r\int_{B(x,r)} \alpha^2(y,2^k s) \frac{dyds}{s}
\leq C |B(x,r)| \int_{2^{-k}r}^r \frac{ds}{s} \leq C k r^d.
\]
We are left with 
\[
I_2 = \int_0^{2^{-k}r}\int_{B(x,r)} \alpha^2(y,2^k s) \frac{dyds}{s}
=  \int_{u=0}^{r}\int_{B(x,r)} \alpha^2(y,u) \frac{dydu}{u}
\leq C r^d
\]
after setting $u = 2^k s$ in the integral, and because $\alpha$
satisfies the Carleson condition. We return to \eqref{a6.88} and get that
\[
\int_{0}^r \int_{B(x,r)} a(y,s)^2 \,\frac{dyds}{s}
\leq C  \sum_{k \geq 0}  2^{-\alpha k} (k+1) r^d \leq C r^d,
\]
and Lemma \ref{L6DC} follows.
\end{proof}

\section{We collect estimates and conclude}
\label{Scollect}

The goal of this section is to put ourselves in the framework of Theorem~\ref{Itsf1},
at least for our soft distance $D_\alpha$. Then it will 
remain only to prove the two results on degenerate elliptic operators on
$\Omega_0 = \R^n \sm \R^d$ (namely, Theorems~\ref{Itsf1} and \ref{Itai1}) 
in order to establish Theorem~\ref{Main}.

That is, we choose the missing function $h$, use it to construct a change of variables $\rho$,
compute the matrix of the conjugated operator $L_0$ on $\Omega_0$, check that it satisfies
the assumptions of Theorems~\ref{Itsf1} and \ref{Itai1}, and conclude.

After this, we discuss the case of the usual distance in dimension $d=1$,
and rapidly discuss other choices of distance functions.

\subsection{Our choice of $h(x,t)$}

Recall that we still need to choose the function $h$ that comes in the definition
of $\rho$, and we should try to make $D(\rho(x,t))$ close to $|t|$. 
The discussion that follows shows that taking
\begin{equation} \label{a7.1}
h(x,t) = (c_\alpha \lambda(x,|t|))^{1/\alpha}
\end{equation}
is the most reasonable option. Thus we shall take $h$ radial. 
Now we have to be a little bit careful, because we need to check that $h$ satisfies 
the assumptions needed for the construction of $\rho$, 
and this is why we shall put an additional requirement, compared to the previous section. 

\begin{lemma}\label{L7.1}
There exist small constants $c_0(n,d)$, that depends only on $n$ and $d$,
and $c_1(n,d,\alpha)$, that depends only on $n$, $d$, and the exponent $\alpha$ 
in the definition of $D = D_\alpha$, with the following properties.

Let $\Gamma$ be a Lipschitz graph (as in \eqref{defGamma}), and let $C_0$
be as in \eqref{defC0}. Assume that $C_0 \leq c_0(n,d)$, and let $\sigma$
be an Ahlfors regular measure on $\Gamma$ (as in \eqref{defADR}) such that
\begin{equation} \label{a7.3}
\wt\alpha_\sigma(\Phi(x),r) \leq c_1(n,d,\alpha)
\ \text{ for $x\in \R^d$ and } r >0,
\end{equation}
where $\Phi(x) = (0,\varphi(x))$, and the numbers $\wt\alpha_\sigma(\Phi(x),r)$
are defined by \eqref{a6.3}. Choose $h$ as in \eqref{a7.1}. Then the assumptions
of the previous sections are satisfied, we can construct the bi-Lipschitz mapping $\rho$,
and the Carleson measure condition \eqref{DrhoisCM} holds.
\end{lemma}

Let us rapidly comment the statement before we prove it.
We do not mention $\eta$ or $\theta$ because we can choose them, once and for all.
Our assumption that $C_0$ be small enough is not new, and we use it very often in the
construction. As we shall see soon, the additional assumption \eqref{a7.3} is only used
to check \eqref{H2} for our choice of $h$; this makes sense because if $h$ varies too wildly,
$\rho$ is unlikely to be injective.

We claim that if ($C_0$ is small enough and) we choose $\sigma$
sufficiently close to the surface measure on $\Gamma$, i.e., if 
\begin{equation} \label{a7.4}
(1-c_2(n,d,\alpha)) \H^d_{\vert \Gamma} \leq \sigma
\leq (1+c_2(n,d,\alpha)) \H^d_{\vert \Gamma}
\end{equation}
for a small enough $c_2(n,d,\alpha)$, then $\sigma$ is Ahlfors regular and
\eqref{a7.3} holds. But of course \eqref{a7.3} may also hold for different reasons. 

To prove the claim, we test $\wt\alpha_\sigma(\Phi(x),r)$ on the flat measure 
$\mu = \H^d_{\vert P_x}$, where $P_x = \R^d + (0,\varphi(x))$. 
Denote by $\pi$ the orthogonal projection on $\R^d$, by $\wt \sigma$ the pushforward of 
$\sigma$ on $P'_x=\R^d$, and by $\wt\mu$ the pushforward of $\mu$ on $P'_x$. 
With our assumptions, it is easy to check that
\begin{equation} \label{a7.5}
(1-c_2(n,d,\alpha)) \wt\mu \leq \wt\sigma
\leq (1+c_2(n,d,\alpha)+c_0(n,d)) \wt\mu,
\end{equation}
and then, for any $f\in Lip(\Phi(x),r)$, 
\begin{equation} \label{a7.6}
\int_\Gamma f d\sigma - \int_{P_x} f d\mu
= \int_{\R^d} f(y,\varphi(y)) d\wt\sigma - \int_{\R^d} f(y,\varphi(0)) d\wt\mu = I+II,
\end{equation}
where
\begin{equation} \label{a7.7}
|I| = \Big|\int_{\R^d} [f(y,\varphi(y))-f(y,\varphi(0))] d\wt\sigma(y) \Big|
\leq \int_{\R^d \cap B(x,r)} |\varphi(y)-\varphi(x)| d\wt\sigma(y)
\leq C C_0 r^{d+1},
\end{equation}
and 
\begin{eqnarray} \label{a7.8}
|II| &=& \Big| \int_{\R^d} f(y,\varphi(0)) [d\wt\sigma(y)-d\wt\mu(y)] \Big|
\leq  \|f\|_\infty (c_2(n,d,\alpha)+c_0(n,d)) \int_{\R^d \cap B(x,r)} d\wt\mu(y)
\nn\\
&\leq& C (c_2(n,d,\alpha)+c_0(n,d)) r^{d+1}
\end{eqnarray}
by \eqref{a7.5} and \eqref{a6.4}. We compare with Definition \ref{D6.1},
get \eqref{a7.3}, and prove the claim.

\begin{proof}
Let us now prove the lemma. With our assumptions, we can define the 
planes $P(x,r)$, and prove the results of Section \ref{SCMforD}; 
in particular $C^{-1} \leq \lambda(x,r) \leq C$ (by \eqref{a6.41})
and Lemma~\ref{CMfornablah} says that 
$r |\nabla_{x,r}| \lambda(x,r) \leq C \alpha(x,r) = C \wt\alpha_\sigma(\Phi(x),4r)$
(by \eqref{a6.12}) and satisfies the Carleson measure condition.

The first information gives \eqref{H1}, and the second one yields \eqref{H2}
and \eqref{H3}. The smallness condition in \eqref{H2} really requires something
like \eqref{a7.3}, while the other conditions would follow from Section \ref{SCMforD}.

At this point we can construct $\rho$, and we aim for \eqref{DrhoisCM}.
Set $r=|t|$ as usual, and recall from \eqref{defrhobis} that
$\rho(x,t) = \Phi_r(x) + h(x,t) R_{x,r}(0,t)$, where $\Phi_r(x) \in P(x,t)$ and 
$R_{x,r}(0,t)$ is orthogonal to $P'(x,t)$. Hence by \eqref{a7.1}
\begin{equation} \label{a7.9}
\dist(\rho(x,t),P(x,r)) = h(x,t) r = (c_\alpha \lambda(x,r))^{1/\alpha} r.
\end{equation}
We also know from \eqref{controlofdist} and \eqref{controlofdist2} that 
\begin{equation} \label{a7.10}
|\rho(x,t)-\Phi(x)| \leq 2r h(x,t)r \text{ and } \dist(\rho(x,t),\Gamma) \geq h(x,t)r/2.
\end{equation}
So, if we take the constant $C_2$ in Lemma \ref{L6D} larger than twice the constant $C$
in  \eqref{H1}, $z=\rho(x,t)$ satisfies the assumption \eqref{a6.58} of that lemma,
and \eqref{a6.59} says that 
\[
\big| D(\rho(x,t))^{-\alpha} - c_{\alpha} \lambda(x,r) \dist(z,P(x,r))^{-\alpha} \big|
\leq C r^{-\alpha}\sum_{k \geq 0}  
2^{-\alpha k} \alpha(x,2^k r) = C r^{-\alpha} a(x,r),
\]
where $a$ is the Carleson function of Lemma \ref{L6DC}.
We use \eqref{a7.9} to replace $\dist(z,P(x,r))$ and get that
$|D(\rho(x,t))^{-\alpha} - r^{-\alpha}| \leq C r^{-\alpha} a(x,r)$,
or equivalently 
\begin{equation} \label{a7.11}
\Big| r^\alpha D(\rho(x,t))^{-\alpha} - 1| \leq C a(x,r).
\end{equation}
Recall from \eqref{equivD} that $C^{-1}\dist(z,\Gamma) \leq D(z) \leq C \dist(z,\Gamma)$.
Since $\dist(\rho(x,t),\Gamma)$ is also equivalent to $|t| = r$ by \eqref{a7.10}, we see that
\begin{equation} \label{a7.12}
C^{-1}|t| \leq D(\rho(x,t)) \leq C \dist(z,\Gamma) \leq C|t|.
\end{equation}
This allows us to apply the reciprocal of $s \to s^{\alpha}$, which is Lipschitz
in the range where $r^\alpha D(\rho(x,t))^{-\alpha}$ lives, to deduce from \eqref{a7.11}
that
\begin{equation} \label{a7.13}
\Big| \frac{r} {D(\rho(x,t))} - 1| \leq C a(x,r).
\end{equation}
The Carleson measure condition \eqref{DrhoisCM} now follows from Lemma \ref{L6DC}.
\end{proof} 

\subsection{Computations for the conjugated operator $L_0$}
\label{SS7.2}

Recall that we started with an operator $L$ on $\Omega = \R^n \sm \Gamma$,
formally defined as $L = - \div D^{d+1-n} \nabla$ (see \eqref{defL}), and where 
$D$ is either the soft distance given by \eqref{IdefD} or (in dimension $d=1$)
the usual Euclidean distance (see \eqref{a2.12}).

We have defined a change of variables $\rho : \Omega \to \Omega_0$,
and in this subsection we compute the conjugated operator of $L$ by $\rho$,
and then show that it satisfies the assumptions of theorems of the introduction.

For the rigorous definition of $L$, we refer to \cite{DFMprelim},
where definitions and solutions were given in terms of the weight $w$ defined
on $\Omega$ by
\begin{equation} \label{weightw}
w(z) = \dist(z,\Gamma)^{d+1-n},
\end{equation}
and an accretive bilinear form on the Sobolev space  
\begin{equation}
W(\Omega) = \left\{ u \in L^1_{loc}(\Omega), \, \int_\Omega |\nabla u|^2 w < +\infty \right\}.
\end{equation}
Recall that we checked in \eqref{equivD} that 
\begin{equation} \label{equivw}
C^{-1} D(z)^{d+1-n} \leq w(z) \leq C^{-1} D(z)^{d+1-n},
\end{equation}
where $C>0$ depends only on the dimensions $d$ and $n$ and the parameter $\alpha$.
In the next lemma we compute the effect of $\rho$ on that bilinear form.

\begin{lemma} \label{newLaplacianprop}
Let $\Gamma$, $L$, and $\rho$ be as above; in particular assume that $\rho$ is smooth
and bi-Lipschitz. Then, for $u,v\in W(\Omega)$,
\begin{equation} \label{a7.18}
\int_{X\in \Omega} \nabla u(X)\cdot \nabla v(X) \frac{dX}{D(X)^{n-d-1}} 
= \int_{(x,t) \in \Omega_0}  A_\rho(x,t) \nabla [u\circ  \rho](x,t)
\cdot \nabla [v\circ \rho](x,t) \, dx \, dt,
\end{equation} 
where $\Omega_0 = \R^d \times [\R^{n-d}\setminus \{0\}]$, 
\begin{equation} \label{a7.19}
A_\rho(x,t) = \left(\frac{1}{D(\rho(x,t))}\right)^{n-d-1}
|\det(J(x,t))| (J(x,t)^{-1})^T J(x,t)^{-1}
\end{equation}
for $(x,t) \in \Omega_0$, and $J$ is defined in Definition \ref{defJ}.
\end{lemma}

\begin{proof}
First of all, since $\rho$ is (smooth and) bi-Lipschitz,
$\wt u = u\circ \rho$ and $\wt v = v\circ \rho$ lie
in the Sobolev space $W(\Omega_0)$ associated to the domain $\Omega_0$
and the weight $|t|^{d+1-n}$. Thus the right-hand side makes sense because
$A_\rho(x,t)$ is bounded by $|t|^{d+1-n}$ (by \eqref{a7.12} 
and directly Definition \ref{defJ}).
We claim that
\begin{equation} \label{a7.20}
\nabla \wt u = \Jac \ ([\nabla u]\circ \rho)
\end{equation}
(the product of two matrices). Let us for a moment forget about the decomposition 
$\R^n = \R^d \times \R^{n-d}$, write $X = (x_1, \ldots x_n)$ for the generic point 
of $\R^n$, and denote by $\rho_\ell$ the $\ell^{th}$ component of $\rho$. 
Thus $\Jac_{k,\ell} = \dr_{x_k}\rho_\ell$
(which is coherent with our definition \eqref{a4.6}-\eqref{a4.7}). The $k^{th}$
component (line) of $\nabla \wt u$ is
\[
\dr_{x_k} [u \circ \rho] = \sum_{\ell} \dr_{x_k} \rho_\ell [\dr_{x_\ell} u]\circ \rho
= \sum_{\ell} \Jac_{k,\ell} ([\nabla u]\circ\rho)_\ell,
\]
which proves \eqref{a7.20}. The substitution rule yields
\begin{eqnarray} \label{a7.21}
\int_{Z\in \Omega}  \nabla u(Z)\cdot  \nabla v(Z)  \frac{dZ}{D(Z)^{n-d-1}}  
&=&  \int_{(x,t) \in \Omega_0}  
\left< [\nabla u]\circ \rho, [\nabla v]\circ \rho\right> |\det(\Jac)|  
\frac{dxdt}{(D\circ\rho)^{n-d-1}} 
\nn\\
&=& \int_{\Omega_0} \left< \Jac^{-1 }\nabla \wt u, \Jac^{-1 }\nabla \wt v\right> |\det(\Jac)| 
\frac{dx dt}{D(\rho(x,t))^{n-d-1}} .
\end{eqnarray}
But by \eqref{JisJacQ}, there exists an orthogonal matrix $Q$ such that $\Jac = JQ^{-1}$.
Thus $|\det(\Jac)| = |\det(J)|$, and (setting $U = \nabla \wt u$ and $V = \nabla \wt u$)
\[
\left< \Jac^{-1} U, \Jac^{-1} V\right> = \left< QJ^{-1} U, QJ^{-1} V\right>
= \left< J^{-1}U, J^{-1}V\right> = \left< (J^{-1})^TJ^{-1}U, V\right>,
\]
so that \eqref{a7.21} is the same as 
\[
\int_{Z\in \Omega}  \nabla u(Z)\cdot  \nabla v(Z)  \frac{dZ}{D(Z)^{n-d-1}}  
= \int_{\Omega_0} \left< (J^{-1})^T J^{-1 }\nabla \wt u, \nabla \wt v\right> |\det(J)| 
\frac{dx dt}{D(\rho(x,t))^{n-d-1}};
\]
the lemma follows.
\end{proof}

We now check that the matrix $A_\rho(x,t)$ given by \eqref{a7.19} has an appropriate
decomposition, in fact a little stronger than the one required for
Theorems \ref{Itsf1} and \ref{Itai1}.

\begin{lemma} \label{MainS1}
Let $\Gamma$, $L$, and $\rho$ satisfy the assumptions above, let $A_\rho(x,t)$ be as in
\eqref{a7.19}, and let 
\begin{equation} \label{a7.23}
\A(x,t) = |t|^{n-d-1} A_\rho(x,t) 
=  \left(\frac{|t|}{D(\rho(x,t))}\right)^{n-d-1}
|\det(J(x,t))| (J(x,t)^{-1})^T J(x,t)^{-1}
\end{equation}
be the corresponding normalized $n \times n$ matrix. Then 
$\A$ is uniformly bounded and elliptic (or equivalently, $A_\rho$ satisfies 
\eqref{1.2.1a} and \eqref{1.2.2a}), and we can write
\begin{equation} \label{a7.24}
\A(X)= \left( \begin{array}{cc}
\A^1(X) + \C^1 & \C^2(X) \\ \C^3(X) & b(X)I_{n-d}+\C^4(X)  \end{array} \right), 
\end{equation}
where $\A^1(X), \C^1(X) \in M_{d\times d}(\R)$ are $d \times d$ matrices, 
$\C^2(X) \in M_{d\times (n-d)}(\R)$, $\C^3(X) \in M_{(n-d)\times d}(\R)$,
$b$ is a function on $\Omega_0$, $I_{n-d}$ is the identity matrix on $\R^{n-d}$, 
$\C^4(X) \in M_{d\times d}(\R)$, 
\begin{equation} \label{a7.25}
\C^1, \C^2, \C^3, \C^4, \text{ and } |t|\nabla_{x,t} \A^1 
\text{ satisfy the Carleson measure condition,}
\end{equation}
\begin{equation} \label{a7.26}
C^{-1} \leq b \leq C \ \text{ on } \Omega_0,
\end{equation}
and
\begin{equation} \label{a7.27}
|t|\nabla b \text{ satisfies the Carleson measure condition.}
\end{equation}
\end{lemma}

\ms
These are the same conditions as for Theorem \ref{Itsf1}, 
except that we also give a decomposition of the upper left block into a Carleson piece 
and a smooth piece, and that the upper right block satisfies a Carleson measure condition.
The conditions for Theorem \ref{Itai1} (regarding $L_0$ and its matrix) are even weaker.

Thus $\A^1+\C^1$ is uniformly bounded and elliptic (because $\A$ is),
and of course the near-diagonal form of the lower right block, which comes from 
the fact that $\rho$ is nearly isometric
the $t$-variables, is important. The constant $C$ (in \eqref{a7.26} and implicit in 
the Carleson conditions) depends only on $n$ and $d$ (given that we already 
forced $C_0$ and $c_1(n,d,\alpha)$ to be small).

\begin{proof}
First observe that $\A(x,t) =  \left(\frac{|t|}{D(\rho(x,t))}\right)^{n-d-1} \Ab(x,t)$,
where $\Ab(x,t)$ is the matrix of \eqref{a5.2}, and for which Lemma \ref{main1.3}
gives a nice description. Set $f(x,t) =  \left(\frac{|t|}{D(\rho(x,t))}\right)^{n-d-1}$
to save space. We know that 
\begin{equation} \label{a7.28}
C^{-1} \leq f \leq C
\end{equation}
(by \eqref{a7.12}), and that 
\begin{equation} \label{a7.29}
\text{$|f-1|$ satisfies the Carleson measure condition,}
\end{equation}
by \eqref{DrhoisCM} (and because \eqref{a7.12} allows us to take the $(n-d-1)$-th power.
We start from the decomposition $\Ab = \Ab^1 + \Ab^2$ given by Lemma \ref{main1.3},
which gives us the expression 
\begin{equation} \label{a7.30}
\A = f \Ab = \Ab^1 + (f-1) \Ab^1+ f\Ab^2.
\end{equation}
By Lemma \ref{main1.3}, $\Ab^2$ satisfies the Carleson measure condition,
hence also $(f-1) \Ab^1+ f\Ab^2$, by \eqref{a7.28} and \eqref{a7.29}. 
And we also have 
$\Ab^1 = \begin{pmatrix} \Ab^1_1 & 0 \\ 0 & \bb I_{n-d} \end{pmatrix}$
for some function $\bb$. 

Let us write $(f-1) \Ab^1 + f\Ab^2 = \begin{pmatrix} \C^1 & \C^2 \\ \C^3 & \C^4 \end{pmatrix}$;
This gives a decomposition of $\A$ as in \eqref{a7.24}, where we just need to take
\begin{equation} \label{a7.31}
b = \bb \ \text{ and } \A^1 = \Ab^1_1,
\end{equation}
where $\Ab^1_1$ is the upper left block of $ \Ab^1$.

The Carleson property \eqref{a7.25} holds, by definition of the $\C^j$ and because 
Lemma \ref{main1.3} says that $|t|\nabla \Ab^1\in CM$.
The uniform bound \eqref{a7.26} follows from \eqref{a5.5}. 
We still need to check \eqref{a7.27}, i.e., that $|t|\nabla b \in CM$. 
Recall that 
\begin{equation} \label{a7.32}
b = \bb = h^{n-d-2} \det(J'_1)
\end{equation}
by \eqref{a7.31} and \eqref{a5.7}. Observe that $h^{n-d-2}$ and $\det(J'_1)$
are both bounded (see \eqref{H1} concerning $h$). It was proved in (iii) of Lemma \ref{lemCM2}
that $|t|\nabla_{x,t} \det(J'_1) \in CM$, and $|t|\nabla h^{n-d-2} \in CM$ because 
$|t|\nabla_{x,r}h \in CM$ (by \eqref{H2}) and $h^{-1}$ is bounded. This proves \eqref{a7.27},
and Lemma \ref{MainS1} follows.
\end{proof}

At this point we completed the proof of Theorem \ref{Main}, with our soft distance 
$D_\alpha$, and modulo the two results on degenerate elliptic operators that will be 
treated in the last sections.

\subsection{Other distance functions, Euclidean distance}
We pulled out in \eqref{IdefD} one formula for a distance function that works for our purpose,
and looks reasonably natural. In this subsection, we give a sufficient condition, on a possibly
different distance function $D$ on $\Omega$, for Theorem \ref{Main} to stay true when
$L$ is defined with this distance. This will include the special case of the Euclidean distance to 
$\Gamma$, but only when $d=1$.

The sufficient condition that we give now is just chosen so that the proof above works.
We are still given a Lipschitz graph $\Gamma$, with small enough constant $C_0$, 
and a function $D : \Omega \to (0,+\infty)$, such that
\begin{equation} \label{a7.33}
C_4^{-1} \dist(z,\Gamma) \leq D(z) \leq C_4 \dist(z,\Gamma)
\ \text{ for } z\in \Omega
\end{equation}
for some constant $C_4 \geq 1$.
We pick a function $\eta$ as in Section \ref{Sorthbasis}, and use it to define
$\Phi_r$ and the approximate tangent $d$-plane $P(x,r)$. We assume that we have 
the following slightly weaker analogue of Lemma \ref{L6D}. We can find a function $a$
defined on $\R^d \times (0,+\infty)$ and such that
\begin{equation} \label{a7.34}
a \in CM(C_5), 
\end{equation}
and a function $\wt\lambda$, defined on $\Omega_0$ such that
\begin{equation} \label{a7.35}
C_4^{-1} \leq \wt\lambda \leq C_4
\end{equation}
\begin{equation} \label{a7.36}
\|\nabla \wt \lambda\|_\infty \leq \epsilon,
\end{equation}
and
\begin{equation} \label{a7.37}
\nabla \wt \lambda \in CM(C_5), 
\end{equation}
with the following property. For $x\in \R^d$, $r > 0$, and $z\in \Omega$ such that
\begin{equation} \label{a7.38}
C_4^{-1} r \leq |z-\Phi_r(x)| \leq C_4 r \ \text{ and } z-\Phi_r(x) \perp P(x,r),
\end{equation}
we have
\begin{equation} \label{a7.39}
\Big| \frac{D(z)}{\dist(z,P(x,r))} - \wt\lambda(x,r)\Big| =
\Big| \frac{D(z)}{|z-\Phi_r(x)|} - \wt\lambda(x,r)\Big|  \leq a(x,r).
\end{equation}

Of course it looks a little unpleasant that the condition depends on our construction 
of $P(x,r)$, but this is not so complicated. The reader should not pay too much attention to
the names of $C_4$ and $C_5$. We put the same constant $C_4$ in \eqref{a7.33} and 
\eqref{a7.35} because we think they may be proved at the same time (and they look very similar),
and gave a different name to $C_5$, mostly for psychological reasons, because in practice they 
will probably depend on $C_4$.

\begin{theorem} \label{T40}
For each choice of $C_4, C_5 \geq 1$, we can find $C_0 > 0$ and $\epsilon > 0$,
depending on $n$, $d$, our choice of $\eta$, $C_4$, and $C_5$, such that if 
$\Gamma$ is a Lipschitz graph with Lipschitz constant less than $C_0$, and $D$
satisfies the assumptions above, then the operator $L$ defined by \eqref{defL} with this function 
$D$ satisfies the conclusion of Theorem \ref{Main}.
\end{theorem}

\begin{proof}
We prove this first, and then comment more.
We shall just need to modify slightly the proof of Lemma \ref{L7.1}. 
This time, we take $h(x,t) = \wt \lambda(x,t)^{-1}$; the assumptions \eqref{H1}, \eqref{H2},
and \eqref{H3} follow from \eqref{a7.35}-\eqref{a7.37}, and by taking $\epsilon$ small
we can ensure that $C_{h0}$ in \eqref{H2} is as small as we want (depending on the other constants).
Since we also assume $C_0$ to be small enough, we can construct the bi-Lipschitz change of variable
$\rho$.

Now we want to show that \eqref{DrhoisCM} holds.
Let $(x,t) \in \Omega_0$ be given, set $r = |t|$, recall that 
$\rho(x,t) = \Phi_r(x) + h(x,t) R_{x,r}(t)$ hence (as in \eqref{a7.9})
\begin{equation} \label{a7.41}
\dist(\rho(x,t),P(x,r)) = h(x,t) r = (\wt\lambda(x,t))^{-1} r.
\end{equation}
Now consider $z = \rho(x,t)$, observe that $z-\Phi_r(x) \perp P'(x,r)$ 
by definition of $R_{x,r}$, and $C_4^{-1} r \leq |z-\Phi_r(x)| \leq C_4 r$
by \eqref{a7.35}. Hence \eqref{a7.38} holds 
and we have \eqref{a7.39}. Then, by \eqref{a7.41},
\begin{equation} \label{a7.42}
\Big| \frac{D(z)}{(\wt\lambda(x,t))^{-1} r} - \wt\lambda(x,r)\Big| =
\Big| \frac{D(z)}{\dist(z,P(x,r))} - \wt\lambda(x,r)\Big|  \leq a(x,r).
\end{equation}
Because of \eqref{a7.35}, this implies that 
$|r^{-1} D(\rho(x,t)) - 1| \leq C a(x,r)$, and then \eqref{DrhoisCM} follows,
because \eqref{a7.33} allows us to take inverses.

This was our analogue of Lemma \ref{L7.1};
we may now continue the argument as in Subsection\ref{SS7.2}, where 
the specific form of $D$ was not used, and conclude as before.
\end{proof}

We proved in Section \ref{SCMforD} that the functions $D_\alpha$ satisfy
the assumptions of Theorem \ref{T40}, and this reflects nice regularity properties 
of our Lipschitz graph $\Gamma$, as well as $D_\alpha$ itself.
We may use this work to prove that some other function $D$ works as well,
by controlling $D_{\alpha}^{-1} D$. That is, if $D$ is equivalent to
$D_\alpha$ (and the Euclidean distance to $\Gamma$) as in \eqref{a7.33},
and if we control $\frac{D}{D_\alpha}$ a little bit like we controlled 
$\frac{D_\alpha}{\dist(z,P(x,r))}$ above, then we may be able to prove more easily that
$D$ satisfies the conditions above.

We could also observe that we could also define $D_\rho$ by the fact that 
$D_\rho(\rho(x,t)) = |t|$. This defines a perfect function $D_\rho$ for the conditions above 
(with $\wt \lambda = 1$ and $a = 0$), but the reader may have thought that this was a 
very special choice, designed for the change of variables to work. Nonetheless,
we could also say that the more natural $D_\alpha$ work because they are close to $D_\rho$.

\ms
We end this section with the case of the Euclidean distance function, 
defined as in \eqref{a2.12} by 
\begin{equation} \label{a7.43}
D_E(z) = \dist(z,\Gamma) \ \text{ for } z\in \Omega.
\end{equation}

\ms
\begin{corollary} \label{C44}
There exists $C_0 > 0$ such that if $\Gamma \subset \R^n$ is a one-dimensional
Lipschitz graph, with Lipschitz constant at most $C_0$, then the operator $L$ 
defined by \eqref{defL} with this function $D_E$ of \eqref{a7.43}
satisfies the conclusion of Theorem \ref{Main}.
\end{corollary}

\begin{proof}
We could prove this as a consequence of Theorem \ref{T40}, which is a little unfair because
some things are actually simpler in dimension $1$, so we recall the main steps anyway.

As in the previous case, the main point is to compare $\dist(z,\Gamma)$ and 
$\dist(z,P(x,r))$ for some $z$ near $x$, and we shall use the function-theoretic analogue
of the P. Jones $\beta$-numbers. The point is to measure how well our small Lipschitz 
function $\varphi$ is approximated by affine functions. 

Denote by $\mathfrak A$ the set of affine functions $\mathfrak a : \R \to \R^{n-1}$, 
and set
\begin{equation} \label{betainfty}
\beta(x,r) = \frac1r \inf_{\mathfrak a \in \mathfrak A} 
\sup_{y\in B(x,r)} |\varphi(y) - \mathfrak a(y)|
\end{equation}
for $x\in \R$ and $r > 0$. We divide by $r$ to get a dimensionless number,
which is clearly bounded by $C_0$ (try $\mathfrak a(y) = \varphi(x)$).
The name comes from a paper of P. Jones \cite{Jones}, where related numbers
were used to quantify the distance from a set (like $\Gamma$) to lines or planes.
A result of Dorronsoro \cite{Dorronsoro} says that
\begin{equation} \label{a7.46}
\beta \in CM(CC_0^2). 
\end{equation}
The reader may view it as a consequence of the so-called geometric lemma of P. Jones
(which is also valid for more general sets), but it is in fact anterior. 
Also, unfortunately for us, it is only valid in dimension $1$.

For our purpose, it is more convenient to use the following variant, where the approximating
function $\mathfrak a$ is computed directly from $\varphi$ as a reasonable guess.
Let $\eta$ and $\eta_r$ be as in Section \ref{Sorthbasis}, define an affine function
$\mathfrak a_{x,r}$ by
\begin{equation} \label{a7.47}
\mathfrak a_{x,r}(y) = \varphi_r(x) + (y-x)\nabla_x\varphi_r(x),
\end{equation}
where in the present case $\nabla_x$ is just a derivative, and then set
\begin{equation} \label{betaetai}
\beta_\eta (x,r) = \frac1r  \sup_{y\in B(x,r)} |\varphi(y) - \mathfrak a_{x,r}(y)|
\end{equation}
for $x\in \R$ and $r > 0$.
Of course $\beta_\eta (x,r) \geq \beta(x,r)$ because $\mathfrak a_{x,r} \in \mathfrak A$, but 
$\beta_\eta (x,r)$ is not much larger in general.

\begin{lemma} \label{CMforbetaieta}
There exists a constant $C>0$, that depends only on $\eta$ and $n$,
such that for $x\in \R^d$ and $r>0$,
\begin{equation} \label{betaetaibybetai}
\beta_\eta (x,r) \leq C \beta(x,r). 
\end{equation}
As a consequence, 
\begin{equation} \label{a7.51}
\beta_\eta (x,r) \leq CC_0 \ \text{ for $x\in \R^d$ and $r>0$,} 
\end{equation}
and (by \eqref{a7.46})
\begin{equation} \label{a7.52}
\beta_\eta \in CM(CC_0^2).
\end{equation}
\end{lemma}

\ms
\begin{proof}
We only need to prove \eqref{betaetaibybetai}. This would also be true in higher dimensions, with
the same proof, but \eqref{a7.46} is not in general.
Let $x\in \R^d$ and $r>0$ be given, and pick an affine function $\mathfrak a$
such that 
\begin{equation} \label{a7.53}
\sup_{y\in B(x,r)} |\varphi(y) - \mathfrak a(y)| \leq 2r\beta(x,r).
\end{equation}
Since $\mathfrak a$ is affine, we may write $\mathfrak a(y) = (y-x)a+b$; then
\begin{eqnarray} \label{a7.54}
\beta_\eta(x,r) &=& r^{-1} \sup_{y\in B(x,r)} |\varphi(y) - \mathfrak a_{x,r}(y)|
\leq 2\beta(x,r) + r^{-1} \sup_{y\in B(x,r)} |\mathfrak a(y) - \mathfrak a_{x,r}(y)|
\nn\\
&\leq& 2\beta(x,r) + |a-\nabla_x\varphi_r(x)| +  r^{-1} |b - \varphi_r(x)|,
\end{eqnarray}
where we simply wrote the two affine functions and subtracted term by term.
We now estimate the two terms. First, we use the fact that $\eta_r$ is even with
integral $1$ to prove that
\begin{eqnarray} \label{a7.55}
|b-\varphi_r(x)| &=&  \Big| b - \int_\R \eta_r(x-y) \varphi(y) dy \Big|
= \Big| \int_{\R} \eta_r(x-y) [\varphi(y)-b-(y-x) a] dy \Big|
\nn\\
&\leq& \sup_{y\in B(x,r)} |\varphi(y)-b-(y-x) a| \, \int_{\R} \eta_r(x-y) dy
\leq 2 r \beta(x,r)  
\end{eqnarray}
because $\eta$ is supported in $B(x,r)$ and by \eqref{a7.53}. 
Similarly, observe that $\int_{\R} \nabla_x \eta_r(x-y)dy =0$ 
(because $\eta_r$ is compactly supported) and $\int_{\R} \nabla\eta_r(x-y) \cdot (y-x)dy$ 
is the identity matrix on $\R^n$  (integrate by parts), so
\begin{eqnarray} \label{a7.56}
|a-\nabla_x\varphi_r(x)| &=& \Big| a - \int_{\R} \nabla_x \eta_r(x-y) \varphi(y) dy \Big|
= \Big| \int_{\R} \nabla_x \eta_r(x-y) (b+(y-x)a-\varphi(y) dy) \Big| 
\nn\\
&\leq&  2r \|\nabla_x \eta_r\|_\infty  \sup_{y\in B(x,r)} |b+(y-x)a-\varphi(y)| 
\leq C _\eta \beta(x,r).
\end{eqnarray}
Lemma \ref{CMforbetaieta} follows from \eqref{a7.54}-\eqref{a7.56}.
\end{proof}

We now proceed as with the preceding results, but with $h = 1$ (so \eqref{H1}, \eqref{H2}, 
and \eqref{H3} hold trivially). Thus $\rho(x,t) = \Phi_r(x) + R_{x,t}(0, t)$ is a bi-Lipschitz mapping
(as usual, if $C_0$ is small enough) and (as in \eqref{a7.9} or \eqref{a7.41}), 
\begin{equation} \label{a7.57}
\dist(\rho(x,t),P(x,r)) = |\rho(x,t)-\Phi_r(x)| = |t| = r.
\end{equation}
We want to compare this with $\dist(\rho(x,t),\Gamma)$, and more precisely show that
\begin{equation} \label{a7.58}
(1-\beta_\eta(x,r)) r \leq \dist(\rho(x,t),\Gamma) \leq (1+\beta_\eta(x,r)) r.
\end{equation}
First observe that
\[\begin{split}
\dist(\rho(x,t),\Gamma) &\leq |\rho(x,t)-\Phi_r(x)| + \dist(\Phi_r(x),\Gamma)
\leq r + |\Phi_r(x)-\Phi(x)| 
\\
&= r + |\varphi_r(x)-\varphi(x)| = r + |\mathfrak a_{x,r}(x)-\varphi(x)|  \leq r + r \beta_\eta(x,r), 
\end{split}\]
which proves the second inequality. Write $\rho(x,r) = (y,s)$, with $y\in \R$ and $s\in \R^{n-1}$. Observe that since $P(x,r)$ is almost horizontal (because $\varphi_r$ too is $C_0$-Lipschitz),
so $|y-x| \leq 2C_0 r$ (recall that $\Phi_r(x) = (x,\varphi_r(x))$, and now the closest point(s)
of $\Gamma$ to $\rho(x,r)$ must be of the form $\xi = (z,\varphi(z))$, with $|z-y| \leq 2C_0 r$
too. Thus $|z-x| \leq 4 C_0 r < r$, and $|\varphi(z)-\mathfrak a_{x,r}(z)| \leq r \beta_\eta(x,r)$.
Hence 
\[\begin{split}
r &= \dist(\rho(x,t),P(x,r)) \leq |\rho(x,t)-(z,\mathfrak a_{x,r}(z))|
\leq |\rho(x,t)-\xi|+|\xi - (z,\mathfrak a_{x,r}(z))| 
\\& = \dist(\rho(x,t),\Gamma) + |\varphi(z)-\mathfrak a_{x,r}(z)| 
\leq \dist(\rho(x,t),\Gamma) + r \beta_\eta(x,r),
\end{split}\]
in particular because $P(x,r)$ is the graph of $\mathfrak a_{x,r}$. 
This completes our proof of \eqref{a7.58}, which itself implies that
\begin{equation} \label{distrhobybetai}
\left|\frac{|t|^{n-d-1}}{\dist(\rho(x,t),\Gamma)^{n-d-1}} - 1 \right| \leq 2\beta_\eta(x,r)
\end{equation}
by \eqref{a7.51}, since $C_0$ is small. 
Since $\beta_\eta \in CM$ by \eqref{a7.52}, this proves the crucial Carleson property
\eqref{DrhoisCM} for $D = D_E$.

At this point, we can follow the same route as above, i.e., use 
Lemma \ref{newLaplacianprop} to compute the matrix of $L_0$, then Lemma \ref{main1.3}
to put this matrix in the appropriate form, and finally observe that the assumptions of
Theorems \ref{Itsf1} and \ref{Itai1} are satisfied. This completes our proof of Corollary \ref{C44}
modulo Theorems \ref{Itsf1} and \ref{Itai1}. 
\end{proof}

\section{Square function estimates}
\label{SSQR}

In this section we prove
Theorem \ref{Itsf1}. Since we want
the two last sections to be independent of the previous ones,
we recall (or slightly modify) some of the notation.

\medskip
Throughout this section, $\Omega_0=\rn\setminus \Gamma_0$ with 
$\Gamma_0=\rd \subset \R^n$ (with a small abuse of notation). 
We write $X=(x, t)\in \RR^d\times \RR^{n-d}$ for points in $\overline {\Omega_0}=\rn$, 
and similarly $Y=(y,s)$.

For $x\in \R^d$, let 
$$
\gamma(x) :=\{Y=(y,s)\in \Omega_0 :\,|y-x|<a |s|\} 
$$
be the non-tangential cone of aperture $a>0$ (it actually looks like a rotated cone, but we will keep referring to it simply as a cone throughout the discussion). 
Unless otherwise stated, the estimates hold for all $a>0$ (fixed throughout a given theorem) 
and constants can depend on $a$. 
Also define the truncated cone $\gamma^h(x)$, $h> 0$, by 
\begin{equation}\label{defcone}
\gamma^h(x) 
:=\{Y=(y,s)\in \Omega_0:\,|y-x|<a |s|, \,0<|s|<h\}. 
\end{equation}
We often write $\gamma^Q$ in place of $\gamma^{l(Q)}$ when $h=l(Q)$, the side-length of some cube $Q$ (see \eqref{a2.27}). 

Going further, we denote the balls in $\rn$
by $B_r(X)=B(X, r) =\{Y\in \rn: \,|X-Y|<r\}$, $X\in \rn$, $r>0$, 
and denote the boundary balls by $\Delta_r(x):=\{y\in \rd:\, |x-y|<r\}$, $x\in \rd$, $r>0$.
Sometimes we write $B(x,t; r)$ in place of $B(X,r)$ for $X=(x,t)$. 
The tent region (which in our case looks like 
a punctured ball) is defined by
$$T(\Delta_r(x)):=B(x,0;r)\cap \Omega_0. $$
Recall the definition of the Carleson measure condition given in 
Definition \ref{defCMI}. 

\begin{definition} \label{defCM8}
We say that a function $u$ defined on 
$\Omega_0$
satisfies the Carleson measure condition (in short, $u\in CM$) if 
\[|u(y,s)|^2 \, \frac{dy \, ds}{|s|^{n-d}}\]
is a Carleson measure, that is, if 
\[\sup_{{\rm balls }\,\Delta\subset \R^d}\frac{1}{|\Delta|}\dint_{T(\Delta)}
|u(y,s)|^2\,\frac{dyds}{|s|^{n-d}} < +\infty.\]
\end{definition}

A matrix-valued function $\A$ satisfies the unweighted elliptic and bounded conditions if 
there exists $C_1>0$ such that
\begin{equation} \label{ABounded}
|\A(X)\xi \cdot \zeta| \leq C_1 |\xi| |\zeta| 
\qquad \text{ for $X \in \Omega_0$ and } \xi,\zeta\in \R^n,
\end{equation} 
and
\begin{equation} \label{AElliptic}
|\A(X)\xi \cdot \xi| \geq C_1^{-1} |\xi|^2 \qquad 
\qquad \text{ for $X \in \Omega_0$ and } \xi \in \R^n.
\end{equation}
Notice that if $\A$ satisfies \eqref{ABounded}--\eqref{AElliptic}, 
then the matrix $\Ak = |t|^{d+1-n}\A$ satisfies  \eqref{1.2.1}--\eqref{1.2.2} for the boundary 
$\Gamma_0$ and the domain $\Omega_0$.
 We are interested in the operator
\begin{equation} \label{a8.5}
L_0 = \diver |t|^{d+1-n} \A \nabla = \diver \Ak \nabla,
\end{equation}
where $\A$ satisfies the unweighted elliptic and bounded conditions (and some more).

We shall state the main result of this section in terms of weak solutions of
$L_0 u =0$, which we define now. Denote by $W^{1,2}_{loc}(\Omega_0)$
the set of functions $u\in L^2_{loc}(\Omega_0)$ whose derivative 
(in the sense of distribution on $\Omega_0$) also lies in $L^2_{loc}(\Omega_0)$.
A function $u \in  W^{1,2}_{loc}(\Omega_0)$ 
is called a {\bf weak solution} of $L_0 u =0$
if for $\varphi \in C^\infty_0(\Omega_0)$,
\begin{equation} \label{weaksol}
\dint_{(x,t) \in \Omega_0} \A \nabla u \cdot \nabla \varphi \, \frac{dx dt}{|t|^{n-d-1}} = 0.
\end{equation}

\ms
The heroes of this section are the following four functions, defined on $\R^d$
(but maybe infinite) for $u \in W_{loc}^{1,2}(\Omega_0)$, 
namely the non-tangential maximal function $Nu$ and its truncated version $N^Qu$, given by 
\begin{equation} \label{a8.7}
Nu(x)=\sup_{Y \in \gamma(x)} 
|u(Y)|
\ \text{ and } \ 
N^Qu(x)=\sup_{Y \in \gamma^Q(x)}  
|u(Y)|
\end{equation}
for $x\in \R^d$, the square function $Su$, defined by 
\begin{equation} \label{a8.8}
Su(x)=\left(\dint_{\gamma(x)} 
|\nabla u(Y)|^2 \frac{dY}{|Y-(x,0)|^{n-2}}\right)^{1/2},
\end{equation}
and its truncated version defined by
\begin{equation} \label{a8.9}
S^Qu(x)=\left(\dint_{\gamma^Q(x)} 
|\nabla u(Y)|^2 \frac{dY}{|Y-(x,0)|^{n-2}}\right)^{1/2}.
\end{equation}

Theorem~\ref{tsf1} below is just a restatement of Theorem~\ref{Itsf1}.

\begin{theorem}\label{tsf1} 
Let $\A$ be an elliptic matrix satisfying \eqref{ABounded}--\eqref{AElliptic}.  Assume that $\A$ has the following structure: 
\begin{equation}\label{eqsf1A}
\A=\left( \begin{array}{cc}
 \A^1 & \A^2 \\ \C^3 & bI_{(n-d)} + \C^4  \end{array} \right), 
\end{equation}
where $\A^1$ and $\A^2$ can be any matrix valued measurable functions (in respectively $M_{d\times d}$ and $M_{d\times (n-d)}$),  $I_{(n-d)}\in M_{(n-d)\times (n-d)}$
denotes the identity matrix,
and  \smallskip
\begin{itemize}
\item
$|t|\nabla b$ satisfies the Carleson measure condition with a constant $M$,  \smallskip
\item $\lambda^{-1} \leq b\leq \lambda$ for some constant $\lambda>0$,  \smallskip
\item both $\C^3$ and $\C^4$ satisfies the Carleson measure condition with a constant $M$.
\end{itemize}
Consider the elliptic operator $L_0= \diver |t|^{d+1-n} \A \nabla$. 
Then there exists $k_0>0$, depending on the aperture $a$ of the involved cones only and $C>0$, depending on the ellipticity parameters of $\A$, $\lambda$, the dimensions, $a$, 
and $M$ only, such that
for every weak solution $u$ of $L_0$ and every cube $Q\subset \rd$,
we have
\begin{equation}\label{eqsf2} 
\|S^Q u \|_{L^2(Q)}^2 \leq C \| N^{2Q} u \|_{L^2(k_0Q)}^2.
\end{equation}
where $k_0Q$, $k_0>0$, stands for the
cube with the same center as $Q$ and sidelength $k_0l(Q)$.
\end{theorem}

\ms
\bp For simplicity we will take $a=1$ throughout the argument 
($a$ being the aperture of the access cones); the modifications for a general $a$ are straightforward.

Let $\Phi\in C_0^\infty(\RR)$ be such that $0\leq \Phi \leq 1$, $\Phi=1$ on $B(0,1)$, and 
$\Phi$ is supported in $B(0,2)$. Let $u$ be as in the statement, let a cube $Q \subset \R^d$
be given, denote by $\delta(x)$ the usual Euclidean distance from $x\in \rd$ to $Q$, 
and define 
$$\Psi(x,t):=\Phi\left(\frac{\delta(x)}{|t|}\right)\, \Phi\left(\frac{|t|}{l(Q)}\right)\,  
\Phi\left(\frac{ 2\eps}{|t|}\right)  \ \text{ for } (x,t)\in \rn,$$ 
where $\eps\ll l(Q)$ will eventually tend to $0$. 
 Define 
\begin{eqnarray*}
E_1&:=&\{(x,t)\in \Omega_0:\, x\in 10\,Q, \,|t|\leq\delta(x)\leq 2|t|\}, \\
E_2&:=&\{(x,t)\in \Omega_0:\, x\in 10\,Q, \,l(Q)\leq |t| \leq 2l(Q)\},\\
E_3&:=&\{(x,t)\in \Omega_0:\, x\in (1+8\eps)Q, \, |t| \leq 2\eps \leq 2|t|\},
\end{eqnarray*}
and observe that
\begin{eqnarray}\label{eqsf4}
|\nabla \Psi(x,t)| 
&\lesssim& \frac 1{|t|}\, \1_{E_1}(x,t) 
+\frac{1}{l(Q)}\, \1_{E_2}(x,t)+\frac{1}{\eps}\, \1_{E_3}(x,t)
\nn\\
&\lesssim& \frac 1{|t|}\left(\1_{E_1}(x,t)+\1_{E_2}(x,t)+\1_{E_3}(x,t)\right)
\end{eqnarray}
for $(x,t) \in \Omega_0$. Set 
\begin{equation}\label{eqsf6} 
J = J(\varepsilon) :=\dint_{\Omega_0} \frac{1}{|t|^{n-d-2}}\,|\nabla u|^2 \Psi^2\, dx dt.
\end{equation}
This integral is finite, because $\nabla u \in L^2_{loc}(\Omega_0)$ and we integrate 
over a subset of $\{x\in 2Q, \, \eps<|t|<2l(Q)\}$, which is compact in $\Omega_0$. 
In addition, we claim that
\begin{equation} \label{a8.19}
\|S^Q u\|_{L^2(Q)}^2 \leq \liminf_{\varepsilon \to 0} J(\varepsilon).
\end{equation} 
Indeed, the definition \eqref{a8.9} yields
\[ 
\|S^Q u\|_{L^2(Q)}^2 = \int_{x\in Q}  \dint_{\gamma^Q(x)} 
|\nabla u(Y)|^2 \frac{dY}{|Y-(x,0)|^{n-2}} dx.
\]
Let $Y \in \gamma^Q(x)$ for some $x\in Q$. Write $Y = (y,s)$; then 
\[
Y \in H := \big\{ (y,s) \in \Omega_0 \, ; \, \delta(y) \leq |s| \leq l(Q) \big\}.
\]
Then by Fubini,
\[\begin{split}
\|S^Q u\|_{L^2(Q)}^2 &\leq \dint_{(y,s) \in H} |\nabla u(y,s)|^2
\bigg\{ \int_{x\in \Delta(y,|s|)} \frac{dx}{|(y,s)-(x,0)|^{n-2}} \bigg\} dy ds
\\
& \leq C \dint_{(y,s) \in H} |\nabla u(y,s)|^2 \frac{dy \, ds}{|s|^{n-d-2}}.
\end{split}\]
Then observe that with our definitions,
$\Phi\left(\frac{\delta(y)}{|s|}\right)\, \Phi\left(\frac{|s|}{l(Q)}\right) = 1$
when $(y,s) \in H$, which means that $\Psi(y,s) = 1$ for $\varepsilon$ small enough
(depending on $(y,s)$); \eqref{a8.19} follows.

\ms
We shall prove that for every $\varepsilon>0$,
\begin{equation} \label{a8.20}
J \leq C (AJ)^{1/2} + A,
\end{equation}
where 
\begin{equation} \label{a8.21}
A  = \|N^{2Q} (u)\|_{L^2(10Q)}^2.
\end{equation}
This implies, since $J < +\infty$, that $J \leq C A$, i.e.,
\[
J(\varepsilon) \leq C \|N^{2Q}( u)\|_{L^2(10Q)}^2. 
\]
Then \eqref{eqsf2} will follow from \eqref{a8.19}, 
by taking the limit when $\varepsilon \to 0$.
So it is enough to check \eqref{a8.20}.

To this end, since $\A$ is uniformly elliptic and $b$ is bounded from above, we write
\begin{multline}\label{eqsf5}
J\lesssim \dint \frac{\A\nabla u \cdot\nabla u}{|t|^{n-d-1}}\, \Psi^2\, |t|\,\frac{1}{b} dx dt 
\\[4pt]
=\dint \frac{\A}{|t|^{n-d-1}}\,\nabla u\cdot\nabla \left(u \Psi^2\, |t|\,\frac{1}{b}\right) dx dt
+\dint \frac{\A}{|t|^{n-d-1}}\,\nabla u \cdot \frac{\nabla b}{b^2}  u \Psi^2\, |t|\, dx dt\\[4pt]
-\dint \frac{\A}{|t|^{n-d-1}}\,\nabla u\cdot  \nabla (|t|) \, u \Psi^2\, \,\frac{1}{b}\, dx dt -2 \dint \frac{\A}{|t|^{n-d-1}}\,\nabla u\cdot  \nabla \Psi \, u\Psi\, |t|\,\frac{1}{b}\, dx dt\\[4pt]
=: I_0 + I_1 + I_2 + I_3,
\end{multline}
where we just computed the four pieces of $\nabla \left(u \Psi^2\, |t|\,\frac{1}{b}\right)$ 
to get the main equality.
The first integral $I_0$ is zero, because $u$ is a weak solution and $\Psi$ is smooth and compactly supported in $\Omega_0$ (so that $u \Psi^2\, |t|\,\frac{1}{b}$ is a valid test function in \eqref{weaksol} under our assumptions on $u, \Psi, b$; see \cite[Lemma 8.16]{DFMprelim}).

Now, due to the boundedness of $\A$, the fact that $b\geq \lambda^{-1}>0$, 
and the Cauchy-Schwarz inequality, we have
\[\begin{split} 
|I_1| & \lesssim  \dint |\nabla u| |\nabla b| u \Psi^2 \, |t|^2
\frac{dx \, dt} {|t|^{n-d}} 
\leq \left(\dint  |t|^2 |\nabla b|^2 u^2 \Psi^2\, 
\frac{dx \, dt}{|t|^{n-d}} \right)^\frac12 J^\frac12. 
\end{split}\] 
The following estimate will be used a few times in the rest of the argument.
Thanks to
our first assumption on $|t|\nabla b$, $d\mu = \big||t|\nabla b \big|^2 \frac{dx \, dt}{|t|^{n-d}}$
is a Carleson measure on $\Omega_0$ (see Definition~\ref{defCM8}), and the 
Carleson inequality (see for instance \cite{Stein93}; the proof is only written in codimension $1$
but it goes through) 
says that
\[
\dint |t|^2 |\nabla b|^2 u^2 \Psi^2\, \frac{dx\, dt}{|t|^{n-d}} 
= \dint u^2 \Psi^2\, d\mu 
\leq C \|\mu\|_{CM}\|N(u\Psi)\|_{L^2(\R^n)}^2 ,
\]
where $\|\mu\|_{CM} = \sup_{{\rm balls }\,\Delta\subset \R^d}\frac{1}{|\Delta|}\dint_{T(\Delta)}(|s||\nabla b(y,s)|)^2\,\frac{dyds}{|s|^{n-d}}$ is the Carleson norm 
of $\mu$. Hence
\begin{equation}\label{eqsf7}
|I_1|\lesssim \|N(u\Psi)\|_{L^2(\rd)} J^{1/2}\leq \|N^{2Q}(u)\|_{L^2(10\,Q)} J^{1/2}.
\end{equation}

The integral $I_3$ contains $\nabla \Psi$, which we estimate with \eqref{eqsf4},
the boundedness of $\A$ and $b^{-1}$, and Cauchy-Schwarz inequality. This yields
\begin{eqnarray} \label{eqsf8}
|I_3| = 
\dint \frac{\A}{|t|^{n-d-1}}\,\nabla u\cdot  \nabla \Psi \, u\Psi\, |t|\,\frac{1}{b}\, dx dt
&\lesssim& \dint_{E_1 \cup E_2 \cup E_3} |\nabla u| |u|  |\Psi| \, \frac{dx \, dt}{|t|^{n-d-1}}
\nn\\
&\lesssim& J^{1/2}\left(\dint_{E_1\cup E_2 \cup E_3} 
u^2 \,\frac{dx \, dt}{|t|^{n-d}}\right)^{1/2}.
\end{eqnarray}
We will start with $E_2$:
\begin{equation}\label{eqsf9}
\dint_{E_2} u^2 \,\frac{dx \, dt}{|t|^{n-d}}
\leq \frac{1}{|l(Q)|^{n-d}} \dint_{E_2} \left(N^{2Q}u(x)\right)^2\, dxdt 
= C \int_{10\,Q} \left(N^{2Q}u\right)^2\, dx.
\end{equation} 
Similarly, 
\begin{equation}\label{eqsf11}
\dint_{E_3} u^2 \,\frac{dx \, dt}{|t|^{n-d}}
\leq \frac{1}{\eps^{n-d}} \dint_{E_3} \left(N^{2Q}u(x)\right)^2\, dxdt 
= C \int_{10\,Q} \left(N^{2Q}u\right)^2\, dx.
\end{equation} 
Finally, 
\begin{multline}\label{eqsf10}
\dint_{E_1} u^2 \,\frac{dx \, dt}{|t|^{n-d}} 
\leq \dint_{x\in 10\,Q, \, \delta(x)/2\leq |t|\leq \delta(x)} \left(N^{2Q}u(x)\right)^2\, 
\frac{dx\,dt}{|t|^{n-d}}
\\[4pt]
\leq  \int_{10\,Q} \left(N^{2Q}u\right)^2\, \left(\int_{\delta(x)/2
\leq |t|\leq \delta(x)}\frac{dt}{|t|^{n-d}}\right)\, dx 
\leq C  \int_{10\,Q} \left(N^{2Q}u\right)^2\, dx. 
\end{multline}

At this point we are left with the most delicate term, $I_2$. 
If the coordinates of $t\in \R^{n-d}$ are $(t_{d+1},\dots,t_n)$, 
\begin{multline}\label{eqsf12}
I_2 =-\dint \frac{\A}{|t|^{n-d-1}}\,\nabla u \cdot \nabla (|t|)\, u \Psi^2\, \,\frac{1}{b}\, dx dt
\\[4pt]
=-\dint \sum_{i=d+1}^n\sum_{j=1}^n\frac{\A_{ij}}{|t|^{n-d-1}}\,\partial_j u\, u \Psi^2\, \frac{t_i}{|t|}\,\frac{1}{b}\, dx dt.
\end{multline}
At this point
use the special form of $\A$. Notice that the upper part
of $\A$ does not contribute to the sum, and denote by $I_{21}$ the part that comes from
$bI_{n-d}$, and by $I_{22}$ the remaining part, that comes from $\C^3$ and $\C^4$; thus
\[
I_{21} = -\dint \sum_{i=d+1}^n\frac{1}{|t|^{n-d-1}}\,\partial_i u\, u \Psi^2\, \frac{t_i}{|t|}\, dx dt,
\]
where the two terms with $b$ conveniently cancel, and
\[
I_{22} = -\dint \sum_{i=d+1}^n\sum_{j=1}^n\frac{\C_{ij}}{|t|^{n-d-1}}\,\partial_j u\, u \Psi^2\, \frac{t_i}{|t|}\,\frac{1}{b}\, dx dt,
\]
where the $\C_{ij}$ are the coefficients of $\C := \begin{pmatrix} 0 & 0 \\ \C^3 & \C^4 \end{pmatrix}$ and satisfy the Carleson condition.

The term $I_{22}$ is estimated exactly like 
$I_1$, only using the Carleson condition on the $\C_{ij}$,
in place of the Carleson condition for $|t|\nabla b$.
We are left with $I_{21}$, which we write as 
\begin{multline}\label{eqsf13}
I_{21} =
-\frac 12 \dint \sum_{i=d+1}^n\frac{1}{|t|^{n-d-1}}\,\partial_i (u^2 \Psi^2)\, \frac{t_i}{|t|}\, dx dt\\[4pt]+
\frac 12\dint \sum_{i=d+1}^n\frac{1}{|t|^{n-d-1}}\,u^2 \partial_i(\Psi^2)\, \frac{t_i}{|t|}\, dx dt 
=: I_{211}+I_{212}.
\end{multline}
For the first term, observe that $\sum_{i=d+1}^n \partial_i (u^2 \Psi^2) \frac{t_i}{|t|}
= \partial_r(u^2\Psi^2)$ (the derivative in the radial direction).
We switch to polar coordinates, abusing the notation slightly by writing $u$ and $\Psi$ to mean a composition of the corresponding functions with the mapping of the change of the coordinates,
use Fubini and our assumptions on $u$ and $\Psi$, integrate by parts in polar coordinates, 
and get that
\begin{equation}\label{eqsf14}
I_{211}
=C \int_{\rd} \int_{{\mathbb S}^{n-d-1}} \int_\eps^{\infty} 
\partial_r(u^2\Psi^2)\,drd\omega dx=0.
\end{equation}
For the remaining integral $I_{212}$, observe that by \eqref{eqsf4}, and then 
\eqref{eqsf9}-\eqref{eqsf10},
\begin{eqnarray} \label{a8.31}
|I_{212}| &\leq& C \dint \frac{1}{|t|^{n-d-1}}\,u^2 \Psi |\nabla\Psi| \, dx dt
\leq \dint_{E_1 \cup E_2 \cup E_3} \frac{1}{|t|^{n-d}}\,u^2 
\, dx dt
\nn\\
&\leq& C \int_{10\,Q} \left(N^{2Q}u\right)^2\, dx =  CA
\end{eqnarray}
The main estimate \eqref{a8.20}, and then Theorem \ref{tsf1}, follow.
\ep

\begin{remark}\label{rSfCm} As pointed out after the statement of Theorem~\ref{Itsf1}, 
the Carleson measure estimate \eqref{Ieqsf2a} has a version for characteristic functions
of Borel sets,  which says that if $H$ is a Borel subset of $\Gamma_0 =\R^d$,
$u_H$ is the weak solution defined by \eqref{a2.24}, and $Q\subset \R^d$ is
any cube, we have
\begin{equation}\label{Ieqsf2}
\|S^Q u_H\|_{L^2(Q)}^2 \leq C |Q|,
\end{equation}
which is enough for Section \ref{SAinfty}, and follows at once from Theorem \ref{Itsf1},
again because $|u_H| \leq 1$. 

Here we prove, as promised in the Introduction, that the Carleson measure estimate 
\eqref{a2.26} follows from \eqref{Ieqsf2}.
Let $x\in \Gamma_0$ and $r > 0$ be given, and let $Q$ denote the cube centered on 
$x$ and with sidelength $l(Q) = 2r$. Set
\begin{eqnarray} \label{a2.37}
I &=& \|S^Q u_H\|_{L^2(Q)}^2 = \int_{z\in Q} S^Q u_H(z)^2 dz
\nn\\
&=& \int_{z\in Q} 
\int_{(y,s)\in \gamma^Q (z)} |\nabla u_H(y,s)|^2 \, \frac{dyds}{|(y,s)-(z,0)|^{n-2}} \, dz,
\end{eqnarray}
and observe that  by \eqref{a2.27}, $|(y,s)-(z,0)| \leq (1+a)|s|$ when $(y,s)\in \gamma^Q (z)$.
Thus by Fubini
\begin{eqnarray} \label{a2.38}
I &\geq& (1+a)^{2-n} \int_{z\in Q}\int_{(y,s)\in \gamma^Q (z)} |\nabla u_H(y,s)|^2 |s|^{2-n} dydsdz
\nn\\
&=&  (1+a)^{2-n} \int_{y\in \R^d}\int_{s=0}^{l(Q)}  |\nabla u_H(y,s)|^2 \theta(y,s) |s|^{2-n} dyds,
\end{eqnarray}
where (by \eqref{a2.27} again)
\begin{equation} \label{a2.39}
\theta(y,s) = \big|\big\{z\in Q : \, (y,s) \in \gamma^Q(z) \big\}\big| = \big|Q \cap B(y,as)\big|.
\end{equation}
Notice that $\theta(y,s) \geq C^{-1} (as)^d$ for $y\in Q$, so 
\begin{equation} \label{a2.40}
I \geq C^{-1}(1+a)^{2-n} a^d \int_{y\in Q}\int_{s=0}^{l(Q)}  |\nabla u_H(y,s)|^2 |s|^{d+2-n} dyds.
\end{equation}
 For \eqref{a2.26} we need to estimate
\begin{equation} \label{a2.41}
 \int_{(y,s) \in \Omega_0 \cap B(x,r)} \big(|s|\nabla u_H(y,s)\big)^2\,\frac{dyds}{|s|^{n-d}} 
 \leq \int_{y\in Q} \int_{s=0}^{l(Q)} |\nabla u_H(y,s)|^2 \,\frac{dyds}{|s|^{n-d+2}} 
\leq C_a I 
\end{equation}
by \eqref{a2.40}. But $I \leq C |Q| \leq C r^d$ by \eqref{Ieqsf2}, so \eqref{a2.26} really follows 
from \eqref{Ieqsf2}.
\end{remark}

\section{From Carleson measure estimates for solutions to $A^\infty$ property of harmonic measure}
\label{SAinfty}

Throughout this section, like the previous one, $\Gamma_0 = \R^d$ and 
$\Omega_0 = \R^n \setminus \Gamma_0$. 
We consider an operator $L_0 = - \diver |t|^{d+1-n} \A \nabla$, where $\A$ is a matrix-valued 
function defined on $\Omega_0$ satisfying the ellipticity and boundedness conditions \eqref{ABounded}--\eqref{AElliptic}.

For $X\in \Omega_0$, $\omega^X = \omega^X_{\Omega_0,L_0}$ is the harmonic measure 
defined near \eqref{a2.6}.
The existence of the harmonic measure, the fact that the harmonic measure is a probability measure, and the fact that the function $u_H$ defined as $u_H(X) = \omega^X(H)$, with $H \subset \R^d$ a Borel set, is a weak solution to $L_0 u_H = 0$ can be found in \cite[Section 9]{DFMprelim}; see
Lemmata 9.23 and 9.30 in particular.
We need the Harnack inequality and H\"older continuity of $u_H$ at the boundary (see \cite[Section 8]{DFMprelim}, Lemmas 8.42 and 8.106). 

\begin{lemma}[Harnack] \label{HarnackI}
Let $H\subset \R^d$ be a Borel set. Let the function $u_H$ defined as above by 
$u_H(X) = \omega^X(H)$, $X\in \Omega_0$. 
Let $B$ be a ball such that $2B\subset \Omega$, then 
\begin{equation} \label{Harnack1}
\sup_B u_H \leq C \inf_B u_H,
\end{equation}
where $C>0$ depends only on the dimensions $d$, $n$ and the ellipticity constants.
\end{lemma}

\begin{lemma}[H\"older at the boundary] \label{HolderB}
Let $H\subset \R^d$ be a Borel set. If the ball $\Delta:= \Delta_r(x) \subset \R^d$ doesn't intersect $H$, then for any $s<r$
\[\sup_{B(x,0;s)} u_H \leq C\left(\frac sr\right)^\alpha,\]
where $C,\alpha$ are two positive constants that depend only on the dimensions $d$ and $n$ and the ellipticity constants of $L_0$.
\end{lemma}

Given some $\Delta=\D_r(x)$, $x\in \rd$, $r>0$, as above, a point $A_{\Delta_r(x)}=(x, t)$ in 
$\Omega_0$ such that $|t|=r$ is referred to as a cork-screw point of $\Delta_r(x)$. 
Here we use the special shape of $\Gamma_0$ to get a cork-screw constant $1$ 
(i.e. we can choose $\tau = 1$ in \eqref{Corkscrew}), 
but this does not matter; $A_{\Delta_r(x)}$ is clearly not uniquely defined for $n-d>1$ 
and whenever we write $A_{\Delta_r(x)}$ we mean that any such point is suitable. 
The following three properties of the harmonic measure, whose proof can be found in 
\cite[Section 11]{DFMprelim}, will also be used repeatedly throughout the section.

\begin{lemma}[Nondegeneracy]\label{lnondeg} For any $x\in\rd$ and $r>0$, 
$$\omega^{A_{\D_r(x)}}(\D_{r}(x)) \geq C,$$
where $C>0$ depends on $n$, $d$, and the ellipticity constants of $L_0$ only.
\end{lemma}

\begin{lemma}[Doubling]\label{ldoubling} For any $x\in\rd$, $r>0$, and  
$Y\in \Omega_0\setminus B_{4r}(x,0)$, 
$$\omega^Y(\D_{2r}(x)) \leq C\,  \omega^Y(\D_{r}(x)),$$
where $C>0$ depends on $n$, $d$, and the ellipticity constants only.
In particular, if
$\Delta \subset \R^d$ is a ball satisfying $2\Delta \subset \Delta_r(x)$, then
$$\omega^{A_{\D_r(x)}}(2\Delta) \leq C\,  \omega^{A_{\D_r(x)}}(\Delta).$$
\end{lemma}

\begin{lemma}[Change of Pole]\label{lChOfPole} For any $x\in\rd$, $r>0$, 
any $Y\in \Omega_0\setminus B_{2r}(x,0)$, and any Borel set $E\subset \Delta:=\D_r(x)$, 
$$\frac{\omega^Y(E)}{\omega^Y(\Delta)}\approx \omega^{A_\Delta}(E), $$
where the implicit constants depend on $n$, $d$, and the ellipticity constants only. 
\end{lemma}

We recall what we mean by absolute continuity (or $A^\infty$ property) in our context.

\begin{definition}\label{dhmai} 
We say that the harmonic measure is $A^\infty$ 
(with respect to the Lebesgue measure) on $\rd$ if for 
every $\eps>0$, there exists $\delta>0$ such that for every ball $\Delta\subset \rd$,
every ball $\Delta'\subset \Delta$ and every Borel set $E\subset \Delta'$, 
\begin{equation}\label{eqhmai} 
\mbox{if }\, \frac{\omega^{A_{\Delta}}(E)}{\omega^{A_{\Delta}}(\Delta')} <\delta \mbox{ then } \frac{|E|}{|\Delta'|}<\eps.
\end{equation}
\end{definition}

Here is the main result of this section. 

\begin{theorem}\label{tai1} 
Let $L_0 = - \diver |t|^{d+1-n} \A \nabla$, where the real matrix-valued function $\A$
satisfies the ellipticity and boundedness conditions \eqref{ABounded}--\eqref{AElliptic}.
Assume that we can find $K \geq 0$ such that
for any Borel set $H\subset \rd$, the solution $u$ defined by $u(X) = \omega^X(H)$, 
$X\in \Omega_0$, satisfies a Carleson measure estimate
\begin{equation}\label{eqai2} 
\sup_{{\rm balls }\,\Delta\subset \rd}\frac{1}{|\Delta|}\dint_{T(\Delta)}(|t||\nabla u|)^2\,\frac{dydt}{|t|^{n-d}}\leq K.
\end{equation}
Then the harmonic measure is $A^\infty$ with respect to the Lebesgue measure on $\rd$ in the sense of Definition~\ref{dhmai}.
\end{theorem}

With our usual convention, \eqref{eqai2} means that $|t||\nabla u|$ satisfies the
Carleson measure condition, with constant at most $K$.
Before we prove 
the theorem, we combine Theorem~\ref{tai1} with Theorem~\ref{tsf1}.
 
\begin{corollary}\label{cai5} Let $\A$ be an elliptic matrix satisfying \eqref{ABounded}, \eqref{AElliptic} in $\Omega_0=\rn\setminus \Gamma_0$ with $\Gamma_0=\rd$.  
Assume that $\A$ has the following structure:
\begin{equation}\label{eqsf1B}
\A=\left( \begin{array}{cc}
 \A^1 & \A^2 \\ \C^3 & bI_{(n-d)} + \C^4  \end{array} \right), 
\end{equation}
where $\A^1$ and $\A^2$ can be any matrix valued measurable functions (in respectively $M_{d\times d}$ and $M_{d\times (n-d)}$),  $I_{(n-d)}\in M_{(n-d)\times (n-d)}$
denotes the identity matrix,
and  \smallskip
\begin{itemize}
\item
$|t|\nabla b$ satisfies the Carleson measure condition with a constant $M$,  \smallskip
\item $\lambda^{-1} \leq b\leq \lambda$ for some constant $\lambda>0$,  \smallskip
\item both $\C^3$ and $\C^4$ satisfies the Carleson measure condition with a constant $M$.
\end{itemize}
Then the harmonic measure associated with $L_0 = - \div |t|^{d+1-n}\A \nabla$ 
is $A^\infty$ with respect to the Lebesgue measure on $\rd$ in the sense of Definition~\ref{dhmai}, with the implicit constants depending on $n$, $d$, ellipticity constants of $\A$, $M$, and $\lambda$.
\end{corollary}

\bp 
We have already given the proof in the Introduction. The Corollary follows from 
Theorems~\ref{tai1} and \ref{tsf1} combined with Remark~\ref{rSfCm} -- see the discussion following Theorem~\ref{Itsf1}.
\ep

\noindent {\it Proof of  Theorem~\ref{tai1}}. 
Now we pass to the proof of Theorem~\ref{tai1}. Much of the argument follows the lines of \cite{KKiPT}, and even more so \cite{DPP2015}, but we will aim for a self-contained exposition. 

\vskip 0.08 in \noindent {\bf Step I: $\eps_0$-good cover and construction of functions with large oscillations on small sets.} 
Let $R>0$ and $x\in \rd$ be given, and consider the ball $\Delta = \Delta_R(x)$.
Just as in \cite{DPP2015}, we start by observing that $\omega^{A_{\Delta}}$ is positive 
and doubling near $\Delta$, by Lemma~\ref{ldoubling} and Lemma~\ref{HarnackI}.
Here the geometry is fairly simple, and it is not hard to check that 
\begin{equation} \label{a9.13}
\omega^{A_{\Delta}}(\Delta(y,t)) \leq C \omega^{A_{\Delta}}(\Delta(y,t) \cap \Delta)
\ \text{ for $y\in \Delta$ and } 0 < t < R. 
\end{equation}
From this it is also easy to deduce that $\Delta$, with the Euclidean distance and 
the (restriction to $\Delta$ of the) 
measure $\omega = \omega^{A_{\Delta}}$, is a space of homogeneous type.

This is pleasant, because we can use \cite{Ch} directly, to construct a
dyadic system of pseudo-cubes on $\Delta$ associated to $\omega$
and satisfying the following properties. Otherwise, we could always have followed
the construction near $\Delta$, and replaced $\Delta$ with a finite union of initial cubes
to start the argument.

There exist constants $0<c<1$ and $M>1$, that depend only 
the doubling constant of $\omega$, 
and then a collection $\mathbb{D} = \cup_{k \geq k_0} \mathbb{D}_k$
of Borel subsets of $\Delta$, with the following properties. For each integer $k \geq k_0$,
we write 
$$
\mathbb{D}_k:=\{Q_{j}^k\subset \Delta: j\in \mathfrak{I}_k\},
$$ 
where $\mathfrak{I}_k$ denotes some index set depending on $k$, but some times
we will forget about the indices and just write $Q \in \mathbb{D}_k$
for any of the $Q_{j}^k$, and refer to $Q$ as a pseudo-cube of generation $k$.
These pseudo-cubes have properties that are very similar to the properties of the usual dyadic
cubes of $\R^d$, as follows:

\begin{list}{$(\theenumi)$}{\usecounter{enumi}\leftmargin=.8cm
\labelwidth=.8cm\itemsep=0.2cm\topsep=.1cm
\renewcommand{\theenumi}{\roman{enumi}}}

\item $\D=\cup_{j} Q_{j}^k \,\,$ for any $k \geq k_0$.

\item If $m > k$ then either $Q_{i}^{m}\subseteq Q_{j}^{k}$ or
$Q_{i}^{m}\cap  Q_{j}^{k}=\emptyset$.

\item $Q_i^m \cap Q_j^m=\emptyset$ if $i\neq j$. 

\item Each pseudo-cube $Q\in\mathbb{D}_k$ has a ``center'' $x_Q\in \D$ such that
\begin{equation}\label{cube-ball}
\Delta(x_Q,2^{-k})\subset Q \subset \Delta(x_Q,M 2^{-k}).
\end{equation}

\item If $Q_i^m \subsetneq Q_j^k$, then $\omega(Q_i^m)<c\, \omega (Q_j^k)$.

\item $\omega(\partial Q_i^m)=0$ for all $i, m$.

\item
$\mathbb{D}_{k_0}$ is composed of a single pseudo-cube, which we often call $Q_0$,
but is equal to $\Delta$. 
\end{list}

Let us make a few comments about these cubes. 
We decided to use a dyadic scaling because it is convenient, but then we do not say
that if $Q \in \mathbb{D}_{k+\ell}$ and $R$ is the cube of $\mathbb{D}_k$ that contains
$Q$ (it is unique by ($ii$), and it is called an ancestor of $Q$) is necessarily strictly
larger (as a set) than $Q$. 

We also decided to use Borel sets so that the pseudo-cubes of a same generation 
are disjoint; another option would have been to take closed pseudo-cubes that are almost
disjoint (by ($vi$)).

The condition ($vi$) is a slightly weaker version of a condition that says that small
neighborhoods of $\d Q_i^m$ have a small $\omega$-measure. This is the same
``small boundary condition'' that gives the existence of a center $x_Q$.

Condition ($v$) is usually not stated, but it follows from the doubling condition,
and the fact that since $Q_i^m \subsetneq Q_j^k$ it has a sibling
(another pseudo-cube of generation $m$, which is contained in the parent of $Q_i^m$)
which is therefore contained in $Q_j^k \sm Q_i^m$, and has a comparable $\omega$-measure
by \eqref{cube-ball}.

Because of \eqref{cube-ball}, we know that $2^{-k_0} \approx R$.

In the setting of a general space of homogeneous type, 
this decomposition was obtained
by Christ \cite{Ch}, with the dyadic parameter $1/2$ replaced by some constant $\delta \in (0,1)$
(which allows him to take different cubes at each generation).
In fact, one may always take $\delta = 1/2$ (cf.  \cite[Proof of Proposition 2.12]{HMMM}).
In the presence of the Ahlfors regularity 
property, the result already appears in \cite{DS1,DS2}. 
Some predecessors of this construction have appeared in \cite{David88} and \cite{David91}.

We can use the pseudo-cubes to do a Whitney decomposition of any open set $\O\subset\D$;
we get that
\begin{equation} \label{a9.15}
\O = \bigcup_{i,m} Q_i^m,
\end{equation}
where the $Q_i^m$ are pairwise disjoint and for each $Q_i^m$ in this decomposition, 
$\dist\{Q_i^m, \D\setminus \O\} \approx \diam\, Q_i^m \approx 2^{-m}$.
Simply take the largest (i.e., of the earliest generations) pseudo-cubes $Q_i^m$
such that $\dist(Q_i^m, \Delta \sm Q_i^m) \geq C 2^{-m}$ (for some large constant $C$),
and proceed as usual.

We denote by $k(Q)$ the generation of the pseudo-cube $Q$ 
(i.e., the integer $k$ such that $Q \in \mathbb{D}_k$),
and set $\ell(Q) = 2^{-k(Q)}$; thus $\ell(Q)\approx \diam(Q)$ by \eqref{cube-ball}.
By analogy, we call $\ell(Q)$ the length of $Q$.

\begin{definition}\label{deps-cover} 
Fix any small $\eps_0>0$ and any Borel set $E\subset \Delta_r\subseteq \Delta$. 
We say that a collection of nested open sets $\{\O_i\}_{i=1}^k$ is a 
good $\eps_0$-cover for $E$ of length $k\in \NN$ if 
$$
E\subseteq \O_{k}\subseteq \O_{k-1}\subseteq...\subseteq \O_0=\Delta_r, 
$$
and for every $l=1,...,k$
\begin{equation}\label{eq11.12-bis} 
\O_l=\bigcup_{i \in I(l)} S_i^l 
\end{equation}
where the $S_i^{l}$, $i\in I(l)$, are disjoint 
elements of $\dd$
and for all $1\leq l\leq k$ and $i \in I(l-1)$, 
\begin{equation}\label{eq11.12} 
\omega 
(\O_l\cap S_i^{l-1})\leq \eps_0\, \omega
(S_i^{l-1}).
\end{equation}
\end{definition}

Observe that we changed the notation for our  pseudo-cubes from $Q$ to $S$ in order 
to not confuse the numerology: $Q_i^m$ is always an 
element of $\dd_m$ and in particular, has length $2^{-m}$, 
while $S_i^l$ is an element of a decomposition of an open set $\O_l$ into pseudo-cubes.

Notice also that if $j\in I(l)$ and $i\in I(l-1)$ are such that $S_j^{l} \cap S_i^{l-1} \neq \emptyset$, then the the property ($ii$) above gives that $S_j^l \subset S_i^{l-1}$; as a consequence, \eqref{eq11.12} forces $S_j^{l}$ to be a pseudo-cube of higher generation than $S_i^{l-1}$. In particular, by \eqref{cube-ball}
\begin{equation} \label{a9.19}
\ell(S_j^{m}) \leq 2^{l-m} \ell(S_i^{l})
\end{equation}
if $m>l$ and $S_j^{m} \cap S_i^{l} \neq \emptyset$.

When $k$ is large, we expect $\omega(E)$ to be quite small.
The next proposition, which we take from  \cite{KKiPT, DPP2015}, says that
some converse is true too.

\begin{proposition}\label{p11.8} \cite{KKiPT, DPP2015} 
For every $\eps_0>0$ sufficiently small there exists $\delta_0>0$ 
such that if 
$E\subset\Delta_r \subseteq \Delta$ and 
\begin{equation}\label{eq11.14}
\frac{\omega
(E)}{\omega
(\Delta_r)} \leq \delta
\end{equation}
for some $\delta \leq \delta_0$,
then $E$ has a good $\eps_0$ cover of length $k \geq C^{-1} 
\frac{\log \delta}{\log \eps_0}$.
\end{proposition}

Here and below, unless specified otherwise, ``for any $\eps_0$ sufficiently small" is to be interpreted in the sense that there exists a numerical constant (that may also depend on the doubling constant for $\omega$) 
such that for all $\eps_0$ smaller than this constant a certain property holds. The same applies to certain parameters being ``sufficiently large".  

Recall that the goal is to prove \eqref{eqhmai}. To this end, we fix some $\eps>0$ and take $E\subset \Delta'=\Delta_r\subset \Delta$ satisfying \eqref{eq11.14}, 
with $\delta$ to be defined later. Consider a good $\eps_0$-cover of $E$ 
($\eps_0$ to be determined below as well) relative to $\D_r$ and the doubling measure 
$\omega = \omega^{A_\D}$, as in Definition~\ref{deps-cover}. 
For each $S_i^l$ in the dyadic decomposition \eqref{eq11.12-bis} of $\O^l$, 
we denote by $\D_i^l=\D(x_i^l, r_i^l)$ the Euclidean ball on $\rd$ guaranteed by \eqref{cube-ball}, i.e., such that $\D_i^l\subset S_i^l\subset M\D_i^l$. 
Thus $r_i^l = \ell(S_i^l)$. 
Then select in $\D_i^l$ two roughly equal parts
\begin{equation} \label{a9.21}
\widehat \D_i^l:=\D(x_i^l, r_i^l/2) \ \text{ and } 
\widetilde \Delta^l = \Delta(y^l_i,r^l_i/10),
\end{equation}
for some $y^l_i \in \Delta^l$ that we chose so that
$\Delta(y^l_i,r^l_i/5) \subset \Delta^l_i \sm \widehat \Delta^l_i$.
Following a similar construction in \cite{DPP2015}, set
$\widehat\O_l := \bigcup_i \widehat \D_i^l \subset \O_l$ for $0\leq l\leq k$,
then, for $0\leq l\leq k$, $l$ even, take
\begin{equation}\label{eq11.15}
f_l(y)= \1_{\widehat\O_l}
\ \text{ and } f_{l+1}= -f_l\, \1_{\O_{l+1}} = - \1_{\widehat\O_l \cap \O_{l+1}}.
\end{equation}
Thus $f_l + f_{l+1} = \1_{\widehat\O_l \sm \O_{l+1}} \leq \1_{\O_l \sm \O_{l+1}}$.
Finally we set 
\begin{equation}\label{eq11.17}
f= \sum_{l=0}^k f_l.
\end{equation}
One can observe that $f$ is a characteristic function of a Borel set,
because the $\widehat\O_l \sm \O_{l+1}$ are disjoint, and also disjoint
from $\widehat \O_k$ if $k$ is even.
Let $u$ be the weak solution of the Dirichlet problem associated to $f$, 
as in \eqref{a2.6} (but for $\Omega_0$ and $L_0$), i.e., set
$u(X) = \int_{\Gamma_0} f(x) d\omega_{\Omega_0,L_0}
^X(x)$.
The point of the upcoming discussion is to show that $u$ exhibits large oscillations in 
Whitney cubes (cubes in $\Omega_0$ of side-length roughly comparable to their distance 
to $\Gamma_0$)  
which yields a large square function on $E$, and hence,  
an estimate from below on the quantity under the supremum in \eqref{eqai2} 
by a large multiples of 
$|E|/|\Delta_r|$. 
Then, invoking \eqref{eqai2}, we will arrive at \eqref{eqhmai}.

\vskip 0.08 in \noindent {\bf Step II: the solution with data $f$ exhibits large oscillations 
on Whitney cubes.}
Take any $x\in E$ and for $0 \leq l \leq k$, pick $i\in I(l)$ such that 
$x\in S_i^l$. Then simply set $S^l= S_i^l$.
Write $\D^l = \D(x^l,r^l)$ for the Euclidean ball given by \eqref{cube-ball}.  
Denote by $\widehat A^l$ a corkscrew point of $\widehat \D^l$, 
which we take such that 
$\widehat A^l=(x^l, t^l)$ with $|t^l|=r^l/2$. We have 
$\omega^{\widehat A^l}(\widehat \D^l)\gtrsim 1$
by Lemma~\ref{lnondeg}. To fix the notation, we specify that there exists $\eta_1\in (0,1)$ depending on $n$, $d$, and the ellipticity constants only, such that 
\begin{equation}\label{eq11.19-bis}
\omega^{\widehat A^l}(\widehat \D^l)\geq \eta_1 
\end{equation}
Also, 
\begin{multline}\label{eq11.20} 
u(\widehat A^l) := \int_\D f  \,d\omega^{\widehat A^l}\geq \int_\D (f_l+f_{l+1})\, d\omega^{\widehat A^l} \geq \int_{\D^l} (f_l+f_{l+1})\, d\omega^{\widehat A^l}\\[4pt]
= \omega^{\widehat A^l}(\widehat \D^l)-\omega^{\widehat A^l}(\widehat \D^l\cap\O_{l+1}), 
\end{multline}
where each integral is to be interpreted simply as a harmonic measure of the corresponding Borel set (as all integrands are characteristic functions of Borel sets), and both inequalities are due to the fact that the left-hand side is a harmonic measure of a bigger set than the right-hand side. 
Now, by Lemma~\ref{lChOfPole} and then Lemma~\ref{ldoubling} 
\begin{equation}\label{eq11.21}
\omega^{\widehat A^l}(\widehat \D^l\cap\O_{l+1})
\lesssim \frac{\omega^{A_\D}(\widehat \D^l\cap\O_{l+1})}{\omega^{A_\D}(\widehat \D^l)}
\lesssim \frac{\omega^{A_\D}(S^l\cap\O_{l+1})}{\omega^{A_\D}(M\Delta^l)} \leq  
\frac{\omega^{A_\D}(S^l\cap\O_{l+1})}{\omega^{A_\D}(S^l)} \leq\eps_0, 
\end{equation}
where in the last inequality we used \eqref{eq11.12}. 
Hence, $\omega^{\widehat A^l}(\widehat \D^l\cap\O_{l+1})<C\eps_0$ for some $C>0$ depending 
on $n$, $d$, and ellipticity constants only. 
Combing this with \eqref{eq11.20} and \eqref{eq11.19-bis}, we conclude that for 
\begin{equation}\label{eq11.22}
\eps_0<\eta_1/(2C), 
\end{equation}
we have $u(\widehat A^l)\geq \eta_1/2$. 
Moreover,
by the interior H\"older continuity of solutions (Lemma 8.40 in \cite{DFMprelim}), 
there exists $C_1>0$, depending on $n$, $d$, and ellipticity constants only, such that 
\begin{equation}\label{eq11.23}
u(x,t)\geq \eta_1/4  
\ \text{ for } (x,t) \in B(\widehat A^l, C_1^{-1} r^l).
\end{equation}
This shows that $u$ is large in a part of the Whitney box associated with $S^l$. 

Let us now show that there is another part of a (somewhat fattened) Whitney box on which 
$u$ is small. Recall the definition of $\widetilde \Delta^l = \widetilde \Delta^l_i = B(y^l_i,r^l_i/10)$
in \eqref{a9.21}, set $g=\chi_{\rd\setminus [\D^l\setminus\widehat\D^l]}$, and 
denote by $v$ the solution with data $g$, i.e., set $g(X) = \int g(x) d\omega^X(x)
=\omega^X(\rd\setminus [\D^l\setminus\widehat\D^l])$ for $X \in \Omega_0$.
We claim that 
\begin{equation} \label{a9.29}
g+\1_{\O_{l+1}\cap \D^l}\geq f.
\end{equation}
If $x \in \rd\setminus [\D^l\setminus\widehat\D^l]$, then $g(x) \geq 1 \geq f(x)$.
So we may assume that $x\in \D^l\setminus\widehat\D^l$, and also 
in $\D^l\setminus \O_{l+1}$, because otherwise $\1_{\O_{l+1}\cap \D^l}(x) \geq 1 \geq f(x)$.
But then $x$ does not lie on $\widehat\O_l$ (because $\widehat\O_l \cap \D^l = \widehat \D^l$
(since the $S^l_i$ are disjoint), so $f_l(x) \leq 0$; we also have $f_m(x) \leq 0$ 
for $m>l$ because $\widehat \O_m \subset \O_{l+1}$, and the previous $f_m$ do not 
contribute either because they are supported away from $\O_l$. The claim follows.

As before, the inequality in \eqref{a9.29} implies the corresponding inequality for the 
Dirichlet solutions, because each $\omega^X$ is a probability measure.

Since the boundary data of $v$ is zero on $2 \widetilde \Delta^l$, 
Lemma~\ref{HolderB} yields
\begin{equation}\label{eq11.24-bisbis} v(y,s)\leq C \rho^\alpha \ \mbox{ for } 
y\in \widetilde \Delta^l \text{ and } |s|\leq \rho\, r^l,\, 
\end{equation}
where $C$ and $\alpha$ depend only on the ellipticity constants and dimension, 
and the small $\rho < 1$ will be chosen soon.

It remains to control the solution with data $\1_{\O_{l+1}\cap \D^l}$. 
However, by the same argument as  for \eqref{eq11.21}, 
$\omega^{A_{\wt \Delta^l}}(\O_{l+1}\cap \D^l)\leq C\,\eps_0$,
and hence, for any $(y,s)$ such that $y\in \widetilde \Delta^l$ and $|s|=\rho\, r^l$, 
$$\omega^{(y,s)}(\O_{l+1}\cap \D^l)\leq C_{\rho} \,\eps_0,$$
by Lemma \ref{HarnackI}. 
Here $C_\rho$ depends on the length of the Harnack chain from $A_{\wt \Delta^l}$
to $(y,s)$, which, in turn, depends on $\rho$. Combining \eqref{eq11.24-bisbis} with the inequality above and applying the Harnack inequality again, we deduce that there is $C_2>0$, that depends on $\rho$, 
the ellipticity constants, and the dimension only, such that 
\begin{equation}\label{eq11.24}
u(y,s)\leq 2(C\,\rho^\alpha+C_\rho \eps_0), \quad\mbox{whenever } 
|y -y^l|
< \frac{r^l}{C_2} \mbox{ and } \left||s|-\rho\,r^l\right|< \frac{r^l}{C_2},
\end{equation}
where $y^l = y^l_i$
is the center of $\widetilde \Delta^l$. 

Now we make the following choices. Having fixed $\eta_1>0$ as above, 
depending on the ellipticity constants, $n$, and $d$ only, we choose $\rho$ so that 
$2\,C\rho^\alpha< \eta_1/16$  
and then $\eps_0$ so that \begin{equation}\label{eq11.25}
2 C_\rho \eps_0< \eta_1/16 
\end{equation} 
and so that \eqref{eq11.22} is satisfied. 

We conclude that there exist $\eta_1, \rho, C_1, C_2$ and $\eps_0$ in $(0, 1)$, all 
depending on the ellipticity constants, $n$, and $d$ only,
so that \eqref{eq11.23} holds and simultaneously 
\begin{equation}\label{eq11.24-bis}
u(y,s)\leq \eta_1/8 
\quad\mbox{whenever } |y-y^l | 
< \frac{r^l}{C_2} \mbox{ and } \left||s|-\rho\,r^l\right|<  \frac{r^l}{C_2}.
\end{equation}
In other words, $u$ oscillates by at least $\eta_1/8$ within any ball that contains 
the two sets of
\eqref{eq11.23} and  
\eqref{eq11.24-bis}.

\vskip 0.08 in \noindent {\bf Step III: large
oscillations near the $S^l$
yield a large square function.} First of all, we take the aperture of cones $a$ in the definition of the square function large enough (depending, in particular, on $\rho, C_1, C_2$) so that for all $x\in S^l$ the sets in \eqref{eq11.24-bis} and \eqref{eq11.23} are contained in $\gamma^{r}(x)$ (see the definition \eqref{defcone}). The solutions are locally in Sobolev spaces 
and hence are absolutely continuous on lines, and so 
the integral of $\nabla u$ between points of the sets of \eqref{eq11.23} and \eqref{eq11.24-bis}
is not too small. We average and use the Cauchy-Schwarz inequality to get that
\begin{equation}\label{eq11.29}
\left(\eta_1/8\right)^2 \lesssim \iint_{(y,s)\in \gamma^{r}(x): (\rho-C_2^{-1}) r^l\leq |s|\leq (1/2+C_1^{-1}) r^l} |\nabla u|^2 \,\frac {dyds}{|(y,s)-(x,0)|^{n-2}} 
\end{equation}
for any $x\in S^l$. 

Choose a `real' cube $Q_r \supset \Delta_r$ with sidelength $l(Q_r) \approx r$. 
Let $x\in E$ be given. For each $l$, $x$ lies in $\O_l$, hence by \eqref{eq11.12-bis}
we can find $S^l = S^l_i$ that contains $x$. We want to sum the inequalities \eqref{eq11.29}
over $l$, and even though the sets $Z_l(x) = \big\{ (y,s)\in \gamma^{r}(x) \, ; \,  
(\rho-C_2^{-1}) r^l\leq |s|\leq (1/2+C_1^{-1}) r^l \big\}$ where we integrate are not disjoint, 
thanks to  \eqref{a9.19},
their overlap is bounded with a bound that depends on $C_1$, $C_2$, $\rho$, 
$n$, and $d$ only. Thus
\begin{equation}\label{eq11.30}
k\left(\eta_1/8\right)^2 \lesssim S^{Q_r}(x)^2  \ \mbox{ for every } x\in E.
\end{equation}

Recall from Proposition~\ref{p11.8} that $k$, the length of the good $\eps_0$-cover, 
is at least $C^{-1} \frac{\log \delta}{\log \eps_0}$, at least if $\varepsilon_0$ and $\delta$
were chosen small enough. Hence
\begin{equation}\label{eq11.31}
\frac{\log \delta}{\log \eps_0} \lesssim S^{Q_r}(x)^2 
\end{equation}
for $x\in E$ and $Q_r$ as above, with implicit constants that depend on the ellipticity 
constants and the dimension only.

\vskip 0.08 in \noindent {\bf Step IV: conclusion.}
Let $Q_r \supset \Delta_r \supset E $ be the cube chosen in Step III. It follows from \eqref{eq11.31} that 
\begin{equation}\label{eq11.32}
\frac{\log \delta}{\log \eps_0} |E| \lesssim \int_E S^{Q_r(x)}(x)^2\, dx 
\lesssim \int_{Q_r} S^{Q_r}(x) ^2\, dx 
\leq CK |Q_r| \leq C|\Delta_r|, 
\end{equation}
by \eqref{eqai2}. Therefore, 
$$\frac{|E|}{|\Delta_r|}\leq C\,\frac{\log \eps_0}{\log \delta}, $$
for some constant $C$ that depends 
on the ellipticity constants and the dimension only.
Recalling that we we started from $E$ such that $\omega(E) \leq \delta \omega(\Delta_r)$
as in \eqref{eq11.14}, and we want to prove that $|E| < \varepsilon |\Delta_r|$ as in \eqref{eqhmai}.
Thus it only remains to choose $\delta$, small enough, and such that 
$C \,\frac{\log \eps_0}{\log \delta}<\eps$. \ep

\end{document}